\RequirePackage[l2tabu, orthodox]{nag}
\documentclass[oneside,english]{article}
\usepackage{dsfont} 
\usepackage{mathptmx}    %
\usepackage[T1]{fontenc}
\usepackage{babel}
\usepackage{amsmath}
\allowdisplaybreaks
\usepackage{amssymb}
\usepackage{amsthm}
\usepackage[a4paper]{geometry}
\usepackage{graphicx}
\usepackage{microtype}
\usepackage{siunitx}
\usepackage{booktabs}
\usepackage[unicode=true,pdfusetitle,
bookmarks=true,bookmarksnumbered=false,bookmarksopen=false,
breaklinks=false,pdfborder={0 0 1},backref=false,colorlinks=false]{hyperref}
\usepackage{fixltx2e}
\usepackage{breakurl}
\usepackage[disable]{todonotes}
\usepackage{verbatim}
\usepackage{enumerate}
\usepackage{soul}
\usepackage{mathtools}
\mathtoolsset{showonlyrefs=true}
\usepackage{commath} 
\usepackage{ulem}
\usepackage{thm-restate}

\usepackage{algorithm}
\usepackage{algpseudocode}
\newenvironment{varalgorithm}[1]
{\algorithm\renewcommand{\thealgorithm}{#1}}
{\endalgorithm}

\usepackage[square,numbers,sort]{natbib}
\newcommand{\EE}{\mathbb{E}}
\newcommand{\VV}{\mathbb{V}}
\newcommand{\R}{\mathbb{R}}
\newcommand{\N}{\mathbb{N}}
\newcommand{\Z}{\mathbb{Z}}
\newcommand{\Prob}{\mathbb{P}}

\newcommand{\Law}[1]{\operatorname{Law}(#1)}

\newcommand{\E}{\mathbb{E}}

\newtheorem{theorem}{Theorem}

\newtheorem{corollary}{Corollary}

\newtheorem{definition}{Definition}

\newtheorem{remark}{Remark}
\newtheorem{example}{Example}

\newtheorem{lemma}{Lemma}

\newcommand{\mY}{\mathcal{Y}}
\newcommand{\ma}{a}
\newcommand{\vb}{b}
\newcommand{\vB}{B}

\DeclarePairedDelimiter\ceil{\lceil}{\rceil}
\DeclarePairedDelimiter\floor{\lfloor}{\rfloor}

\DeclareMathOperator{\argmin}{arg\min}

\begin{document}
\title{Markov Chain Approximations to Stochastic Differential Equations by Recombination on Lattice Trees}

\author{Francesco Cosentino\thanks{\scriptsize{Mathematical Institute, University of Oxford  \& The Alan Turing Institute, \url{name.surname@maths.ox.ac.uk}}}
\and Harald Oberhauser\thanks{{\scriptsize{Mathematical Institute, University of Oxford,  \url{oberhauser@maths.ox.ac.uk}}}}
\and Alessandro Abate\thanks{\scriptsize{Department of Computer Science, University of Oxford  \& The Alan Turing Institute, \url{name.surname@cs.ox.ac.uk}}}
}

\date{}
\maketitle
 \begin{abstract}
   We revisit the classical problem of approximating a stochastic differential equation by a discrete-time and discrete-space Markov chain.
   Our construction iterates Carath\'eodory's theorem over time to match the moments of the increments locally.
   This allows to construct a Markov chain with a sparse transition matrix where the number of attainable states grows at most polynomially as time increases.
   Moreover, the MC evolves on a tree whose nodes lie on a ``universal lattice'' in the sense that an arbitrary number of different SDEs can be approximated on the same tree.
   The construction is not tailored to specific models, we discuss both the case of uni-variate and multi-variate case SDEs, and provide an implementation and numerical experiments. 
 \end{abstract}

\section{Introduction}
\normalem
Consider a stochastic differential equation (SDE) with generic vector fields $\mu$ and $\sigma$ 
\begin{align}\label{eq:SDE_ch6sec1}
  \dif X_t = \mu(X_t) \dif t + \sigma(X_t) \dif W_t, \quad X_0=x \in \R^d
\end{align}
and driven by a multi-dimensional Brownian motion $W$.
The aim of this article is to construct a map
\begin{align}
\mathbf{DISCRETIZE}_n: (\mu, \sigma, x) \mapsto (M_n,x)  
\end{align}
parametrized by a discretization parameter $n=1,2,\ldots$ that takes as inputs the drift $\mu$ and the diffusion vector field $\sigma$, as well as the starting position $x$, and which outputs the transition matrix $M$ of a Markov chain\footnote{Throughout we refer to a discrete-time and discrete-space Markov process when we refer to a Markov chain (MC).} (MC),  denoted as $X^n = (X^n_i)_{i=0,1,2,\ldots}$ with starting value $X^n_0=x$, and such that
\begin{align}
  \Law{X^n} \to \Law{X} \text{ as }{n \to \infty}\text{ and the state space of the Markov Chain }X^n \text{ is a sparse tree.}    
\end{align}
Naive discretization strategies generically lead to an explosion of the state space of the MC.
\begin{example}\label{Example}
  Consider the one-dimensional SDE 
  \begin{align}\label{eq:simple SDE}
    \dif X_t = \sigma(X_t)\dif W_t, \quad X_0=x \in \R
  \end{align}
and the Markov chain $X^n$ given (implicitly)  
    as the discretized Euler scheme 
    \[
      X^n_{i+1} =X^n_{i} + n^{-1/2} \sigma(X^n_{i}) B_i,
    \]
    where $B_i$ denote independent Bernoulli random variables.
    Then $X^n$ converges in law to $X$, but for generic volatility $\sigma$ the cardinality of the state space of $X^n$ grows in time $i$ as $O(2^i)$ because the tree nodes of the random walk ``do not recombine''. 
    However, the special case $\sigma(x)=1$ shows that ``small state spaces'' are possible since in this case the simple symmetric random walk 
    \[X^n_{i+1} = X^n_i + n^{-1/2} B_i, \]
converges weakly to $X$ -- in this case a Brownian Motion -- and the cardinality of the state space grows only linearly with the number of steps since $X^n$ evolves on a ``recombining tree''.  
\end{example}
Although Example~\ref{Example} is elementary, it motivates the arguably most popular approach to time-space discretization, namely to transform the SDE \eqref{eq:SDE_ch6sec1} into an SDE with constant volatility where naive time-space discretization strategies work. 
The main contribution of this work is to completely avoid such transformations (change of measure, Doss--Sussman, Lamperti, etc.) by using ideas from convex geometry about reducing the support of discrete measures; so called \emph{recombination of measures}.
This results in several attractive properties of $\mathbf{DISCRETIZE}_n$ compared to classical approaches: in particular, it naturally extends to the multi-variate case and even for one-dimensional diffusion it relaxes assumptions on the vector fields $\mu$ and $\sigma$.
Moreover, the construction is universal in the sense that when applied to many SDEs $X,Y,..$, the resulting MCs $X^n,Y^n,\ldots$ evolve in the same state space that is a lattice, in fact on the same recombining tree with different branching probabilities. %

\paragraph{Related Work.}
The problem of approximating an SDE by a MC arises in many areas: 
\begin{description}
\item[Mathematical Finance.]
  One of the earliest motivations came from American option pricing and resulted in the famous Cox--Ross--Rubinstein~\cite{Cox1979} model which uses a MC to approximate the Black--Scholes SDE. 
  A big improvement was made in the work of \citet{Nelson1990}, which uses the Lamperti transform to construct a recombining tree so that the number of states of the MC increases at most linearly with the number of steps, thus avoiding the exponential explosion of states in \cite{Cox1979}. 
  This approach has been refined and extended to models beyond Black--Scholes, such as multi-variate SDEs \cite{Boyle1989,Ekvall1996}, such as GBM, CIR and CEV models \cite{Boyle1988,Gamba2001,Gamba2007,Nawalkha2007, Xu2009,ANLONG1995,Finucane1996}, regime-switching \cite{R.H.2010,Liu2013,Wahab2009}, stochastic volatility models \cite{Florescu2008,Leisen1998,Moretto2009,Leisen1998,Moretto2009}, and many other model-specific solutions \cite{Ho1995,Haahtela2010,Lok2019}. 
  We also highlight that approaches that do not rely on transformations have been developed such as \cite{RUBINSTEIN1994,Rubinstein1998,Simonato2008} that use probability density functions approximations such as Edgeworth expansions, model specific approaches \cite{Haahtela2010,Lok2019}, or moment approximations \cite{Costabile2012,Costabile2016,LECCADITO2012,Ji2010}.
All of these have areas where they are advantage but among the drawbacks are
they rely on either genuinely one-dimensional arguments, exploit model-specific
properties, can result in negative transition probabilities that arise due to
transformations or moment approximations.

\item[Optimal Control.]
  American option pricing can be seen as a special case of an optimal control problem where MC approximations are classic; see \citet{Kushner1990} and~\cite[Chapter 10, Theorem 6.2]{Kushner2001} for general background. 
  An essential tool in this literature is to construct MCs by matching local moments \cite{Kushner1990,Kushner2001}\footnote{See e.g. \cite[Equation (1.3), page 71]{Kushner2001}.} which ensures convergence to the SDE.
  These local consistency conditions are also central to our construction (Definition~\ref{def:local_consistency} and Lemma~\ref{lem:eqv_loc_cons})\footnote{In particular, for the schemes we build the conditions required by \cite{Kushner1990,Kushner2001} can be proved to be equivalent thanks to the bounds of Theorems~\ref{th:mm_d=1}, \ref{th:mm_d=2}, \ref{th:mm_d=d}.}. 
  Note that in general the topology of weak convergence is too coarse to guarantee that the solution of an optimal control problem for MCs can be transferred to SDEs but under additional assumptions this holds; in particular, we show that this is the case for our lattice-tree model approximation under essentially the same assumption used in the approximation of \cite{Kushner2001}. 
  We also highlight \cite{Haahtela2010,Lok2019}

\item[Formal Verification.]
  More recently, the formal verification community has started to investigate how to formally verify SDEs \cite{HuP03,BJ06,ZMMAL14,ZTA17}. 
  Software tools, such as PRISM \cite{Kwiatkowska2004} or Storm \cite{dehnert2017storm}, exploit structural properties of the transition matrix of MCs to allow scalable formal proofs that are applied to check whether probabilistic temporal-logic properties hold or not.  
  These tools can be leveraged for the formal verification of SDEs by using them on MC approximations that are close in law to the given SDE.
  In this context, it is especially important that firstly, the discretization applies to large classes of SDEs under generic assumptions on drift and volatility vector fields; secondly, it is essential that the approximating MC transition matrix has structure, such as sparsity. %

\end{description}
Similarly, the topic of reducing the support of a measure while matching statistics given as integrals against a set specified functions (such as monomials) is well studied:
\begin{description}
\item[Recombination.] If the measure is discrete then it is a direct consequence of Carath\'eodory's Theorem that such a reduction is possible.
However, the proof of Carath\'eodory's Theorem is not constructive and the design of algorithms that carry out such a measure construction efficiently is still the subject of recent research; e.g.~just over the last ten years the articles \cite{maria2016a,Maalouf2019,Litterer2012,hayakawa2021estimating,hayakawa2021positively,Hayakawa2020MonteCarloCC, Cosentino2020} provide novel algorithms.
  Recombination has already been used in the context of SDE simulations: in order to make ``Cubature on Wiener Space'' \cite{Lyons2004CubatureOW} efficient, a recombination step is applied iteratively over time, similar in spirit to our construction; see~\cite{Litterer2012} for a discussion on how powerful this can be for cubature methods on Wiener space. 
  However, there the focus is to match the marginal distribution of an SDE at a fixed time but the total number of possible states that can be visited over time grows very quickly.
  This makes this approach too expensive to store the whole model which in turn is required for the above mentioned applications in finance, optimal control, and formal verification.
  Nevertheless, combining the result of this paper with cubature paths for Brownian motion might be an interesting future research venue.
\end{description}

{%

\paragraph{Contribution and Outline.}
Our approach is inspired by the above mentioned local consistency conditions.
The novelty is that we iterate Carath\'eodory's recombination to obtain the algorithm $(M_n,x)\coloneqq \mathbf{DISCRETIZE}_n(\mu,\sigma,x)$, such that the resulting MC $X^n=(X^n_i)$ with transition matrix $M_n$ satisfies the following points: 
\begin{enumerate}
\item the MC $X_n$ is a always a stochastic process, that is the transition matrix $M_n$ is a stochastic matrix.
  This is noteworthy, since one of the biggest drawbacks of the landmark paper~\citet{Nelson1990} and many of the above mentioned literature, is that it can result in ``negative transition probabilities'' between nodes, already in the one-dimensional case.
\item the transition matrix $M_n$ is sparse and as a consequence the support of $X_n$ grows slowly.
  This allows to store the whole model, which in turn allows the use in optimal stopping or formal verification algorithms. 
\item the approach is not tailored to specific models under the standard regularity assumptions on the vector fields, and it extends to the multi-dimensional case.
  For dimensions $d\ge 3$ some additional assumptions are needed, but notice that there are very few baselines to compare to, since most results are model-specific or only apply in the one-dimensional case. 
\item The state space is a universal lattice in the terminology of \citet{Chen1999}: for every $n$ there exists a $\gamma_n$ such that $\mathbf{DISCRETIZE}_n$ applied to many SDEs $X,Y,..$ gives MCs $X^n,Y^n,\ldots$ that evolve on the same recombining tree that has as nodes a subset of the lattice $\gamma_n\Z^d$.%
\end{enumerate}
The structure of the paper is as follows: 
Section~\ref{sec: main construction} provides general background and informally describes the algorithm $\mathbf{DISCRETIZE}_n$. 
Sections~\ref{sec: Models for One-Dimensional Diffusions}, 
\ref{sec: Models for Two-Dimensional Diffusions} and~\ref{sec: Models for Multi-Dimensional Diffusions} contain the main theoretical results and are divided based on the dimension treated: Section \ref{sec: Models for One-Dimensional Diffusions} treats one-dimension state spaces and Section \ref{sec: Models for Two-Dimensional Diffusions} treats two-dimensional state spaces.
The general case is treated in \ref{sec: Models for Multi-Dimensional Diffusions} and makes stronger assumptions on the vector fields.
Section~\ref{sec:prev_works} compares to previous work and expands on applications such as optimal stopping. 
Finally, Section~\ref{sec:Algorithm and Experiments} turns the theoretical results of the previous sections into algorithms and presents numerical experiments. 
\section*{Acknowledgments}
The authors want to thank The Alan Turing Institute and the University of Oxford for the financial support given. 
FC is grateful to Christian Bayer and Terry Lyons for helpful remarks.
FC is supported by The Alan Turing Institute, TU/C/000021, under the EPSRC Grant No. EP/N510129/1. 
HO is supported by the EPSRC grant “Datasig” [EP/S026347/1], The Alan Turing Institute, and the Oxford-Man Institute. 
AA is supported in part by the HICLASS project (113213) from the ATI , BEIS, and Innovate UK. 

\section{Background}\label{sec: main construction}
In this Section, we provide background and outline the basic construction. 
\paragraph{Lattice Approximations.} \label{sec: diffusion moments}
Throughout we consider a $d$-dimensional SDEs in Ito form
\begin{align}\label{eq:SDE_ch6sec1}
  \dif X_t^{s,x} = \mu(X_t^{s,x}) \dif t + \sigma(X_t^{s,x}) \dif W_t, \quad\quad X_s^{s,x}=x \in \R^d,
\end{align}
on a time interval $[0,T]$ where $\mu:\R^d\to\R^d$ and $\sigma:\R^d\to\R^{d\times h}$. 
When there is no confusion, we denote $X^{s,x}$ as $X$ and we use $X^{x}$ if $s=0$.  
We recall that if $\mu,\sigma$ are Lipschitz continuous and have linear growth,  then there exists a pathwise unique solution to Equation~\eqref{eq:SDE_ch6sec1}, see for example~\citet{NIkeda2014} and for weaker conditions see \cite{Stroock2006}. 
Our goal is find a sequence $(X_n)$ of MC that converges fast in weak topology to $X$.
\begin{definition}\label{def:conv}
  Let $(X^n)_{n \ge 1}$ be a sequence of Markov chains and denote $X^n=(X^n_i)_{i \ge 0}$ to emphasize the time-coordinate $i$ .
  Further, let $(\gamma_n)_{n \ge 1}$ {$\subset (0,\infty)$} a sequence that converges to $0$. 
  We say that $(X^n)_{n \ge 1}$  
  \begin{enumerate}
  \item 
is a lattice approximation to $X$
  if 
$
    X^n_i \in \gamma_n \mathbb{Z}^d \text{ for all } n> 0, \, i \ge 0;
$
\item 
converges weakly to $X$ with {respect to a class of functions $\mathbb{F}$}, if for every $f \in \mathbb{F}$, $t \in [0,T]$
$
    \lim_{n \to \infty} \E[f(X^n_{\lfloor n t/T\rfloor})] = \E[f(X_t)];
    $
    \item 
converges weakly to $X$ with rate $\alpha$ {respect to the class of functions $\mathbb{F}$} if for every $f\in\mathbb{F}$, $t \in [0,T]$ 
      \begin{align}
        \E[f(X_t)] - \E[f(X^n_{\lfloor n t/T \rfloor })] = O(n^{-\alpha}) .
      \end{align}

  \end{enumerate}
\end{definition}
\noindent
Two important classes $\mathbb{F}$ of test functions for our approach are the continuous differentiable functions with polynomial growth and, resp.~with linear growth: 
$C_P^l(E,F)$ denotes the space of functions from $E$ to $F$ $l$ times continuously differentiable with polynomial growth, including their derivatives;  
similarly $C_b^l(E,F)$ denotes the space of functions from $E$ to $F$ $l$ times continuously differentiable with linear growth, with uniformly bounded derivatives.
\paragraph{State Space Growth.}
An essential requirement is that the number of states that the MC can attain grows slowly as time progresses.
Example~\ref{Example} shows that in naive discretization schemes the support of $X_i^n$ grows exponentially in $i$, but sometimes the growth can be at most polynomial.
\begin{definition}\label{def:recombines}%
  A Markov chain $X^n=(X^n_i)_{i=0,1,2,\ldots}$ is sparse if $\operatorname{card}[\operatorname{supp}(X_{i}^{n})]=O(i^d)$ as $i\to \infty$. 
\end{definition}
\noindent
The same definition appears in \citet{Nelson1990} but there they use instead the term ``recombining'' for what we call a sparse MC in Definition \ref{def:recombines}.
We prefer the term ``sparse'' for two reasons: firstly, to avoid confusion with the ``recombination of measure'' that is essential in our construction; secondly, we want to highlight that, unlike previous work, the emphasis of our approach is not the construction of a recombining tree -- instead, we focus directly on controlling the size of the support. 
If a MC is sparse in the sense of Definition~\ref{def:recombines}, then its transition matrix has sparse row entries (since there are $d^i$ lattice points and only $O(i^d)$ can be reached). 
That the resulting MC evolves on a recombining lattice tree is a consequence of this approach, rather than its main goal (when $n$ gets large, see Section~\ref{sec:Algorithm and Experiments} for examples).  

\paragraph{Local Consistency.}
Denoting with \[\Sigma(x) = \sigma(x)\sigma(x)^\top\] and applying Ito's lemma\footnote{$\mu, \sigma \in C^4_b$ is sufficient; 
~\cite{Kloeden1992, Stroock2006} for weaker conditions} 
shows that there exists a $q\ge1$ and $\alpha>0$ and natural number $N$ such that for all $n>N$
{
\begin{align}  |\EE[X_{t+n^{-1}}-X_{t}|X_{t}=x]-\mu(x)n^{-1}|\leq&c(1+|x|^{q})n^{-1-\alpha}\\|\EE[(X_{t+n^{-1}}-X_{t}-n^{-1}\mu(X_{t}))^{\otimes2}|X_{t}=x]-\Sigma(x)n^{-1}|\leq&c(1+|x|^{q})n^{-1-\alpha},
\end{align}
}
where the constant $c=c(\mu,\sigma)$ depends only the vector fields. 
These two estimates characterize the SDE %
and it is a classic result that MCs that locally approximate these two moments converge weakly. 
\begin{definition}\label{def:local_consistency} [Local consistency \cite[page 328]{Kloeden1992}]
  We say that a sequence $(X^n)_{n \ge 1}$ of MCs is locally $\delta(n)$-consistent with $X$, if there exists a function $n \mapsto \delta(n)$ such that $\delta(n)\to 0$ as $n\to\infty$ and 
  for every $i\leq n$ %
  \begin{align*}\label{eq:KP_consistency}
    \EE\left(\EE\left[\frac{X_{i+1}^{n}-X_{i}^{n}}{n^{-1}}\Big|X_{i}^{n}\right]-\mu(X_{i}^{n})\right)^{2}=O(\delta(n)), 
    \text{ and }\EE\left(\EE\left[\frac{(X_{i+1}^{n}-X_{i}^{n})^{\otimes2}}{n^{-1}}\Big|X_{i}^{n}\right]-\Sigma(X_{i}^{n})\right)^{2}=O(\delta(n)).
\end{align*}
\end{definition}
\begin{theorem}\label{th:KP}\cite[Theorem 9.7.4]{Kloeden1992} %
  Let $(X^n)$ be locally $\delta(n)$ consistent with $X$ and further assume %
\begin{align}
\EE[|X_{i+1}^{n}-X_{i}^{n}|^{6}]\leq&{\delta(n)n^{-2}}\text{ and }\EE[\max_{i}|X_{i}^{n}|^{2q}]\leq  c(1+|X_{0}|^{2q}) \quad \text{ for every } q\ge 1,
\end{align}
and that the vector fields coefficients $\mu, \sigma$ are in $C^4_b$.
Then\footnote{The rate of convergence is stated in the proof of \cite[Theorem 9.7.4]{Kloeden1992}.} %
\begin{align}\label{eq:KP_conv}
|\EE f(X_{n}^{n})-\EE f(X_{T})|\leq c\cdot\left(\min\left\{ \sqrt{\delta(n)},n^{-1/2}\right\} \right),\quad f\in C_{P}^{3}(\R^{d},\R), 
\end{align}  
where the constant $c$ {depends only on $x, \mu, \sigma, f$.}%
\end{theorem}
Clearly for SDEs the convergence criterion \eqref{eq:KP_conv} is equivalent to the convergence criterion of Definition~\ref{def:conv}.
\begin{remark}\label{rem:weak_conv}
  Our regularity assumptions follow closely \cite{Kloeden1992} from which Theorem \ref{th:KP} is taken. 
  The literature also provides similar results under similar regularity assumptions.
  For example~\citet{Stroock2006,Ethier2009} provides general results regarding Markov processes, whereas \citet{Kushner1990,Kushner2001} generalises to controlled SDEs (we return to this in Section~\ref{sec:Control and Optimal Stopping}) although none of these comes with explicit rates, as in Theorem \ref{th:KP}.
\end{remark}

\paragraph{Lattice-Tree models by Recombination.}\label{sec: building tree models}
We construct a lattice approximation $X^n$ to $X$ by specifying the transition matrix $P^{(n)}_{x,y} \coloneqq \Prob(X_{i+1}^n=y|X_i^n=x)$ inductively, by iterating over $i$.
Let $X_0=x$ a.s.,  %
and define $X^n_0\coloneqq x$.
Now assume that we can construct a lattice-valued random variable $Y_x$ such that 
\begin{align}\label{eq:increment moments}
\EE[Y_x] &\approx {\mu(x)}{n^{-1}},\quad\quad 
 \EE[Y_x^{\otimes 2}] \approx {\mu(x)^{\otimes 2}}{n^{-2}}+{\Sigma(x)}{n^{-1}}. 
\end{align}
We then set $P_{x,y}^{(n)} \coloneqq \Prob(Y_x=y-x)$ which determines the row $P^{(n)}_{x,:}$.
For each $y$ such that $P^{(n)}_{x,y}\neq 0$ we then repeat the same procedure for the row $P^{(n)}_{y,:}$, see Algorithm~\ref{algo:tree_recomb}.
If additionally each $Y_x$ has a small support and the support is constrained on a lattice $\gamma_nQ\Z^d$, with $Q$ denoting a $d\times d$-matrix, then it follows that the transition matrix $P^{(n)}$ is sparse and that 
\begin{align}\label{eq:two moments}
  \EE[X^n_1 - X^n_0|X^n_0=x]  &\approx {\mu(x)}{n^{-1}},\quad\,\,\,
  \EE[(X^n_1 - X^n_0)^{\otimes 2}|X^n_0=x] \approx {\mu(x)^{\otimes 2}}{n^{-2}}+{\Sigma(x)}{n^{-1}}, 
\end{align}
with $X^n$ a MC that has the lattice $Q\gamma\Z^d$ as its state space. 
So far we have made the strong assumption that the lattice-valued random variable $Y_x$ exists: 
this is not obvious, since formally this means to solve the non-linear and lattice-constrained system  
\begin{align}\label{eq:fundamental_system}
  \begin{cases}
    \sum_{i=1}^{r}p_{i}l_{i}\approx {\mu(x)}{n^{-1}}\\
    \sum_{i=1}^{r}p_{i}l_{i}^{\otimes2}\approx {\mu(x)^{\otimes 2}}{n^{-2}}+{\Sigma(x)}{n^{-1}}%
  \end{cases}\\\text{subject to}:p_{k}\ge0,\quad\sum_{k=1}^{r}p_{k}=1,\quad l_{i}\in\gamma Q\Z^{d}, 
\end{align}
where $\sum_ip_i \delta_{l_i}$ denotes the law of the increment $Y_i$.
The main result of the following sections is to show that -- perhaps surprisingly -- this is in general possible, and it can be done efficiently in terms of rate of convergence as well as size of the state space of $X^n$. 
In order to show the existence of such a lattice-valued random variable $Y_x$ that has small support $r$ and that approximately matches the first two moments \eqref{eq:two moments}, we make use of a classic theorem of Carath\'eodory.
\begin{theorem}[Carath\'eodory~\cite{DeLoera2013} ]\label{th:cath}
  Let $\{x_i\}_{i=1}^N$ be a set of $N$ points in $\R^n$ and $N > n+1$.
  Any point $z$ that lies in the convex hull of these $N$ points, 
  $z$ can be expressed as a convex combination of maximum $n+1$ points, 
  i.e.~$z=\sum_{j=1}^{n+1}p_j x^\star_j$, where $x^\star_j\in\{x_i\}_{i=1}^N$ and $0\leq p_i\leq 1$, $\sum_i p_i=1$ .
\end{theorem}
Section~\ref{sec: Models for One-Dimensional Diffusions} discusses this for $d=1$ and Section~\ref{sec: Models for Two-Dimensional Diffusions}  resp.~Section~\ref{sec: Models for Multi-Dimensional Diffusions} for $d=2$ resp.~$d\ge 3$.

\section{Lattice-Tree Models for One-Dimensional Diffusions}\label{sec: Models for One-Dimensional Diffusions}
We now apply the methodology outlined at the end of Section~\ref{sec: main construction} to one-dimensional SDEs.
The first step 
establishes the existence of a random variable $Y$ that matches (or closely approximates) any given first two moments and is additionally supported on a lattice.
The second step 
uses this random variable to define the MC. 
\subsection*{Matching the Moments}%
For $\gamma>0$, $\ma\in\R$ denote 
\begin{align}
\floor{\ma}_{\gamma}:=&\max\{\gamma z,z\in\Z\,\,s.t.\,\,\gamma z\le\ma\},\quad\quad
\ceil{\ma}_{\gamma}:=\min\{\gamma z,z\in\Z\,\,s.t.\,\,\gamma z>\ma\}.
\end{align}
If $\gamma=1$ we suppress the subscript $\gamma$ and the above reduces to the standard notation for the ceiling and floor functions. 
\begin{theorem}\label{th:t_d=1}
Let $\ma,\vb \in\R$, $\vb\ge0$, and $\gamma>0$. 
Then there exists a random variable $Y$ such that
\begin{enumerate}
\item $ \E[Y]=\ma \text{ and } |\EE [Y^{ 2}] - \ma^2- \vb^2| \le \frac{\gamma^2}{4},$
\item $\operatorname{supp}(Y) \subset \{ y \in \gamma \Z : |y| \le \sqrt{a^2+b^2}+\gamma \} .$
\end{enumerate}
Moreover, if  
\begin{align}\label{eq: condition variance}
  \ma^{2}+\vb^{2}\ge(\left\lceil \ma \right\rceil _{\gamma}^{2}-\left\lfloor \ma \right\rfloor _{\gamma}^{2})(\ma-\left\lfloor \ma \right\rfloor _{\gamma}) \gamma^{-1}+\left\lfloor \ma \right\rfloor _{\gamma}^{2},
\end{align}
then one can additionally assume that $\EE[ Y^{ 2}] = \ma^2 + \vb^2$.
\end{theorem}
\begin{proof}
  Wlog suppose $\ma\ge0$.
  For a discrete random variable $Y$ denote $p_i \coloneqq \Pr(Y=y_i)$. 
If  
      \begin{align}\label{eq:barycenter}
      \EE[(Y,Y^2)] = \sum_i p_i (y_i,y_i^2) = (a, m)
      \end{align}
for a $m$ such that $|m - (a^2+b^2)| \le \gamma^2/4$ and $y_i \in \gamma \Z$, then the result follows. 
We distinguish three cases, depending on the position of $(a,a^2+b^2)$ in relation to the points $(\ceil{a}_\gamma, \ceil{a}_\gamma^2)$, $(\floor{a}_\gamma, \floor{a}_\gamma^2)$, as depicted in Figure~\ref{fig:proof_main_th}.   
Formally,%
\begin{align}\label{eq:cases}
  \begin{cases}
\text{Case 1}\,\,  : \ma^{2}\leq\ma^{2}+\vb^{2}\leq\frac{\left\lceil \ma\right\rceil_{\gamma}^{2}-\left\lfloor \ma\right\rfloor_{\gamma}^{2}}{\left\lceil \ma\right\rceil_{\gamma}-\left\lfloor \ma\right\rfloor_{\gamma}}(\ma-\left\lfloor \ma\right\rfloor_{\gamma})+\left\lfloor \ma\right\rfloor_{\gamma}^{2},\\ 
\text{Case 2}\,\, :
  \frac{\left\lceil \ma\right\rceil_{\gamma}^{2}-\left\lfloor \ma\right\rfloor_{\gamma}^{2}}{\left\lceil \ma\right\rceil_{\gamma}-\left\lfloor \ma\right\rfloor_{\gamma}}(\ma-\left\lfloor \ma\right\rfloor_{\gamma})+\left\lfloor \ma\right\rfloor_{\gamma}^{2}\leq\ma^{2}+\vb^{2}\leq\left\lceil \ma\right\rceil_{\gamma}^{2},\\
\text{Case 3}\,\,  : \left\lceil \ma\right\rceil_{\gamma}^{2}\leq\ma^{2}+\vb^{2}. 
  \end{cases}
\end{align}
\begin{figure}[ht]
\centering
\includegraphics[width=0.37\textwidth]{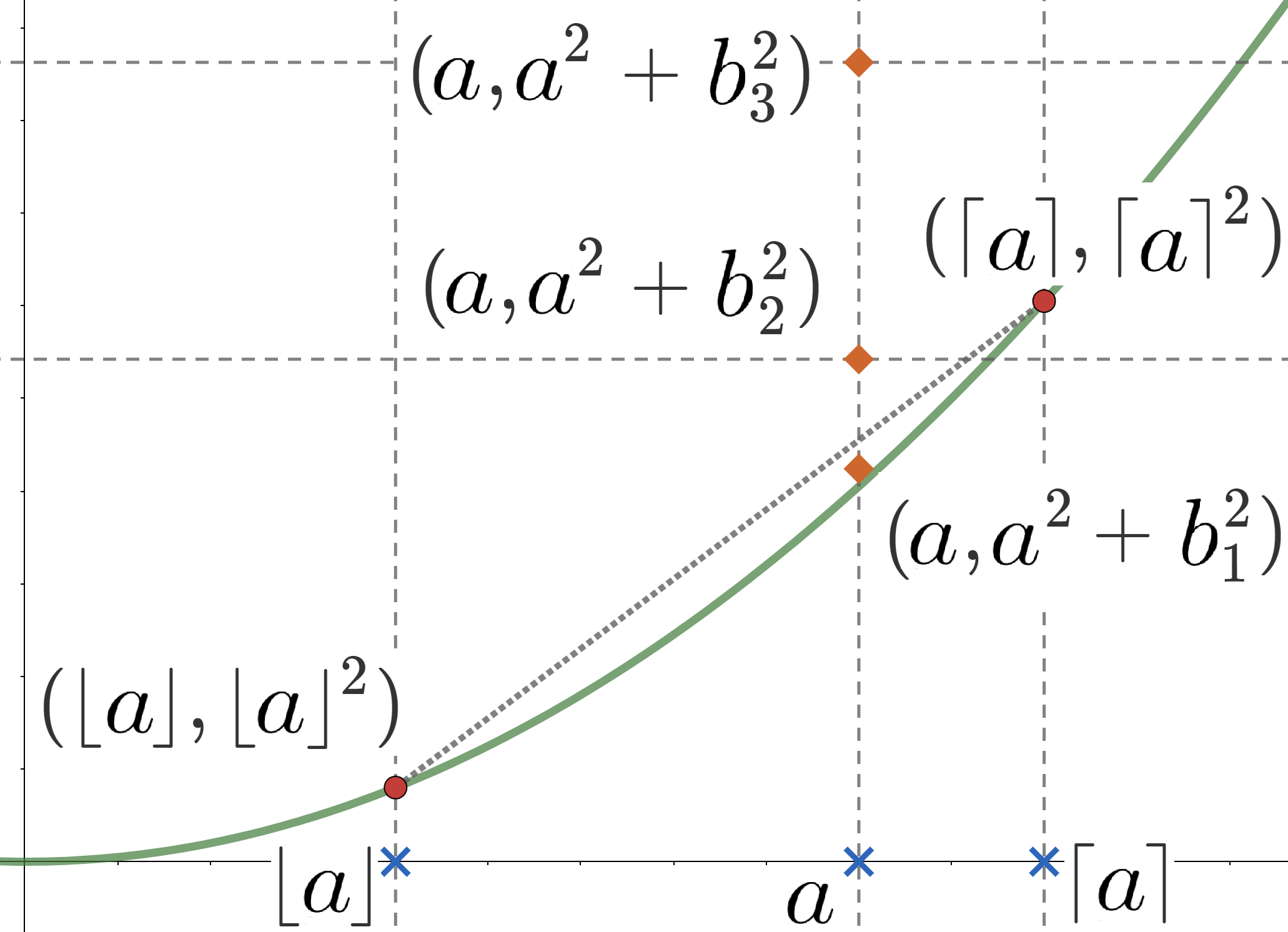}
\caption[Proof of Theorem~\ref{th:t_d=1}.]{Proof of Theorem~\ref{th:t_d=1}. $\vb_i$ represents $b$ under the hypothesis of Case $i$.}\label{fig:proof_main_th}
\end{figure}

\noindent\textbf{Case 1}. %
We claim that we can choose $y_1=\left\lfloor \ma\right\rfloor_{\gamma}$, $y_2=\left\lceil \ma\right\rceil_{\gamma}$, $p_1=\frac{\ma-\left\lfloor \ma\right\rfloor_{\gamma}}{\left\lceil \ma\right\rceil_{\gamma}-\left\lfloor \ma\right\rfloor_{\gamma}}$ and $p_2=1-p_1$.
This gives $\EE[ Y]=\ma$ and the bound on $|\EE Y^2 - \ma^2-\vb^2 |$ follows
from a direct calculation: 
\begin{align}
\EE [Y^{2}]=&\frac{\ma-\left\lfloor \ma\right\rfloor_{\gamma}}{\left\lceil \ma\right\rceil_{\gamma}-\left\lfloor \ma\right\rfloor_{\gamma}}\left\lceil \ma\right\rceil_{\gamma}^{2}+\left(1-\frac{\ma-\left\lfloor \ma\right\rfloor_{\gamma}}{\left\lceil \ma\right\rceil_{\gamma}-\left\lfloor \ma\right\rfloor_{\gamma}}\right)\left\lfloor \ma\right\rfloor_{\gamma}^{2}=\frac{\ma-\left\lfloor \ma\right\rfloor_{\gamma}}{\gamma}(\left\lceil \ma\right\rceil_{\gamma}^{2}-\left\lfloor \ma\right\rfloor_{\gamma}^{2})+\left\lfloor \ma\right\rfloor_{\gamma}^{2}\\=&\frac{\ma-\left\lfloor \ma\right\rfloor_{\gamma}}{\gamma}((\left\lfloor \ma\right\rfloor_{\gamma}+\gamma)^{2}-\left\lfloor \ma\right\rfloor_{\gamma}^{2})+\left\lfloor \ma\right\rfloor_{\gamma}^{2}=(\ma-\left\lfloor \ma\right\rfloor_{\gamma})(\gamma+2\left\lfloor \ma\right\rfloor_{\gamma})+\left\lfloor \ma\right\rfloor_{\gamma}^{2}.
\end{align}
Assuming that $\vb$ satisfies the inequalities~\eqref{eq:cases}-Case 1, results in 
\begin{align}
\EE [Y^{2}]-\ma^{2}-\vb^{2}=&(\ma-\left\lfloor \ma\right\rfloor_{\gamma})(\gamma+2\left\lfloor \ma\right\rfloor_{\gamma})+\left\lfloor \ma\right\rfloor_{\gamma}^{2}-\ma^{2}-\vb^{2}\\\geq&(\ma-\left\lfloor \ma\right\rfloor_{\gamma})(\gamma+2\left\lfloor \ma\right\rfloor_{\gamma})-\frac{\left\lceil \ma\right\rceil_{\gamma}^{2}-\left\lfloor \ma\right\rfloor_{\gamma}^{2}}{\gamma}(\ma-\left\lfloor \ma\right\rfloor_{\gamma})\\\geq&(\ma-\left\lfloor \ma\right\rfloor_{\gamma})(\gamma+2\left\lfloor \ma\right\rfloor_{\gamma})-(\gamma+2\left\lfloor \ma\right\rfloor_{\gamma})(\ma-\left\lfloor \ma\right\rfloor_{\gamma})\geq0.
\end{align}
Denoting $\theta:=(\ma-\left\lfloor \ma\right\rfloor_{\gamma})/\gamma\in[0,1]$ we write 
\begin{align}
\EE[ Y^{2}]-\ma^{2}-\vb^{2}=&(\ma-\left\lfloor \ma\right\rfloor_{\gamma})(\gamma+2\left\lfloor \ma\right\rfloor_{\gamma})+\left\lfloor \ma\right\rfloor_{\gamma}^{2}-\ma^{2}-\vb^{2}\\\leq&(\ma-\left\lfloor \ma\right\rfloor_{\gamma})(\gamma+2\left\lfloor \ma\right\rfloor_{\gamma})+\left\lfloor \ma\right\rfloor_{\gamma}^{2}-\ma^{2}\\\leq&\theta\gamma(\gamma+2\left\lfloor \ma\right\rfloor_{\gamma})+\left\lfloor \ma\right\rfloor_{\gamma}^{2}-(\left\lfloor \ma\right\rfloor_{\gamma}+\theta\gamma)^{2}\leq\theta\gamma^{2}(1-\theta)\leq\frac{\gamma^{2}}{4}. 
\end{align}

\noindent
\textbf{Case 2}. %
The point $(\ma, \ma^2+\vb^2)$ is contained in the convex hull spanned by the four points\footnote{Other solutions can exist.}
  \begin{align}\label{eq:convex hull}
    \{(y,y^2): y \in \{\pm \ceil{\ma}_\gamma,\pm \floor{\ma}_\gamma\}\}. 
  \end{align}
By Carath\'eodory's Theorem~\ref{th:cath}, $(\ma, \ma^2+\vb^2)$ is a convex combination of three points of the set \eqref{eq:convex hull}. 
Hence, \eqref{eq:barycenter} holds with $m=\ma^2+\vb^2$. 

\noindent\textbf{Case 3}. %
Since $\ma^2+\vb^2 \le \ceil*{ \sqrt{\ma^2+\vb^2 } }_\gamma^2 $ the point $(\ma, \ma^2+\vb^2)$ is contained in the convex hull spanned by the four points   
$
  \{(y,y^2) : y \in \{\pm \ceil{a}_\gamma, \pm \ceil*{ \sqrt{a^2+b^2 } }^2_\gamma\} \}, 
$
and by Carath\'eodory's Theorem~\ref{th:cath} we can conclude.

To finish the proof, we need the bounds on the support of the r.v. $Y$, however it is sufficient to note that this is true by construction. 
\end{proof}

For special cases of $\ma,\vb$ one gets stronger results (with simpler proofs).  
\begin{corollary}\label{cor:mu=0_dim1}
  Let $\vb\ge 0$, and $\gamma>0$.
  Then there exists a random variable $Y$ with
  \begin{enumerate}
  \item $\EE[Y] =0 $ and $ \EE [Y^{ 2} ]= \vb^2$, 
  \item $\operatorname{supp}(Y) \subset \{y \in\gamma \Z\text{ : } | y |\leq  \vb + \gamma \}$.
  \end{enumerate}
\end{corollary}

\begin{proof}
Let us call $\ma_{++}=\min\{\gamma z>0, z\in\Z\text{ s.t. }\gamma^2z^2>\vb^2\}$, then one possible solution (but not the only one) is the r.v. $Y$ with support on $\pm\ma_{++}$ with probability $\vb^2/(2\ma_{++}^2)$ and $0$ with probability $1-\vb^2/\ma_{++}^2$.
\end{proof}

\begin{corollary}
  Let $\ma, \vb\in\R$, $\vb\ge 0$, and $\gamma>0$.
  Then there exists a r.v. $Y$ with 
  \begin{enumerate}
  \item$ \EE [Y] = \ma $ and $ \EE [Y^2] =\ma^2 + \vb^2 $;
  \item $\operatorname{supp}(Y) \subset \{y \in\gamma \Z\cup\{\ma\}\text{ : } | y |\leq  \sqrt{\ma^2 + \vb^2} + \gamma \}$.
  \end{enumerate}
\end{corollary}

\begin{proof}
It should be clear from Figure~\ref{fig:proof_main_th} that adding $\ma$ to the lattice $\gamma\Z$, the convex hull of the $6$ points 
$(\pm\lfloor \ma \rfloor_\gamma, \lfloor \ma \rfloor_\gamma^2), 
(\pm \ma, \ma^2), 
(\pm\lceil \ma \rceil_\gamma, \lceil \ma \rceil_\gamma^2), 
(\pm \lceil \sqrt{\ma^2+\vb^2}\rceil_\gamma, \lceil \sqrt{\ma^2+\vb^2}\rceil_\gamma^2)
$ contains the point $(\ma, \ma^2+\vb^2)$ and the Caratheodory's Theorem~\ref{th:cath} concludes the proof.
\end{proof}
\begin{remark}\label{rem:constraint_tree_recomb}
In some applications, it can be required to constrain the support of the r.v. on specific intervals, e.g. in finance the price quantity should be non-negative. 
In this case, from Figure~\ref{fig:proof_main_th}, it should be clear that one could proceed looking for the smallest point $\gamma z_\star$ such that the point $(\ma, \ma^2+\vb^2)$ lays in the convex hull of the points 
$(\gamma, \gamma^2),$ 
$(\lfloor \ma \rfloor_\gamma, \lfloor \ma \rfloor_\gamma^2),$ 
$(\lceil \ma \rceil_\gamma, \lceil \ma \rceil_\gamma^2)$ and 
$(\gamma z_\star, \gamma^2 z^2_\star)$.
\end{remark}
\subsection*{Constructing the Lattice-Tree Model}%
Now we are ready to build the Markov chain $X^n$ which approximates $X$ and such that the state space recombines if the SDE coefficients are bounded 
for the case $d=1$. 
We extend to models with dimension greater than $1$ in the next Sections. 
 \begin{theorem}\label{th:mm_d=1}
   Let $d=1$ and $\mu, \sigma \in C^4_b$, $\beta>0$ and $\sigma_{\min}^2 := \inf_x \sigma^2(x)$.
   Set %
   \begin{align}
     \gamma_{n}\coloneqq 
     \begin{cases}
2n^{-1/2}\sigma_{\min} & \text{, if }\sigma_{\min}>0,\\
\sqrt{n^{-1-\beta}}, & \text{, otherwise}.
\end{cases}
   \end{align}
   Then there exists a lattice approximation $(X^n)_{n \ge 1}$ to $X$  on $X_0 \cup \gamma_n\Z$  respect to the class of functions $f\in C^3_P$ with rate 
   $1/2$ if $\sigma_{\min}>0$, or $\min(\beta,\frac{1}{2})$ otherwise. 
   Moreover
   \begin{align}
     \begin{cases}
      \operatorname{card}[\operatorname{supp}(X_{i}^{n})]\leq\frac{icn^{-1/2}}{\gamma_{n}}+ic &\text{, if }\mu,\sigma \text{ are bounded,}\\
      \operatorname{card}[\operatorname{supp}(X_{i}^{n})]\leq\frac{i(cn^{-1/2}+\gamma_{n})\exp(cin^{-1/2})}{\gamma_{n}}& \text{, if }\mu,\sigma \text{ have linear growth}.
    \end{cases}
     \end{align}
     Note that, if $\mu, \sigma$ are bounded and $\sigma_{\min}>0$, then $\operatorname{card}[\operatorname{supp}(X_{i}^{n})]\le ic$, i.e. $X^n$ recombines.
 \end{theorem}
The proof below allows to state the same result but with the lattice approximation supported on the ``shifted lattice'' $X_0 + \gamma_n\Z$ instead of $X_0 \cup \gamma_n \Z$.
However, an advantage of the above formulation is that several diffusions can be supported on the same lattice, namely $\gamma_n \Z$ and the union of their starting points, whereas shifting the lattice would require to build a complete new lattice every time a diffusion is added.   
We prepare the proof of Theorem~\ref{th:mm_d=1} with the following Lemma. 
\begin{lemma}\label{lemma: convergence assumptions}
  If there exists $\alpha >0 $ and a $q \ge 1$ such that for $m=1, 2$
  \begin{enumerate}
   \item\label{it:my_local_consistency}
     $\left|\EE[(X_{i+1}-X_{i})^{\otimes m}|X_{i}=X^n_i]-\EE[(X_{i+1}^{n}-X_{i}^{n})^{\otimes m}|X_{i}^{n}=X^n_i]\right|\leq c(1+|X^n_i|^{q})n^{-1-\alpha}$,
  \item \label{it: uniform moments}
    $\EE[\max_{i}|X_{i}^{n}|^{q}]\leq  c(1+|X_{0}|^{q})$,
\item\label{it:6power_conditioned}
   $ \EE [ |X_{i+1}^{n}-X_{i}^{n}|^{6} | X_i^n ]\leq c(1+|X_{i}^{n}|^{q})n^{-2-2\alpha}$,
\end{enumerate}
then $X^n$ converges weakly with rate $\min\{\alpha,\frac{1}{2}\}$ to $X$ for $f\in C^3_P$.
\end{lemma}
\begin{proof}
 The first two items imply the local consistency~\eqref{eq:KP_consistency} for $\delta(n)=n^{-2\alpha}$, see Appendix, Lemma~\ref{lem:eqv_loc_cons}.
 Similarly, the second and third items imply the conditions required by Theorem~\ref{th:KP}, in particular the third item, thanks to item~\ref{it: uniform moments}, implies $ \EE [ |X_{i+1}^{n}-X_{i}^{n}|^{6} ]\leq \delta(n)n^{-1}$ with $\delta(n)=n^{-2\alpha}$. 
\end{proof}

\begin{proof}[Proof of Theorem~\ref{th:mm_d=1}]%
We first deal with the strictly elliptic case when $\sigma^2_{\text{min}}\coloneqq \inf_x \sigma^2(x)> 0$. 
Theorem~\ref{th:t_d=1} applied with
$
\ma={\mu(x)}{n^{-1}}$ and $\vb={\sigma(x)}{ n^{-1/2}}
$
guarantees the existence of a random variable $Y_x$ such that
\begin{align}\label{eq: right moments}
  \EE[Y_x]={\mu(x)}{n^{-1}} , \quad \EE[Y_x^2] = {\mu^2(x)}{n^{-2}} +{\sigma^2(x)}{n^{-1}}
\end{align}
and such that 
\begin{align}\label{eq: grid}
\operatorname{supp}(Y_x) =  \left\{ z \gamma_n,z\in\Z\text{ s.t. }|z| \gamma_n\leq\sqrt{ {\mu(x)^{2}}{n^{-2}}+ {\sigma(x)^{2}}{n^{-1}} }+ \gamma_n \right\},
\end{align}
if we can show that  
\begin{align}\label{eq:to_be_shown}
\ma^{2}+\vb^{2}\ge(\left\lceil \ma \right\rceil _{\gamma_n}^{2}-\left\lfloor \ma \right\rfloor _{\gamma_n}^{2})(\ma-\left\lfloor \ma \right\rfloor _{\gamma_n}) \gamma_n^{-1}+\left\lfloor \ma \right\rfloor _{\gamma_n}^{2}.
\end{align}
Substituting \eqref{eq:to_be_shown} yields
\begin{align}\label{eq:to_be_satisfied_ch6_d=1}
\frac{\sigma(x)^{2}}{n}\!&\geq\!\frac{\ceil{\frac{\mu(x)}{n}}_{\gamma_{n}}^{2}\!-\!\floor{\frac{\mu(x)}{n}}_{\gamma_{n}}^{2}}{\gamma_{n}}\!\left(\!\frac{\mu(x)}{n}\!-\!\floor{\frac{\mu(x)}{n}}_{\gamma_{n}}\!\right)\!+\!\floor{\frac{\mu(x)}{n}}_{\gamma_{n}}^{2}\!-\!\frac{\mu^{2}(x)}{n^{2}}\!=\!\theta\gamma_{n}^{2}(1-\theta), 
\end{align}
where $\theta$ in the last equality is defined as 
$
  \theta\coloneqq (\frac{\mu(x)}{n}-\floor{\frac{\mu(x)}{n}}_{\gamma_{n}} )\gamma_{n}^{-1}. 
$
Since $\theta\in[0,1]$, choosing $\gamma_n\le {2}{n^{-1/2}} \sigma_{\min} $ implies \eqref{eq:to_be_satisfied_ch6_d=1} because $\theta(1-\theta)\in[0, 1/4]$ if $\theta\in[0,1]$.
Hence, the existence of the random variable $Y_x$ matching the moments~\eqref{eq: right moments} and with support on the grid \eqref{eq: grid} follows.
By applying this construction to different $x$ we can define a Markov chain $X^{n}$ as described in Algorithm~\ref{algo:tree_recomb} by inductively filling out the transition matrix $P^{(n)}$. \\
If $X^n_0$ is not in the lattice $\gamma_n\Z$ some attention is needed: to have $X_1^n$ in $\gamma_n\Z$ it is necessary to build $Y_{X^n_0}$ on the lattice $\gamma_n\Z-X_0^n$, indeed $X^n_1 = X^n_0+Y_{X^n_0}$ would be in $\gamma_n\Z$. 
This is also equivalent to build the r.v. $\tilde Y_{X^n_0}$ on $\gamma_n\Z$ s.t. 
$\EE\tilde Y_{X^n_0} = \mu(X_0^n)+X_0^n$ and 
$\EE\tilde Y_{X^n_0}^{2} = \sigma(X_0^n)^2 + (\mu(X_0^n)+X_0^n)^2$ and define $Y_{X^n_0}:=\tilde Y_{X^n_0} -X_0^n$ with support on $\gamma_n\Z-X_0$ s.t. 
$\EE Y_{X^n_0} = \mu(X_0^n)$ and 
$\VV Y_{X^n_0}= \sigma(X_0^n)^2$. \\
The resulting Markov chain $X^{n}$ evolves on the state space $\gamma_n \Z$ and it only remains to show the claimed convergence rate and growth of the support as the time progresses.
We do this by showing that the assumptions of Lemma~\ref{lemma: convergence assumptions} apply. 
Lemma~\ref{lemma: convergence assumptions}-Item~\eqref{it:my_local_consistency} holds by~\eqref{eq: right moments}.
Lemma~\ref{lemma: convergence assumptions}-Item~\eqref{it: uniform moments} holds in great generality for Markov chains following this construction, see Lemma~\ref{lem:maximum_scheme}.  
For the remaining Lemma~\ref{lemma: convergence assumptions}-Item~\eqref{it:6power_conditioned} note that for some $z_j\in\Z$
\begin{align}\label{eq:d=1_ineq_assumptions_ch6}
\EE\left|Y_{x}\right|^{6}\!=\!&\sum_{j=1}^{3}p_{j}|2n^{-1/2}\sigma_{\text{min}}z_{j}|^{6}\!\leq\!c\left|n^{-6}\mu(x)^{6}\!+\!n^{-3}\sigma(x)^{6}\!+\!2^{6}n^{-3}\sigma_{\text{min}}^{3}\right|\!\leq\!c(1\!+\!|x|^{6})n^{-3},
\end{align}
which we can apply to $x=X_{i}^n$ and shows that the assumptions of Lemma~\ref{lemma: convergence assumptions} are satisfied which in turn finishes the proof of the convergence rate.

When $\mu, \sigma$ have linear growth, in order to bound the growth of the support we use the discrete {Gr\"onwalls Inequality, cf. Lemma~\ref{lem:Gronw_dis}, }
\begin{align}
  |X_{i}^{n}|\leq&\sum_{j=0}^{i-1}\left|Y_{X_{j}^{n}}\right|\leq\sum_{j=0}^{i-1}n^{-1/2}c(1+|X_{j}^{n}|)+\gamma_{n}\leq (cn^{-1/2}i+i\gamma_{n})\exp ( cin^{-1/2}).
\end{align}
Since the points $l_k \in \gamma_n\Z$ such that $|l_{k}|\leq c $ are $c/\gamma_n$ we have that %
\[
  \operatorname{card}[\operatorname{supp}(X_{i}^{n})]\leq\frac{i(c{n^{-1/2}}+\gamma_{n})\exp({cin^{-1/2})}}{\gamma_{n}}.
\]
If $\mu, \sigma$ are bounded, it is easy to obtain through similar reasoning that 
\[
  |X_{i}^{n}|\le\sum_{j=0}^{i-1}\left|Y_{X_{j}^{n}}\right|\leq icn^{-1/2}+i\gamma_n, 
\] 
which implies that 
\[
  \operatorname{card}[\operatorname{supp}(X_{i}^{n})]\leq\frac{icn^{-1/2}}{\gamma_{n}}+ic.
\]
This finishes the proof of the elliptic case $\sigma^2_{\text{min}}>0$.

The non-elliptic case, $\inf \sigma^2(x)=0$, follows analogously: 
as before we rely on Theorem~\ref{th:t_d=1} that ensures the existence of a random variable $Y_x$ such that 
\begin{align}\label{eq: right moments non-elliptic}
  \EE\left[Y_{x}\right]=n^{-1}\mu(x),\quad\EE[Y_{x}^{2}-n^{-2}\mu^{2}(x)-n^{-1}\sigma^{2}(x)]\le\frac{\gamma_{n}^{2}}{4}
\end{align}
and that is supported on {$\gamma_n\Z$}. 
Note that, unlike in the elliptic case, this time we do not match the second moment exactly,  
and the second moment condition shows that for any any $\beta>0$ we can take $\gamma_n\leq n^{\frac{-1-\beta}{2}}$.

For the rates of convergence it is enough to observe that $\sqrt{\delta(n)}$ of Theorem~\ref{th:KP} is equal to 
$\beta$ due to approximation of the second moment, see Lemma~\ref{lemma: convergence assumptions}-Item~\eqref{it:my_local_consistency}.
Moreover, remember that the rate of convergence in Theorem~\ref{th:KP} is bounded by $1/2$. 

Note that both the choices of $\gamma_n$ satisfy $n\gamma_n^2<c$ for some $c>0$, as required by Lemma~\ref{lem:maximum_scheme}.
\end{proof}
For the special case $\mu(x)=0$ we can prove the following Corollary. 
\begin{corollary}
Let $d=1$, $\mu=0$, $\sigma \in C^4_b$ and for any $c>0$ set $\gamma_{n} = c n^{-1/2}$. 
   If $X_0^n\in\gamma_n\Z$, then there exists a lattice approximation $(X^n)_{n \ge 1}$ to $X$ on $\gamma_n\Z$  respect to the class of functions $f\in C^3_P$ with rate $1/2$. 
   Moreover, the same bounds on the cardinality of the support of Theorem~\ref{th:mm_d=1} hold true. 
\end{corollary}
\begin{proof}
If $\mu(x)=0$ we can use Corollary~\ref{cor:mu=0_dim1} and then follow the same reasoning of the previous proof. 
Since the first two moments can be approximated exactly independently of the lattice considered, we can choose $\gamma_n=c n^{-1/2}$ for any $c>0$. 
The dependence of $\gamma_n$ on $n$ is chosen because required by Lemma~\ref{lem:maximum_scheme}: $\gamma_n$ must satisfy $n\gamma_n^2<c$ for some $c>0$. 
\end{proof}
\section{Lattice-Tree Models for Two-Dimensional Diffusions}\label{sec: Models for Two-Dimensional Diffusions}
We now carry out the same procedure in the two-dimensional case.
The key difference to the one-dimensional case is in the first step: the existence of a random variable that matches the moments closely and is supported on a lattice is much more involved to be shown. 
Nevertheless, we show that by using the eigenvalues of the matrix $\vB$ the system in \eqref{eq:fundamental_system} can be solved.  
\subsection*{Matching the Moments}%
Throughout this Section we fix $\ma \in \R^d$ and $\vB \in \R^{d\times d}$.
We assume that $\vB$ is symmetric and positive definite and denote the eigenvalues of $\vB$ with $(\lambda_i)_i$ ordered from biggest to smallest, 
$\lambda_1 \ge \lambda_2 \ge \cdots \ge \lambda_{\min}$ with $\lambda_{\min}$ denoting the smallest eigenvalue. 
We need the following two theorems to prove the main statements of this Section. 
  \begin{theorem}\label{th:hoff}[\citet{Hoffman1953}]
  If $A$ and $B$ are normal matrices with eigenvalues $\lambda_i[A]$ and $\lambda_j[B]$, then there exists a suitable numbering of the eigenvalues s.t.
  \[
  \sum_i |  \lambda_i[A] - \lambda_i[B]   |^2\le |A-B |^2.
  \]
  \end{theorem}
  \begin{theorem}\label{th:Gerschgorin}[Gerschgorin Circle Theorem \cite{Gerschgorin1931}]
  Let $B\in \R^{d\times d }$ with entries $b_{ij}$, then every eigenvalues $\lambda_i$ of $B$ lies in at least one of the Gershgorin circles 
  \[ 
	  \text{Circle}(b_{ii}):=\{c\in\mathbb{C} : |c-b_{ii}|\le \sum_{j\not=i} |b_{ij}|\}. 
  \]
  \end{theorem}
Now we are ready to prove the main statement. 
\begin{theorem}\label{th:t_d=2}
  Let $d=2$, $\ma \in \R^d$, $\vB \in \R^{d\times d}$ symmetric and positive definite, 
  $\lambda_{\min}>0$, then for every $\gamma$ such that $0<\gamma \le \sqrt{\lambda_{\min} / 3}$ there exists a random variable $Y$ such that
  \begin{enumerate}
  \item $\EE[ Y]=\ma$ and $\EE[ Y^{\otimes 2}]=\ma^{\otimes 2}+\vB$;	
  \item $\operatorname{supp}(Y)\subset \{ y\in \gamma \Z^2 : |y|_\infty \leq |\ma|_\infty + \sqrt{2\lambda_1}+\sqrt{2\lambda_2} + 6\gamma \}$. 
  \end{enumerate}
  Moreover, if $\ma =0$, the above applies for every $0 < \gamma \le \sqrt{\lambda_{\min}}$.
\end{theorem}
\begin{proof}
  We first consider the case when $\ma=0$.
  In this case, note that if $\bar Y$ is a discrete random variable that matches the second moment, $\EE \bar Y^{\otimes 2}=\vB$, and $Z$ is a Bernoulli random variable independent of $\bar Y $, such that $\Prob(Z=\pm 1)=\frac{1}{2}$, then the random variable $Y\coloneqq Z \cdot \bar Y$ has mean $\EE[Y]=a=0$ and $\EE[Y^{\otimes 2}]=\vB$.
  Hence, solving the system~\eqref{eq:fundamental_system} reduces to find $l_1, \ldots, l_r \in\gamma\Z^2$ such that 
    \begin{align}\label{eq:second moment, zero mean}
      \sum^r_{i=1}p_{i}l_{i}l_{i}^{T}=\vB,\quad	p_i\ge0,\quad \sum_{i=1}^r p_i=1,
    \end{align}
for some $r\in\N_+$.
Moreover, since $\vB$ is symmetric, there exists $Q\in\R^{d\times d}$ orthonormal whose columns are the eigenvectors of $\vB$, 
such that $\vB = Q\Lambda Q^T$ and $\Lambda$ is diagonal with elements the eigenvalues of $\vB$, which are strictly positive by assumption. 
Applying this, allows to rewrite \eqref{eq:second moment, zero mean} as
\begin{align}\label{eq:original}
\sum_{i=1}^r p_{i}Q^{T}l_{i}l_{i}^{T}Q=\Lambda,\quad	p_i\ge0,\quad \sum_{i=1}^r p_i=1.
\end{align}
Henceforth, we denote with $q_1=(q_{11}, q_{12})$ and $q_2=(q_{21}, q_{22})$ the two eigenvectors of $\vB$. 
It follows that if $q_1$ is an eigenvector for $Q$, then $-q_1$ is also an eigenvector for the same eigenvalue and $-q_1$ still forms an orthonormal basis with $\pm q_2$.
We use these properties of eigenvalues in the following without loss of generality. \\\\
\textbf{Case $q_{12}=0$ or $q_{21}=0$.}
  Assume $q_{12}=0$ (the case $q_{21}=0$ follows analogously), then $q_1=e_1$ and $q_2=e_2$.
 Hence, $Q=I$ and we can refer to {Theorem~\ref{th:t_d=d},} which applies to $d$ greater than $2$ and $\vB$ semi-positive definite.\\\\
\textbf{Case $q_{11}, q_{21}>0$.}
 We first note that
 \begin{small}
\begin{align}
\begin{cases}
\langle q_{1},q_{1}\rangle^{2}=q_{11}^{2}+q_{12}^{2}=1\\
\langle q_{2},q_{2}\rangle^{2}=q_{21}^{2}+q_{22}^{2}=1 &\\
\langle q_{1},q_{2}\rangle=q_{11}q_{21}+q_{12}q_{22}=0, 
\end{cases}&\text{ if and only if }\quad\begin{cases}
q_{11}=\text{sign}(q_{11})\sqrt{1-q_{12}^{2}}\\
q_{21}=\text{sign}(q_{21})\sqrt{1-q_{22}^{2}}\\
\frac{q_{11}}{q_{12}}=-\frac{q_{22}}{q_{21}}. 
\end{cases}
\end{align}
\end{small}
Using the last equation on the system on the right hand side, we obtain 
\[
\text{sign}(q_{11})\text{sign}(q_{12})\frac{\sqrt{1-q_{12}^{2}}}{|q_{12}|}=-\text{sign}(q_{21})\text{sign}(q_{22})\frac{|q_{22}|}{\sqrt{1-q_{22}^{2}}}. 
\] 
This in turn implies that \[\text{sign}(q_{11})\text{sign}(q_{12})=-\text{sign}(q_{21})\text{sign}(q_{22}).\]
We now claim that without loss of generality, we can assume that $q_{11}>0$, $q_{12}>0$ and $q_{21}>0, q_{22}<0$: indeed if we flip the sign of the eigenvectors they are still eigenvectors and we can re-label $q_1, q_2$. %
Moreover, note that 
\begin{align}
\frac{q_{11}}{q_{12}}=-\frac{q_{22}}{q_{21}}\Longleftrightarrow q_{22}^{2}&=\left(\frac{q_{11}}{q_{12}}\right)^{2}\left(1-q_{22}^{2}\right)\Longleftrightarrow q_{22}^{2}\left(1+\left(\frac{q_{11}}{q_{12}}\right)^{2}\right)=\left(\frac{q_{11}}{q_{12}}\right)^{2}\\\Longleftrightarrow q_{22}&=-\frac{\frac{q_{11}}{\sqrt{1-q_{11}^{2}}}}{\sqrt{1+\frac{q_{11}^{2}}{1-q_{11}^{2}}}}=-q_{11}, 
\end{align}
which implies that 
\begin{align}\label{eq:q1_q2}
  q_{1}=&\left(q_{11},\sqrt{1-q_{11}^{2}}\right) \text{ and }
q_{2}=\left(\sqrt{1-q_{11}^{2}},-q_{11}\right).
\end{align}
In the following, we assume that $1>q_{11}\ge 1/\sqrt{2}$ - otherwise we can state $q_1, q_2$ as functions of $q_{12}$ and proceed similarly.\footnote{The implications regarding $q_1, q_2$ at first glance can seem hostile, however geometrically they are more intuitive.
They depend on the fact that $q_1, q_2$ form an orthonormal basis and changing the signs they are still eigenvectors and they still form an orthonormal basis.}
Now note that if there exist $r$ points  $l_1,\ldots, l_r\in\gamma\Z^2$ s.t.
\begin{align}\label{eq:inCH}
\Lambda\in \operatorname{ConvexHull}\{(Q^T l_1)^{\otimes 2}, \ldots, (Q^T l_r)^{\otimes 2}\}, 
\end{align}
then Equation~\eqref{eq:original} hold true for some $p_i$ to be computed.
It only remains to show that \eqref{eq:inCH} holds.
This is a hard problem, but in our case a simple argument works.
We make use of the simple observation that if $A=\{a_1,\ldots,a_r\} \subset \R^h$ and $b \in \R^h$ s.t.
\begin{align}\label{eq: hypercube}
  \forall \sim \in \{\le, \ge \}^h \quad \exists a \in A \text{ such that } a \sim b \text{ then } b\in \operatorname{ConvexHull}(A). 
\end{align}
We now apply \eqref{eq: hypercube} with $b=\Lambda$ and $A=\{ (Q^T l_1)^{\otimes 2},\ldots,(Q^T l_r)^{\otimes 2} \}$.
Since $\Lambda$ and $(Q^T l_i)^{\otimes 2}$ are symmetric we work in dimension $h=3$ instead of $h=4$.
The premise in Equation~\eqref{eq: hypercube} reduces to show that for each of the eight possible choices of $\sim=(\sim_2, \sim_1, \sim_1, \sim_3 )\in \{\ge,\le\}^3$, there exists a point $l \in \gamma\Z^2$ such that 
\begin{align}\label{eq:sytemCH}
  \left(\begin{array}{cc}
\langle q_{1},l\rangle^{2} & \!\!\!\!\langle q_{1},l\rangle\langle q_{2},l\rangle\\
\langle q_{1},l\rangle\langle q_{2},l\rangle & \!\!\!\!\langle q_{2},l\rangle^{2}
\end{array}\right)\sim\left(\!\!\begin{array}{cc}
\lambda_{1} & \!\!0\\
0 & \!\!\lambda_{2}
                                \end{array}\!\!\right). 
\end{align}
Spelled out in coordinates the above reads as
\begin{small}
\begin{align}
\begin{cases}
\langle q_{1},l\rangle\langle q_{2},l\rangle\!\sim_{1}\!0\\
\langle q_{1},l\rangle^{2}\!\sim_{2}\!\lambda_{1}\\
\langle q_{2},l\rangle^{2}\!\sim_{3}\!\lambda_{2}
\end{cases}&\hspace{-0,5cm}\xLeftrightarrow[\text{Using Equation~\eqref{eq:q1_q2}}]{}\begin{cases}
\left(q_{11}l_{i1}\!+\!\sqrt{1\!-\!q_{11}^{2}}l_{i2}\!\right)\!\left(\!\sqrt{1\!-\!q_{11}^{2}}l_{i1}\!-\!q_{11}l_{i2}\right)\!\sim_{1}\!0\\
\left(q_{11}l_{i1}\!+\!\sqrt{1\!-\!q_{11}^{2}}l_{i2}\right)^{2}\!\sim_{2}\!\lambda_{1}\\
\left(\sqrt{1\!-\!q_{11}^{2}}l_{i1}\!-\!q_{11}l_{i2}\right)^{2}\!\sim_{3}\!\lambda_{2}. 
\end{cases}
\end{align}
\end{small}
We denote $\varphi \coloneqq \frac{\sqrt{1-q_{11}^{2}}}{q_{11}}$, therefore the system we study is 
\begin{small}
\begin{align}\label{eq:system_simplified}
\begin{cases}
\left(l_{i1}+\varphi l_{i2}\right)\left(\varphi l_{i1}-l_{i2}\right)\sim_{1} 0\\\left(l_{i1}+\varphi l_{i2}\right)^{2}\sim_{2} \frac{\lambda_{1}}{q_{11}^{2}}\\\left(\varphi l_{i1}-l_{i2}\right)^{2}\sim_{3}\frac{\lambda_{2}}{q_{11}^{2}}, 
\end{cases}
\end{align}
\end{small}
where $l_i=(l_{i1},l_{i2})$, $l_{i1}\in\gamma\Z$ and $l_{i2}\in\gamma\Z$. 
We now claim that for every of the $r=8$ possible choices of $\sim_1,\sim_2,\sim_3 \in \{\le, \ge \}^3$ one can find a $l_i \in \gamma \Z^2$ such that 
Equation~\eqref{eq:system_simplified} holds. 
All the possible combinations of $\sim_1,\sim_2,\sim_3$ are $
\{ \ge \ge \ge, 
\ge \ge \le, 
\ge \le \ge, 
\ge \le \le, 
\le \ge \ge, 
\le \ge \le, 
\le \le \ge, 
\le \le \le \}
$
and we study them in order. We suppress the subscript $i$ for $l_i = (l_{i1}, l_{i2})$. Let us also resume some bounds 
\begin{align}\label{eq:q_11_varphi_bounds}
\frac{1}{\sqrt{2}}\leq q_{11}<1,\quad1\leq\frac{1}{q_{11}}<\sqrt{2},\quad0\leq\varphi\leq1,\quad1\leq\frac{1}{\varphi}<\infty.
\end{align}
\begin{itemize}
\item $(\ge \ge \ge)$ We have the following system, whose solutions are on the right
\begin{small}
\begin{align}
\begin{cases}
\left(l_{1}+\varphi l_{2}\right)\left(\varphi l_{1}-l_{2}\right)\geq0\\
\left(l_{1}+\varphi l_{2}\right)^{2}\geq\frac{\lambda_{1}}{q_{11}^{2}}\\
\left(\varphi l_{1}-l_{2}\right)^{2}\geq\frac{\lambda_{2}}{q_{11}^{2}}
\end{cases}&\!\!\!\!\!\!\Longleftrightarrow\begin{cases}
\text{Solution 1}\ge\ge & \text{Solution 2}\le\le\\
l_{1}\geq\frac{\sqrt{\lambda_{1}}}{q_{11}}-\varphi l_{2}\quad\quad\text{ or }\quad & l_{1}\leq-\frac{\sqrt{\lambda_{1}}}{q_{11}}-\varphi l_{2}\\
l_{1}\geq\frac{1}{\varphi}\frac{\sqrt{\lambda_{2}}}{q_{11}}+\frac{1}{\varphi}l_{2} & l_{2}\leq-\frac{1}{\varphi}\frac{\sqrt{\lambda_{2}}}{q_{11}}+\frac{1}{\varphi}l_{2}. 
\end{cases}
\end{align}
\end{small}
``Solution 1 $\ge\ge$'' means that we consider the case where $l_{1}+\varphi l_{2}\ge0,$ $ \varphi l_{1}-l_{2}\ge0$ and similarly for the other case ``Solution 2 $\le\le$''. In this and the successive cases we focus on  ``Solution 1 $\ge\ge$''. 
In this case the solution can be obtained choosing e.g. $l_2=0$ and $l_1$ big enough, which however would result in an unbounded solution given that $1/\varphi\geq 1$ is not bounded from above. A bounded solution is\footnote{
In this and the successive cases, we have found really useful plotting the lines ($l_1$ as function of $l_2$) represented by the two equations in ``Solution 1''.
}
\begin{align}
l_{2}=&\left\lfloor -\frac{\sqrt{\lambda_{2}}}{q_{11}}\right\rfloor _{\gamma}\in[-\sqrt{2\lambda_{2}}-\gamma,-\sqrt{\lambda_{2}}]\\l_{1}=&\left\lceil \frac{\sqrt{\lambda_{1}}}{q_{11}}-\varphi\left\lfloor -\frac{\sqrt{\lambda_{2}}}{q_{11}}\right\rfloor _{\gamma}\right\rceil _{\gamma}\in\left[\sqrt{\lambda_{1}},\sqrt{2\lambda_{1}}+\sqrt{2\lambda_{2}}+2\gamma\right], 
\end{align}
which is bounded. This is obtained choosing $l_2$ so that $ \frac{1}{\varphi}\frac{\sqrt{\lambda_{2}}}{q_{11}}+\frac{1}{\varphi}l_{2}\approx 0$ and then using it in the first equation of  ``Solution 1''.
\item $(\ge \ge \le)$ We have the following system whose, solutions are on the right
\begin{small}
\begin{align}
\begin{cases}
\left(l_{1}+\varphi l_{2}\right)\left(\varphi l_{1}-l_{2}\right)\geq0\\
\left(l_{1}+\varphi l_{2}\right)^{2}\geq\frac{\lambda_{1}}{q_{11}^{2}}\\
\left(\varphi l_{1}-l_{2}\right)^{2}\leq\frac{\lambda_{2}}{q_{11}^{2}}
\end{cases}&\!\!\!\!\!\!\Longleftrightarrow\begin{cases}
\text{Solution 1}\ge\ge & \text{Solution 2}\le\le\\
l_{1}\geq\frac{\sqrt{\lambda_{1}}}{q_{11}}-\varphi l_{2}\quad\quad\quad\quad\text{ or }\quad & l_{1}\leq-\frac{\sqrt{\lambda_{1}}}{q_{11}}-\varphi l_{2}\\
\frac{1}{\varphi}l_{2}\leq l_{1}\leq\frac{1}{\varphi}\frac{\sqrt{\lambda_{2}}}{q_{11}}+\frac{1}{\varphi}l_{2} & -\frac{1}{\varphi}\frac{\sqrt{\lambda_{2}}}{q_{11}}+\frac{1}{\varphi}l_{2}\leq l_{1}\leq\frac{1}{\varphi}l_{2}. 
\end{cases}
\end{align}
\end{small}
This case is more tedious. 
We choose $l_1=\gamma n $, $l_2 = \gamma \lfloor n \varphi\rfloor $, 
where $n$ is the smallest integer satisfying $\gamma n\geq\sqrt{2\lambda_{1}} $. 
This implies that the first equation of ``Solution 1'' is satisfied: using Equation~\eqref{eq:q_11_varphi_bounds}
\[
\gamma n\geq\sqrt{2\lambda_{1}}\geq\frac{\sqrt{\lambda_{1}}}{q_{11}}\geq\frac{\sqrt{\lambda_{1}}}{q_{11}}-\varphi\gamma\left\lfloor n\varphi\right\rfloor. 
\]
We choose therefore $l_1=\gamma n=\lceil \sqrt{2\lambda_1}\rceil_\gamma\in [0,\sqrt{2\lambda_1}+\gamma  ] $. 
The same bound for $l_2$ holds true since 
$l_2 = \gamma\lfloor n \varphi\rfloor\leq l_1$, again using Equation~\eqref{eq:q_11_varphi_bounds}.\\
It remains to prove that the second equation of ``Solution 1'' is satisfied. 
Dividing the second equation of ``Solution 1'' by $l_2$ we prove the left inequality
\[
\frac{l_{1}}{l_{2}}=\frac{\gamma n}{\gamma\left\lfloor n\varphi\right\rfloor }\geq\frac{n}{n\varphi}=\frac{1}{\varphi}, 
\]
and, since by assumption $\gamma\le\sqrt{\lambda_2}$ we conclude that 
\begin{small}
\begin{align}
\frac{l_{1}}{l_{2}}-\frac{1}{\varphi}=\frac{\gamma n \varphi-\gamma\left\lfloor n \varphi\right\rfloor }{\gamma\left\lfloor n \varphi\right\rfloor \varphi}\leq\frac{\gamma}{\gamma\left\lfloor n \varphi\right\rfloor \varphi}=\frac{\gamma}{l_{2}\varphi}
\implies
\frac{l_{1}}{l_{2}}\leq\frac{1}{\varphi}\left(1+\frac{\gamma}{l_{2}}\right)\leq\frac{1}{\varphi}\left(1+\frac{1}{l_{2}}\frac{\sqrt{\lambda_{2}}}{q_{11}}\right)
\end{align}
\end{small}
because $1\leq\frac{1}{q_{11}}\leq\sqrt{2}$, which is equivalent to the second equation of ``Solution 1''. 

\item $(\ge \le \ge)$ We have the following system, whose solutions are on the right
\begin{small}
\begin{align}
\begin{cases}
\left(l_{1}+\varphi l_{2}\right)\left(\varphi l_{1}-l_{2}\right)\geq0\\
\left(l_{1}+\varphi l_{2}\right)^{2}\leq\frac{\lambda_{1}}{q_{11}^{2}}\\
\left(\varphi l_{1}-l_{2}\right)^{2}\geq\frac{\lambda_{2}}{q_{11}^{2}}
\end{cases}&\!\!\!\!\!\!\!\!\!\Longleftrightarrow\begin{cases}
\text{Solution 1}\ge\ge & \text{Solution 2}\le\le\\
-\varphi l_{2}\leq l_{1}\leq\frac{\sqrt{\lambda_{1}}}{q_{11}}-\varphi l_{2}\quad\text{ or }\quad & -\frac{\sqrt{\lambda_{1}}}{q_{11}}-\varphi l_{2}\leq l_{1}\leq-\varphi l_{2}\\
l_{1}\geq\frac{1}{\varphi}\frac{\sqrt{\lambda_{2}}}{q_{11}}+\frac{1}{\varphi}l_{2} & l_{1}\leq-\frac{1}{\varphi}\frac{\sqrt{\lambda_{2}}}{q_{11}}+\frac{1}{\varphi}l_{2}. 
\end{cases}
\end{align}
\end{small}
In this case we can take 
\begin{small}
\begin{align}
l_{2}=&\left\lfloor -\frac{\sqrt{\lambda_{2}}}{q_{11}}\right\rfloor _{\gamma}\in[-\sqrt{2\lambda_{2}}-\gamma,-\sqrt{\lambda_{2}}],\quad
l_{1}=\left\lceil -\varphi\left\lfloor -\frac{\sqrt{\lambda_{2}}}{q_{11}}\right\rfloor _{\gamma}\right\rceil _{\gamma}\in\left[0,\sqrt{2\lambda_{2}}+2\gamma\right].
\end{align}
\end{small}
This solution has been found choosing $l_2$ so that $ \frac{1}{\varphi}\frac{\sqrt{\lambda_{2}}}{q_{11}}+\frac{1}{\varphi}l_{2}\approx 0$ and then using it in the first equation of ``Solution 1''. 
The second equation of ``Solution 1'' is satisfied since $l_1$ is positive.
For the first one, the left inequality is satisfied by definition of $l_1$, whilst the right inequality because $\gamma\leq\sqrt{\lambda_1}\leq \sqrt{\lambda_1}/q_{11}. $
\item $(\ge \le \le)$ We have the following system, whose solutions are on the right
\begin{small}
\begin{align}
\begin{cases}
\left(l_{1}+\varphi l_{2}\right)\left(\varphi l_{1}-l_{2}\right)\geq0\\
\left(l_{1}+\varphi l_{2}\right)^{2}\leq\frac{\lambda_{1}}{q_{11}^{2}}\\
\left(\varphi l_{1}-l_{2}\right)^{2}\leq\frac{\lambda_{2}}{q_{11}^{2}}
\end{cases}&\!\!\!\!\!\!\!\!\!\Longleftrightarrow\begin{cases}
\text{Solution 1}\ge\ge & \text{Solution 2}\le\le\\
-\varphi l_{2}\leq l_{1}\leq\frac{\sqrt{\lambda_{1}}}{q_{11}}-\varphi l_{2}\quad\text{ or }\quad & -\frac{\sqrt{\lambda_{1}}}{q_{11}}-\varphi l_{2}\leq l_{1}\leq-\varphi l_{2}\\
\frac{1}{\varphi}l_{2}\leq l_{1}\leq\frac{1}{\varphi}\frac{\sqrt{\lambda_{2}}}{q_{11}}+\frac{1}{\varphi}l_{2} & -\frac{1}{\varphi}\frac{\sqrt{\lambda_{2}}}{q_{11}}+\frac{1}{\varphi}l_{2}\leq l_{1}\leq\frac{1}{\varphi}l_{2}. 
\end{cases}
\end{align}
\end{small}
In this case two possible solutions for the case ``Solution 1'' are $(l_1, l_2)=(0, 0)$ and $(l_1, l_2)=(\gamma, 0)$, since by assumption $\gamma\leq \min\{\sqrt{\lambda_1},\sqrt{\lambda_2}\}$.
\item $(\le \ge \ge)$ We have the following system, whose solutions are on the right
\begin{small}
\begin{align}
\begin{cases}
\left(l_{1}+\varphi l_{2}\right)\left(\varphi l_{1}-l_{2}\right)\leq0\\
\left(l_{1}+\varphi l_{2}\right)^{2}\ge\frac{\lambda_{1}}{q_{11}^{2}}\\
\left(\varphi l_{1}-l_{2}\right)^{2}\ge\frac{\lambda_{2}}{q_{11}^{2}}
\end{cases}&\!\!\!\!\!\!\Longleftrightarrow\begin{cases}
\text{Solution 1}\ge\le & \text{Solution 2}\le\ge\\
l_{1}\geq\frac{\sqrt{\lambda_{1}}}{q_{11}}-\varphi l_{2}\quad\quad\text{ or }\quad & l_{1}\leq-\frac{\sqrt{\lambda_{1}}}{q_{11}}-\varphi l_{2}\\
l_{1}\leq-\frac{1}{\varphi}\frac{\sqrt{\lambda_{2}}}{q_{11}}+\frac{1}{\varphi}l_{2} & l_{1}\geq\frac{1}{\varphi}\frac{\sqrt{\lambda_{2}}}{q_{11}}+\frac{1}{\varphi}l_{2}. 
\end{cases}
\end{align}
\end{small}
We could take simply $l_1=0$ and $l_2$ big enough, but this solution would be unbounded (depending on $1/\varphi$). 
Therefore, we reason similarly to the case $(\ge \ge \ge)$: let us rewrite the equations for ``Solution 1'' respect to  $l_2$
\begin{align}
l_{2}\geq&\frac{1}{\varphi}\frac{\sqrt{\lambda_{1}}}{q_{11}}-\frac{1}{\varphi}l_{1},\quad\quad\quad
l_{2}\geq\frac{\sqrt{\lambda_{2}}}{q_{11}}+\varphi l_{1}, 
\end{align}
and we can take as a solution
\begin{align}
l_{1}=&\left\lfloor \frac{\sqrt{\lambda_{1}}}{q_{11}}\right\rfloor _{\gamma}\in\left[\sqrt{\lambda_{1}}-\gamma,\sqrt{2\lambda_{1}}\right]\\l_{2}=&\left\lceil \frac{\sqrt{\lambda_{2}}}{q_{11}}+\varphi\left\lfloor \frac{\sqrt{\lambda_{1}}}{q_{11}}\right\rfloor _{\gamma}\right\rceil _{\gamma}\in\left[\sqrt{\lambda_{2}},\sqrt{2\lambda_{2}}+\sqrt{2\lambda_{1}}+2\gamma\right].
\end{align}
This solution has been found choosing $l_1$ so that $ \frac{1}{\varphi}\frac{\sqrt{\lambda_{1}}}{q_{11}}-\frac{1}{\varphi}l_{1}\approx 0$ and then using it in the other equation.
\item $(\le \ge \le)$ We have the following system, whose solutions are on the right
\begin{small}
\begin{align}
\begin{cases}
\left(l_{1}+\varphi l_{2}\right)\left(\varphi l_{1}-l_{2}\right)\leq0\\
\left(l_{1}+\varphi l_{2}\right)^{2}\ge\frac{\lambda_{1}}{q_{11}^{2}}\\
\left(\varphi l_{1}-l_{2}\right)^{2}\leq\frac{\lambda_{2}}{q_{11}^{2}}
\end{cases}&\!\!\!\!\!\!\!\!\!\Longleftrightarrow\begin{cases}
\text{Solution 1}\ge\le & \text{Solution 2}\le\ge\\
l_{1}\geq\frac{\sqrt{\lambda_{1}}}{q_{11}}-\varphi l_{2}\quad\quad\quad\quad\text{ or }\quad & l_{1}\leq-\frac{\sqrt{\lambda_{1}}}{q_{11}}-\varphi l_{2}\\
-\frac{1}{\varphi}\frac{\sqrt{\lambda_{2}}}{q_{11}}+\frac{1}{\varphi}l_{2}\leq l_{1}\leq\frac{1}{\varphi}l_{2} & \frac{1}{\varphi}l_{2}\leq l_{1}\leq\frac{1}{\varphi}\frac{\sqrt{\lambda_{2}}}{q_{11}}+\frac{1}{\varphi}l_{2}. 
\end{cases}
\end{align}
\end{small} 
We choose $l_1=\gamma n $, $l_2 = \gamma \lceil n \varphi\rceil $, 
where $n$ is the smallest integer satisfying $\gamma n\geq\sqrt{2\lambda_{1}} $. 
This implies that the first equation of ``Solution 1'' is satisfied: using Equation~\eqref{eq:q_11_varphi_bounds}
\[
\gamma n\geq\sqrt{2\lambda_{1}}\geq\frac{\sqrt{\lambda_{1}}}{q_{11}}\geq\frac{\sqrt{\lambda_{1}}}{q_{11}}-\varphi\gamma\left\lceil n\varphi\right\rceil. 
\]
We choose therefore $l_1=\gamma n=\lceil \sqrt{2\lambda_1}\rceil_\gamma\in [0,\sqrt{2\lambda_1}+\gamma  ] $. 
The same bound for $l_2$ holds true since 
$l_2 = \gamma\lceil n \varphi\rceil\leq l_1$, again using Equation~\eqref{eq:q_11_varphi_bounds}.\\
It remains to prove that the second equation of ``Solution 1'' is satisfied. 
Dividing the second equation of ``Solution 1'' by $l_2$ we prove the left inequality
\[
\frac{l_{1}}{l_{2}}=\frac{\gamma n}{\gamma\lceil n\varphi\rceil }\leq\frac{n}{n\varphi}=\frac{1}{\varphi}, 
\]
and, since by assumption $\gamma\le\sqrt{\lambda_2}$ we conclude that 
\begin{small}
\begin{align}
\frac{l_{1}}{l_{2}}-\frac{1}{\varphi}=\frac{\gamma n \varphi-\gamma\left\lceil n \varphi\right\rceil }{\gamma\left\lceil n \varphi\right\rceil \varphi}
\geq\frac{-\gamma}{\gamma\left\lfloor n \varphi\right\rfloor \varphi}=\frac{-\gamma}{l_{2}\varphi}
\implies
\frac{l_{1}}{l_{2}}\geq\frac{1}{\varphi}\left(1-\frac{\gamma}{l_{2}}\right)\geq\frac{1}{\varphi}\left(1-\frac{1}{l_{2}}\frac{\sqrt{\lambda_{2}}}{q_{11}}\right)
\end{align}
\end{small}
because $-1\!\geq\!-\frac{1}{q_{11}}\!\geq\!-\sqrt{2}$, which is equivalent to the second equation of Solution 1. 
\item $(\le \le \ge)$ We have the following system, whose solutions are on the right
\begin{small}
\begin{align}
\begin{cases}
\left(l_{1}+\varphi l_{2}\right)\left(\varphi l_{1}-l_{2}\right)\leq0\\
\left(l_{1}+\varphi l_{2}\right)^{2}\leq\frac{\lambda_{1}}{q_{11}^{2}}\\
\left(\varphi l_{1}-l_{2}\right)^{2}\geq\frac{\lambda_{2}}{q_{11}^{2}}
\end{cases}&\!\!\!\!\!\!\!\!\!\Longleftrightarrow\begin{cases}
\text{Solution 1}\ge\le & \text{Solution 2}\le\ge\\
-\varphi l_{2}\leq l_{1}\leq\frac{\sqrt{\lambda_{1}}}{q_{11}}-\varphi l_{2}\quad\text{ or }\quad & -\frac{\sqrt{\lambda_{1}}}{q_{11}}-\varphi l_{2}\leq l_{1}\leq-\varphi l_{2}\\
l_{1}\leq-\frac{1}{\varphi}\frac{\sqrt{\lambda_{2}}}{q_{11}}+\frac{1}{\varphi}l_{2} & l_{1}\geq\frac{1}{\varphi}\frac{\sqrt{\lambda_{2}}}{q_{11}}+\frac{1}{\varphi}l_{2}. 
\end{cases}
\end{align}
\end{small}
We can proceed as in the case $\ge \le \ge$:
we can take 
\begin{align}
l_{2}=&\left\lceil \frac{\sqrt{\lambda_{2}}}{q_{11}}\right\rceil _{\gamma}\in\left[\sqrt{\lambda_{2}},\sqrt{2\lambda_{2}}+\gamma\right],\quad\quad
l_{1}=\left\lceil -\varphi\left\lceil \frac{\sqrt{\lambda_{2}}}{q_{11}}\right\rceil _{\gamma}\right\rceil _{\gamma}\in\left[-\sqrt{2\lambda_{2}}-\gamma,0\right]. 
\end{align} 
This solution has been found choosing $l_2$ so that $ \frac{1}{\varphi}\frac{\sqrt{\lambda_{2}}}{q_{11}}+\frac{1}{\varphi}l_{2}\approx 0$ and then using it in the first equation of ``Solution 1''. 
The second equation of ``Solution 1'' is satisfied since $l_1$ is negative.
For the first one, the left inequality is satisfied by definition of $l_1$, whilst the right inequality because $\gamma\leq\sqrt{\lambda_1}\leq \sqrt{\lambda_1}/q_{11}. $
\item $(\le \le \le)$ 
In this case it is easy to see that we can take the origin.
\end{itemize}
In all the cases we have studied the ``Solution 1'', however, for the ``Solution 2'' we can take $-(l_1, l_2)$.
This concludes the proof for the case $\ma=0$. 

For a general $a\in\R^d$ we can build for any $\gamma>0$ a random variable $\tilde Y$ with support on $\gamma\Z^d$ such that $\EE\tilde Y=\ma$. 
This follows from Carath\'eodory's Theorem~\ref{th:cath}, considering the smallest hypercube with vertices in $\gamma\Z^d$ that contains $\ma$, that is of the $2^d$ possible choices of $\sim\in\{\le,\ge\}^d$ determines a point $y_\sim$ on the hypercube around $\ma$ s.t. 
\begin{align}
  y_\sim\in\gamma\Z^{d}  \text{ and } y_\sim \sim \ma.  
\end{align}
We denote the resulting $2^d$ points as $y_1,\ldots,y_{2^d}$ and now estimate the variance $\VV[\tilde Y]$,
\begin{align}
\left|\VV[\tilde Y]\right|=\left|\EE\left[\left(\tilde Y-\ma\right)^{\otimes2}\right]\right|=&\left|\sum_{{i}}\tilde{p}_{i}\left(y_{{i}}-\ma\right)^{\otimes2}\right|\leq\sum_{i}\tilde{p}_{i}|y_{{i}}-\ma|^{2}\leq d\gamma^{2}.
\end{align}
We now build a random variable $Y_\star$ supported on $\gamma\Z^2$ ($d=2$) that is independent of $Y$ and has mean and variance 
$
\EE Y_\star=0 \text{ and }\VV Y_\star=\EE Y_\star^{\otimes 2}=\vB-\VV\tilde Y.
$
It follows that
$
  Y \coloneqq \tilde Y+Y_\star 
$
has mean and variance 
$
  \EE Y=\ma \text{ and }\VV Y=\vB. 
$
To build $Y_\star$ we use the first part of this proof.
Therefore we need to verify that the eigenvalues of $\vB-\VV\tilde Y$ are strictly positive. 
Since $\VV\tilde Y$ and $\vB$ are symmetric it follows from the Hoffman--Wielandt Theorem, Theorem~\ref{th:hoff}, that {there exists a permutation $\pi$ such that for any $k$}
\begin{align}
-|\VV\tilde Y|\leq\lambda_{k}\left[\vB\right.&\left.-\VV\tilde Y\right]-\lambda_{\pi(k)}\left[\vB\right]\leq|\VV\tilde Y|, 
\end{align}
where $\lambda_k[\vB-\VV\tilde Y]$ indicates the $k$-th eigenvalue of $\vB-\VV\tilde Y$. 
Therefore, 
\[
\lambda_{\min}[\vB-\VV\tilde{Y}]\geq \lambda_{\pi(\min)}\left[\vB\right]-|\VV\tilde{Y}|\geq\lambda_{\min}-2\gamma^{2}.
\]
Hence, choosing $\gamma\leq\sqrt{\lambda_{\min}/3}$ guarantees $\gamma\le\sqrt{\lambda_{\min}[\vB-\VV\tilde{Y}]}$ since
\[
\lambda_{\min}-2\gamma^{2}\ge\gamma^{2}\Longleftrightarrow\gamma\leq\sqrt{\lambda_{\min}/3}.
\] 
Note that the cardinality of the support of $Y = \tilde{Y}+Y_\star$, eventually, can be reduced using the Caratheodory's Theorem~\ref{th:cath}.  
We have proved that, by construction, the support of the r.v. $Y_\star$ is included in 
\[
\operatorname{supp}(Y_{\star})\subset\{y\in\gamma\Z^{2}:|y|_{\infty}\leq\sqrt{2\lambda_{1}[B-\VV\tilde{Y}]}+\sqrt{2\lambda_{2}[B-\VV\tilde{Y}]}+2\gamma\}.
\]
Thanks again to the Hoffman--Wielandt Theorem, Theorem~\ref{th:hoff} we have that
\begin{align}
\sqrt{2\lambda_{1}[\vB-\VV\tilde{Y}]}+\sqrt{2\lambda_{2}[\vB-\VV\tilde{Y}]}\leq&\sqrt{2\left(|\VV\tilde{Y}|+\lambda_{\pi(1)}\left[\vB\right]\right)}+\sqrt{2\left(|\VV\tilde{Y}|+\lambda_{\pi(2)}\left[\vB\right]\right)}\\
\leq&\sqrt{2\lambda_{\pi(1)}[\vB]}+\sqrt{2\lambda_{\pi(2)}[\vB]}+2\sqrt{2|\VV\tilde{Y}|}\\\le&\sqrt{2\lambda_{1}[\vB]}+\sqrt{2\lambda_{2}[\vB]}+2\gamma\sqrt{2d}.
\end{align}
This implies that the support of $Y=\tilde Y + Y_\star$ is included in 
\begin{align}
\operatorname{supp}(Y)\subset\{y\in\gamma\Z^{2}:|y|_{\infty}\leq |\ma|_\infty+\sqrt{2\lambda_{1}[B]}+\sqrt{2\lambda_{2}[B]}+6\gamma\}.
\end{align}
\vspace{-0.25cm}
\end{proof}

\begin{remark}
  We conjecture that Theorem~\ref{th:t_d=2} generalizes to dimensions higher than $d=2$, however the above brute-force proof strategy is not helpful, since it is infeasible to solve all the inequalities one by one for the equivalent of System~\eqref{eq:system_simplified} in higher dimensions. 
  Concretely, we believe that for $\ma \in \R^d$, $\vB \in \R^{d\times d}$ symmetric and positive definite, 
  $\lambda_{\min}>0$, then for every $\gamma$ such that $0<\gamma \le \sqrt{\lambda_{\min} / (d+1)}$ there exists a random variable $Y$ such that
  \begin{enumerate}
  \item $\EE[ Y]=\ma$ and $\EE[ Y^{\otimes 2}]=\ma^{\otimes 2}+\vB$; 	
  \item $\operatorname{supp}(Y)\subset \{ y\in \gamma \Z^2 : |y|_\infty \leq |\ma|_\infty + \sum_{i=1}^d\sqrt{2\lambda_i}+ c_d\gamma \}$. 
  \end{enumerate}
  Note that it only remains to show the existence of a random variable $ Y_\star$, such that $\EE[ Y_\star]=0$ and $\EE[  Y_\star^{\otimes 2}]=\vB$ for any positive definite matrix $\vB$, since this allows to simply proceed as in the second part of the proof of Theorem~\ref{th:t_d=2}. 

\end{remark}
\subsection*{Constructing the Lattice-Tree Model}%
Now we are ready to build the Markov chain $X^n$, which approximates $X$ for the 2-dimensional case and such that the state space recombines if the SDE coefficients are bounded. We denote $\Sigma = \sigma(x)\sigma(x)^T$ and with $\lambda_i(x)$ its eigenvalues. 
\begin{theorem}\label{th:mm_d=2}
  Let $\mu, \sigma \in C^4_b$, $d=2$, $\epsilon\coloneqq  \inf_x \lambda_{\min} (x)>0$ and %
  $\gamma_n=n^{-1/2}\sqrt{\epsilon/3}$. 
  There exists a lattice approximation $(X^n)_{n \ge 1}$ to $X$ on $X_0^n\cup\gamma_n\Z^2$  with respect to the class of functions $f\in C^3_P$ with rate
   $1/2$. %
  Moreover, 
 \begin{align}
 \begin{cases}
\operatorname{card}[\operatorname{supp}(X_{i}^{n})]\leq ci^{2} & ,\text{ if \ensuremath{\mu,\sigma} are bounded,}\\
\operatorname{card}[\operatorname{supp}(X_{i}^{n})]\leq ci^{2}{\exp(2cin^{-1/2})} & ,\text{ if \ensuremath{\mu,\sigma} have linear growth.}
\end{cases}
 \end{align}
\end{theorem}
\begin{proof}
We argue as in Theorem~\ref{th:mm_d=1}, but using Theorem~\ref{th:t_d=2} above.
In contrast to the one-dimensional case, the support of the random variable $Y_x$ that gives the increment $X^n_{i+1}- X^n_i$ when $X^n_i=x$ now depends also on the eigenvalues of $\Sigma$. 
However, the Gerschgorin's Circle Theorem \cite{Gerschgorin1931}, Theorem~\ref{th:Gerschgorin}, bounds the eigenvalues of $\Sigma$.  
We conclude by noting that the number of points of the lattice $\gamma_n\Z^2$ in a square of side $c$ are $(c/\gamma_n)^2$, 
which yields the square factor in the cardinality of the support.\\  
If $X^n_0$ is not in the lattice $\gamma_n\Z^2$ more care is needed: for $X_1^n$ to be supported on $\gamma_n\Z^2$ it is necessary to build the random variable $Y_{X^n_0}$ with support on the lattice $\gamma_n\Z^2-X_0^n$, indeed $X^n_1 = X^n_0+Y_{X^n_0}$ would be in $\gamma_n\Z^2$. 
This is equivalent to build the random variable $\tilde Y_{X^n_0}$ on $\gamma_n\Z^2$ such that 
$
  \EE\tilde Y_{X^n_0} = \mu(X_0^n)+X_0^n$ and $ \EE\tilde Y_{X^n_0}^{2} = \Sigma(X_0^n)+ (\mu(X_0^n)+X_0^n)^{\otimes2}. 
$
We define $Y_{X^n_0}\coloneqq \tilde Y_{X^n_0} -X_0^n$ with support on $\gamma_n\Z^2-X_0$ s.t. 
$
  \EE Y_{X^n_0} = \mu(X_0^n) $ and $\VV Y_{X^n_0}= \Sigma(X_0^n).
$
To build $\tilde Y_{X^n_0}$ using Theorem~\ref{th:t_d=2}, we need to bound the eigenvalues of $\Sigma(X_0^n)+ (\mu(X_0^n)+X_0^n)^{\otimes2}$ using that 
\begin{align}
\lambda_{\min}[\Sigma(X_{0}^{n})+(\mu(X_{0}^{n})+X_{0}^{n})^{\otimes2}]=&\inf_{d\not=0}\frac{d^{T}[\Sigma(X_{0}^{n})+(\mu(X_{0}^{n})+X_{0}^{n})^{\otimes2}]d}{d^{t}d}\\\geq&\inf_{d\not=0}\frac{d^{T}\Sigma(X_{0}^{n})d}{d^{T}d}+\inf_{e\not=0}\frac{e^{T}(\mu(X_{0}^{n})+X_{0}^{n})^{\otimes2}e}{e^{T}e}\\\geq&\lambda_{\min}[\Sigma(X_{0}^{n})]+\lambda_{\min}[(\mu(X_{0}^{n})+X_{0}^{n})^{\otimes2}]\\\geq&\lambda_{\min}[\Sigma(X_{0}^{n})],
\end{align}
noting that $\lambda_{\min}[(\mu(X_{0}^{n})+X_{0}^{n})^{\otimes2}]\ge0$, which concludes the proof.
\end{proof}

\section{Lattice-Tree Models for Multi-Dimensional Diffusions}\label{sec: Models for Multi-Dimensional Diffusions}
We now carry out the same procedure in the multi-dimensional case. In this case, our assumptions will be more stringent and we will consider the lattice to depend on the structure of the matrix $\vB$. 
\subsection*{Matching the Moments}%
In this Subsection we solve System~\eqref{eq:fundamental_system} for a general dimension $d$, but on pre-specified lattice. 
We call $Q \in O(d)$ the orthonormal matrix such that $\vB=Q\Lambda Q^T$, with $\Lambda$ the diagonal matrix consisting of the eigenvalues of $\vB$. 

\begin{theorem}\label{th:t_d=d}
  Let $\vB \in \R^{d\times d}$ symmetric and positive {semi-}definite, $\gamma>0$, %
  then there exists a random variable $Y$ such that
  \begin{enumerate}
  \item $\E[Y]=0$ and $\EE [Y^{\otimes 2}]=\vB$;  
  \item $\operatorname{supp}(Y) \subset \{ y \in \gamma Q \Z^d : |y|_\infty \le \sqrt{ \sum_i \lambda_i} \cdot \max_i | Q e_i |+\gamma\}$. 
  \end{enumerate}
\end{theorem}
\begin{proof}%
If all the eigenvalues $\lambda_i$ of $\vB$ are 0, then $Y=0$ a.s. is a solution, so let us suppose that at least one eigenvalue is strictly positive.

Recall that, from the proof of Theorem~\ref{th:t_d=2}, if $\bar Y$ is a discrete random variable that matches the second moment, $\EE \bar Y^{\otimes 2}=\vB$ and $Z$ is a Bernoulli random variable independent of $\bar Y $, 
$\Prob(Z=\pm 1)=\frac{1}{2}$,
then the random variable $Y\coloneqq Z \cdot \bar Y$ has mean $\EE[Y]=0$ and $\EE[Y^{\otimes 2}]=\vB$. 
  Hence, solving the system~\eqref{eq:fundamental_system} reduces to find 
$\bar{l}_i\in\gamma Q\Z^d$ s.t. 
\[
\sum_{i}\bar p_{i}\bar{l}_{i} \bar{l}_{i}^T =\vB,\qquad
\bar p_{i}>0,\,\sum_{i}\bar p_{i}=1, 
\]
then the measure on the points $\bar{l}_i$, $-\bar{l}_i$ with weights respectively $p_{i}=\bar p_i/2$ represents the r.v. $Y$ we are looking for. 
Since we suppose $\vB$ to be positive semi-definite and symmetric, there exists an orthonormal matrix $Q$ such that 
$\vB=Q\Lambda Q^{T}$, whose columns are the eigenvectors of $\vB$, and $\Lambda$ is a diagonal matrix whose values are the non-negative eigenvalues. 
If we indicate $\bar l_i=\gamma Q \bar z_i$, $\bar z_i\in\Z^d$, we have that 
\[
\sum_{i}^{r}\bar{p}_{i}\bar{l}_{i}\bar{l}_{i}^{T}=\vB=Q\Lambda Q^{T}\Leftrightarrow\sum_{i}^{r}\bar{p}_{i}\gamma^{2}Q\bar{z}_{i}\bar{z}_{i}^{T}Q^{T}=Q\Lambda Q^{T}\Leftrightarrow\sum_{i}^{r}\bar{p}_{i}\gamma^{2}\bar{z}_{i}\bar{z}_{i}^{T}=\Lambda.
\]
Therefore, rearranging the last equation we obtain
\begin{align}
\sum_{i}^{r}p_{i}\gamma^{2}\bar{z}_{i}\bar{z}_{i}^{T}=\sum_{j=1}^{d}\frac{\lambda_{j}}{\sum_{k=1}^{d}\lambda_{k}}\left(\sqrt{\sum_{k=1}^{d}\lambda_{k}}\cdot e_{j}\right)^{\otimes2}.
\end{align}
Since $r$ is a free parameter for us, we can take %
$r=d^*+1$, where $d^*$ is the number of different eigenvalues not equal to $0$, and we can solve the problem independently for any eigenvalue, then we merge together all the solutions. Thus, we obtain 
\begin{align}
\bar{l}_{d^{*}+1}=&\boldsymbol{0},&\bar{p}_{d^{*}+1}=&1-\left\{ \frac{\sum_{j=1}^{d}\lambda_{j}}{\left\lceil \sqrt{\sum_{j=1}^{d}\lambda_{j}}\right\rceil _{\gamma}^{2}}\right\}, \\\bar{l}_{i}=&Q\left\lceil \sqrt{\sum_{j=1}^{d}\lambda_{j}}\right\rceil _{\gamma}e_{i},&\bar{p}_{i}=&\frac{m(\lambda_{i})\lambda_{i}}{\sum_{j=1}^{d}\lambda_{j}}\left\{ \frac{\sum_{j=1}^{d}\lambda_{j}}{\left\lceil \sqrt{\sum_{j=1}^{d}\lambda_{j}}\right\rceil _{\gamma}^{2}}\right\},  
\end{align}
where  $m(\lambda_i)$ is the algebraic multiplicity of the eigenvalue $\lambda_i$. 
At this point we have a probability measure comprising the $d^*+1$ weights $\bar{p}_i$ and the points $\bar{l}_1, ..., \bar{l}_{d^*}, \boldsymbol{0}\in \gamma Q\Z^d$ and we refer to this random variable by $\bar Y$. 
We can now build the r.v. $Y$ taking the $2d^*+1$ points $\{\bar l_i, - \bar l_i\}$ 
with probability $p_i=\bar p_i/2$, 
such that $\EE  Y = 0$ and $\EE  Y^{\otimes 2} = \vB$. 
\end{proof}

\subsection*{Constructing the Lattice-Tree Model}%
Although we have not proved Theorem~\ref{th:mm_d=2} for general dimension $d$, we can still prove one particular case. We are ready to build the Markov chain $X^n$ which approximates $X$ in the multi-dimensional case under more stringent conditions than before. 
\begin{theorem}\label{th:mm_d=d}
  Let $\mu(x)=0, \sigma \in C^4_b$, $d\ge2$, %
  $\Sigma(x) = Q\Lambda(x)Q^T$, 
  $\Lambda(x)\in\R^{d\times d}$ a diagonal matrix and 
  $Q$ an orthogonal matrix. 
  For any $c>0$, set $\gamma_n= c n^{-1/2}$ and let us suppose that $X_0^n\in\gamma_nQ\Z^d$. 
   There exists a lattice approximation $(X^n)_{n \ge 1}$ to $X$ on $\gamma_nQ\Z^d$ respect to the class of functions $f\in C^3_P$ with rate
   {$ 1/2$}. 
  Moreover, 
  \begin{align}
\begin{cases}
\operatorname{card}[\operatorname{supp}(X_{i}^{n})]\leq i^{d}c.& ,\text{ if \ensuremath{\mu,\sigma} are bounded,}\\
\operatorname{card}[\operatorname{supp}(X_{i}^{n})]\leq ci^{d}{\exp(dcin^{-1/2})} & ,\text{ f \ensuremath{\mu,\sigma} have linear growth.}
\end{cases}
    \end{align}
\end{theorem}
\begin{proof}
The proof follows the same reasoning used in Theorem~\ref{th:mm_d=1} and~\ref{th:mm_d=2}, 
but using Theorem~\ref{th:t_d=d}. 
The bounds on the cardinality of the support can be obtained as explained in the proof of Theorem~\ref{th:mm_d=2}
\end{proof}
\noindent It is relevant to recall that to obtain $\mu(x)=0$ the Girsanov Theorem can be helpful. 
\section{Variations and Extensions}\label{sec:prev_works}
We briefly discuss variations and possible extensions.

\subsection{Combining State-Space Transformations and Recombination}\label{sec:Transformations}
The classical transformation approach and the recombination approach we present here are not mutually exclusive and can be combined in way that leverages their strengths. 
To demonstrate this, we revisit the problem of approximating an SDE with drift and volatility that are not bounded but have linear growth.
Our Theorem~\ref{th:mm_d=1} guarantees under ellipticity assumptions on $\sigma$, that 
\begin{align}\label{eq:over-linear_growth}
  \begin{cases}
    \operatorname{card}\,[\,\operatorname{support}(X_{i}^{n})\,]\le c\left(i\times {e^{n^{-1/2}i}}\right), & \text{ if \ensuremath{\mu} and \ensuremath{\sigma^{2}} have linear growth,}\\
    \operatorname{card}\,[\,\operatorname{support}(X_{i}^{n})\,]\leq ci, & \text{ if \ensuremath{\mu} and \ensuremath{\sigma^2} are bounded,}
  \end{cases}
\end{align}
which in the case of linear growth leads to exponential growth in time.  
However, inspired by \citet{Nelson1990},we can look for a state-space transformation to first turn the SDE into an SDE with bounded vector fields and subsequently apply Theorem \ref{th:mm_d=1}, but now for the case of bounded vector fields.
Formally, this means that we look for a state-space transformation given by smooth function $f$ such that $Y_t=f(X_t)$, where $X_t$ is the solution of the SDE~\eqref{eq:SDE_ch6sec1}, with bounded $\mu, \sigma$.
Thus, using the Ito's lemma, $Y$ solves
\begin{align}\label{eq:gen_from_Y}
\dif Y_{t}=&\left[\mu(X_{t})\partial f(X_{t})+\frac{\sigma(X_{t})^{2}}{2}\partial^{2}f(X_{t})\right]\dif t+\sigma(X_{t})\partial f(X_{t})\dif W_{t}.
\end{align}
Going back to Equation~\eqref{eq:gen_from_Y} we can see that\footnote{We have to suppose $ \EE_{t}\int_{t}^{t+n^{-1}}\sigma(X_{s})\partial f(X_{s})dW_{s}=0$, which is for example satisfied when the integrand is (locally) square integrable.}
\begin{align}
\EE_{t}[Y_{t+n^{-1}}-Y_{t}]=&\left[\mu(X_{t})\partial f(X_{t})+\frac{\sigma(X_{t})^{2}}{2}\partial^{2}f(X_{t})\right]n^{-1}+O(n^{-2}),\\
\EE_{t}[Y_{t+n^{-1}}-Y_{t}]^{2}=&\left[\sigma(X_{t})\partial f(X_{t})\right]^{2}n^{-1}+O(n^{-2}).
\end{align}
Now, if we suppose to have built the Markov chain approximation $X_i^n$ as in Theorem~\ref{th:mm_d=1}, then we want to see if $Y_i^n=f(X^n_i)$ satisfy Lemma~\ref{lemma: convergence assumptions}-Item~\eqref{it:my_local_consistency}, respect to $Y$. 
Using a Taylor expansion we get %
\begin{align}
\EE_{i}Y_{i+1}^{n}-Y_{i}^{n}=&\EE_{i}\left[f(X_{i+1}^{n})-f(X_{i}^{n})\right]\\=&\EE_{i}\left[\partial f(X_{i}^{n})\Delta X_{i}^{n}+\frac{1}{2}\partial^{2}f(X_{i}^{n})(\Delta X_{i}^{n})^{2}+O(|\Delta X_{i}^{n}|^{3})\right]\\=&\partial f(X_{i}^{n})\EE_{i}\left[\Delta X_{i}^{n}\right]+\frac{1}{2}\partial^{2}f(X_{i}^{n})\EE_{i}\left[\Delta X_{i}^{n}\right]^{2}+O(n^{-3/2}),
\end{align}
where $\Delta X_{i}^{n}=X_{i+1}^{n}-X_{i}^{n}$ and by construction $\EE_{i}O(|\Delta X_{i}^{n}|^{3})\leq c \sum_{j}p_{j}\left|n^{-1/2}z_{j}\right|^{3}\leq c n^{-3/2}$, $z_j\in\Z$. 
Given that $X^n_i$ satisfies Lemma~\ref{lemma: convergence assumptions}-Item~\eqref{it:my_local_consistency} respect to $X$, 
we have that $\EE_{i}[Y_{i+1}^{n}-Y_{i}^{n}]=\EE_{t_i}[Y_{t+n^{-1}}-Y_{t}]+O(n^{-3/2})$. 
For the second moment we proceed similarly
\begin{align}
\EE_{i}\left[Y_{i+1}^{n}-Y_{i}^{n}\right]^{2}=&\EE_{i}\left[f(X_{i+1}^{n})-f(X_{i}^{n})\right]^{2}\\=&\EE_{t}\left[\partial f(X_{i}^{n})\Delta X_{i}^{n}+O(|\Delta X_{i}^{n}|^{2})\right]^{2}\\=&\partial f(X_{i}^{n})^{2}\EE_{t}\left[\Delta X_{i}^{n}\right]^{2}\!+\!2\partial f(X_{i}^{n})\EE_{t}\left[\Delta X_{i}^{n}\!\times\!O(|\Delta X_{i}^{n}|^{2})\right]\!+\!\EE_{t}O(|\Delta X_{i}^{n}|^{4}),
\end{align}
note that $|\EE_{t}[\Delta X_{i}^{n}\times O(|\Delta X_{i}^{n}|^{2})]|\leq c n^{-3/2}$. 
Since 
\[
\EE_{i}\left[\Delta X_{i}^{n}\right]^{2}=\mu(X_{i}^{n})^{2}n^{-2}+\sigma(X_{i}^{n})^{2}n^{-1}+O(n^{-3/2})=\sigma(X_{i}^{n})^{2}n^{-1}+O(n^{-2}),
\]
we can conclude that $\EE_{i}[Y_{i+1}^{n}-Y_{i}^{n}]^{2}=\EE_{t_{i}}[ Y_{t_{i}+n^{-1}}-Y_{t_{i}}]^{2}+O(n^{-3/2})$. 
In the same way we obtain $ \EE_{i}|Y_{i+1}^{n}-Y_{i}^{n}|^{3}=O(n^{-3/2})$, 
therefore we can conclude that $Y^n$ satisfies Lemma~\ref{lemma: convergence assumptions}-Item~\eqref{it:my_local_consistency} respect to $Y$ for $\alpha=1/2$.
Already the example of $Y$ a Geometric Brownian Motion is interesting where $f(x)=e^x$ yields that 
$X$ is a Brownian Motion with a bounded mean and variance, i.e.
\[
  \frac{\dif Y_{t}}{Y_{t}}=\mu \dif t+\sigma \dif W_{t},\quad\quad \dif X_{t}=\left[\mu-\frac{\sigma^{2}}{2}\right]\dif t+\sigma \dif W_{t}.
\]
The state space is not any more a lattice a priori, indeed it is a non linear transformation of a lattice. 
To sum up, this simple variation allows to treat linear growth vector fields by a MC.

\subsection{Optimal Control and Optimal Stopping }\label{sec:Control and Optimal Stopping}
In Section~\ref{sec:Control and Optimal Stopping} we cited optimal stopping as one of the motivations to consider MC approximations.  
However, in general the topology of weak convergence is too coarse to provide robustness of the discretizaton procedures; for example, the solution of the optimal stopping problem for the MC does in general not converge to the solution of the stopping problem for the SDE; see \cite{hoover1984adapted}.  
Nevertheless, we can mimic the arguments in \cite{Kushner1990,Kushner2001} to show that solving the optimal stopping for MC produces by $\mathbf{DISCRETIZE}_n$ yields solutions that are close to the continuous time optimal stopping problem.  
\begin{theorem}(\cite[Chapter 10, Theorem 6.2]{Kushner2001})
Let $X$ be the solution of the SDE~\eqref{eq:SDE_ch6sec1} and $X^n$ be one of the schemes built in Theorems~\ref{th:mm_d=1}, \ref{th:mm_d=2}, \ref{th:mm_d=d}. 
Denote with $\mathcal{T}$ and $\mathcal{T}^n$ the set of almost sure finite stopping times with respect to the natural filtrations of the SDE~\eqref{eq:SDE_ch6sec1} resp.~the filtration generated by $X^n$.
Define the value functions
\begin{align}
V\!(x)\!\!:=\!\min_{\tau\in\mathcal{T}}\EE\!\left[\int_{0}^{\,\tau\wedge1}\!\!\!\!k(X_{s})ds\!+\!g(X_{\tau\wedge1})\Big|X_{0}\!=\!x\right]\text{ and }V^{n}(x)\!\!:=\!\!\!\min_{\tau^{n}\in\mathcal{T}^{n}}\!\EE\!\left[\frac{1}{n}\!\sum_{i=0}^{\tau^{n}\wedge n}\!\!k(X_{i}^{n})\!+\!g(X_{\tau^{n}\wedge n})\Big|X_{0}^{n}\!=\!x\right]\!.
\end{align}
If $\mu, \sigma, g, k$ are bounded Lipschitz continuous\footnote{Other rather general and technical assumptions on $V$ are needed.}, then 
$
V^{n}(x)\to V(x)\text{, as }n\to\infty.
$
\end{theorem}
\begin{proof}
Thanks to \cite[Chapter 10, Theorem 6.2]{Kushner2001}, it is enough to check that the conditions in \cite[Equation (1.3), page 71]{Kushner2001} are satisfied for the approximation schemes of Theorems~\ref{th:mm_d=1}, \ref{th:mm_d=2}, \ref{th:mm_d=d}.
\cite[Equation (1.3), page 71]{Kushner2001} has two different types of conditions: 
the first type of conditions is equivalent to \autoref{lemma: convergence assumptions}-\eqref{it:my_local_consistency}; 
the second type of conditions requires that a.s. $\sup_{i} |X^n_{i+1} - X^n_{i}|\to 0 ,$ as $n\to\infty$, but this is a consequence of the fact that the r.v. $Y_{(\cdot)}$ built in the proof of Theorems~\ref{th:mm_d=1}, \ref{th:mm_d=2}, \ref{th:mm_d=d} are bounded by the bounds of Theorems~\ref{th:t_d=1}, \ref{th:t_d=2}, \ref{th:t_d=d}.
\end{proof}
For brevity we do not pursue this further here, but note that the same logic extends to more general control problems, at least in principle: one needs to introduce controls in the coefficients $\mu$ and $\sigma$ and define the related value functions $V, V^n$; see \cite{Kushner1990,Kushner2001}. 
In this case, one would need to construct the random variable  $Y_{(\cdot)}$ for the different controls to allow for effective computation.

\subsection{Extensions.}
\begin{description}
\item[Hypo-elliptic SDEs.] 
  If we do not know a priori if $\Sigma$ is elliptic, we can take a $\gamma>0$ reasonably small and then start building the random variables $Y_{(\cdot)}$ for the lattice $\gamma\Z^d$. {If we arrive in a point $x$ where $\sqrt{\lambda_{\min}[\Sigma(x)]}\leq \gamma$,} then it is enough to add to the lattice $\gamma \Z^d$ some points in the lattice 
  $\mu(x)+\gamma Q\Z^d$, where $Q$ depends as usual from $\Sigma(x)$. 

\item[Switching SDEs]
Switching diffusions can be treated in a straightforward manner. Suppose that the functions $\Sigma$ and $ \mu$ depend on one (or more) independent dynamic process $\theta(t)$, the ``switch'', in a Markovian way. 
For simplicity, let us suppose that $\theta(t)\in\{0,1\}$. We consider four functions $\Sigma_0, \mu_0$ and $\Sigma_1, \mu_1$:
\begin{align}
\EE_{t}[X_{t+n^{-1}}-X_t]\approx\begin{cases}
\mu_{1}(X_{t}) & \text{if }\theta_{t}=1\\
\mu_{0}(X_{t}) & \text{if }\theta_{t}=0
\end{cases},&\quad\quad
\VV_{t}[X_{t+n^{-1}}-X_t]\approx\begin{cases}
\Sigma_{1}(X_{t}) & \text{if }\theta_{t}=1\\
\Sigma_{0}(X_{t}) & \text{if }\theta_{t}=0
\end{cases}. 
\end{align}
Thus, we can approximate two different processes defined by $\Sigma_0, \mu_0$ and $\Sigma_1, \mu_1$ independently, assuming to know or to be able to build a ``good'' discrete time approximation $\theta^n$ of $\theta(t)$. %

\item[Time-inhomogenuous SDEs.]
  We have presented the case of time-homogeneous SDEs, but the convergence results in~\cite{Kloeden1992} apply to time-inhomogeneous SDEs, hence the same approach extends to time-inhomogenuous SDEs.
  Nevertheless, from a computational perspective, time-inhomogeneous SDEs are more challenging since if lattice points are re-visited then the recombination computation has to be redone because time has increased since the last visit.  
  In principle, it is also possible to go beyond SDEs, but the bottleneck there is to find suitable conditions to replace those given in Definition~\ref{def:local_consistency} to guarantee the weak convergence.
  \end{description}

\section{Algorithm and Experiments}\label{sec:Algorithm and Experiments}
First, let us state the pseudo-code of the algorithm for the construction of the Transition Matrix. We suppose for simplicity $X_0=0$, so that  $X^n_0\in\gamma_nQ\Z^d$ for any $n,$ $Q$.
\begin{varalgorithm}{$\mathbf{DISCRETIZE}$}\caption{Lattice-Tree Markov Chain Approximation}\label{algo:tree_recomb}
\textbf{Inputs}: $\mu$, $\sigma$, $x_0$, $n$ %
\begin{algorithmic}[1]
	\State{Choose $\gamma_n$ and $Q$ as prescribed by Theorem~\ref{th:mm_d=1}, \ref{th:mm_d=2}, \ref{th:mm_d=d}}
	\State{queue\_todo.append$[x_0]$}
	\State{{$i\gets 1$}}
	\While{$i\le n$}
	\While{queue\_todo is not empty}
	\State{$x \gets $queue\_todo$[0]$}
	\State{Solve the non-linear moment matching system~\eqref{eq:fundamental_system} to get
    $(p_j,l_j)_j \subset \R^+ \times \gamma_nQ\Z^d$}\label{step:solve_sytem} 
	\State{Transition Matrix $M_n\gets \{ x$ can go to $\{x+l_j\}$ with probabilities $\{p_j\}$ $\}$}
	\State{queue\_todo\_nextstep.append$[$ for any $  j,$ $x+l_j$ such that $ x+l_j\not \in $  queue\_done $\cup$ queue\_todo\_nextstep $]$ }
	\State{queue\_todo.remove[x]}
	\State{queue\_done.append[x]}
	\EndWhile
	\State{queue\_todo$\gets$queue\_todo\_nextstep}
	\State{queue\_todo\_nextstep.remove[all elements]}
	\State{{$i\gets i+1$}}
	\EndWhile
	\State{\textbf{Return} Transition Matrix $M_n$ and starting value $x_n$}
\end{algorithmic}
\end{varalgorithm}
Remember that, if $X^n_0$ is not in the lattice, $X^n_1$ must be built directly to be in the considered lattice, see the proof of Theorem~\ref{th:mm_d=1}. 

\paragraph{Step \ref{step:solve_sytem}: Solving a non-linear Moment System.} 
Step \ref{step:solve_sytem} is the computationally most challenging step of Algorithm~\ref{algo:tree_recomb} since it requires to solve the non-linear lattice-constrained system \eqref{eq:fundamental_system}.   
We now show that previous work on recombination allows to do this step efficiently.
Since our experiments are in dimension $d=2$, we provide the details for the two-dimensional case, but the same approach works for general dimension $d$ under minor modifications. 
Assuming the assumptions of Theorem~\ref{th:t_d=2},~\ref{th:mm_d=2} are met, then one can replace both $\approx$ signs in \eqref{eq:fundamental_system} by $=$ and solve the resulting system by either 
\begin{enumerate}[(i)]
\item\label{itm:solve by theorem} following the proof of Theorem~\ref{th:t_d=2}, or
\item\label{itm:Randomized Recombination} %
  by using the randomized recombination Algorithm 2 in \citet{Cosentino2020}.
\end{enumerate}
The approach \eqref{itm:solve by theorem} has two disadvantages: firstly the number $r$ of weight-point tuples $(w_j,l_j)_{j=1,\ldots,r}$ it returns is in general larger than needed whereas \eqref{itm:Randomized Recombination} returns the smallest number of tuples that is guaranteed by Carath\'eodory's theorem; secondly, and more importantly in practice, is that approach~\eqref{itm:solve by theorem} is slower than the optimized recombination Algorithm 2 used in approach~\eqref{itm:Randomized Recombination}.
\begin{remark}
There are other recombination algorithms \cite{maria2016a,Maalouf2019,Litterer2012,hayakawa2021estimating,hayakawa2021positively,Hayakawa2020MonteCarloCC} besides Algorithm 2 from\cite{Cosentino2020}. %
However, what makes Algorithm 2 especially attractive in the current situation is that it can take as input a sequence of points -- in our case, points sampled uniformly on the lattice -- and the moments that need to be matched.
This is in contrast to all the other mentioned algorithms, which need as input a discrete measure, that is a collection of \emph{points and weights}.  
Although subtle, this is an important difference that adds to the efficiency of the resulting algorithm.
\end{remark}
If the assumptions of Theorem~\ref{th:mm_d=2} are not met, one could still follow the approach ``in spirit'', by trying to construct a discrete probability measure with a small support that approximates the first and second moments.
Formally, this means we 
\begin{enumerate}[(i)]
  \setcounter{enumi}{2}
\item\label{itm:solve by constrained minimization} try to solve the constrained minimization problem 
  \begin{align}\label{eq:minim}
    (w_{j},l_{j})_j:=&\argmin\left|\sum_{j}\left(\begin{array}{c}
                                                   z_{j}\cdot w_{j}\\
                                                   z_{j}^{\otimes2}\cdot w_{j}
                                                 \end{array}\right)-\left(\begin{array}{c}
                                                                            n^{-1}\mu(x)\\
                                                                            n^{-2}\mu(x)^{\otimes2}+n^{-1}\Sigma(x)
                                                                          \end{array}\right)\right|^{2}\\&\quad\quad\text{s.t.}\quad l_{j}\in\gamma\Z^{2},\quad|l_{j}|\leq c\\&\quad\quad\quad\quad\,0\leq w_{j}\leq1,\quad\sum_{j}w_{j}=1, 
  \end{align}
  and then reduce the discrete probability measure obtained by \eqref{eq:minim} further by using any of the recombination algorithms in \cite{maria2016a,Maalouf2019,Litterer2012,hayakawa2021estimating,hayakawa2021positively,Hayakawa2020MonteCarloCC}. 
\end{enumerate}
We re-emphasize that in the case of \eqref{itm:solve by constrained minimization} 
none of the theoretical guarantees of the previous sections apply; in fact, it is not even guaranteed that  the minimum of the optimization problem~\eqref{eq:minim} is small enough.
Nevertheless, highly optimized optimization software exists for such constrained problems and the intuition remains that the resulting MC is a good approximation in view of its construction.  
In fact, our numerical example below suggests that this procedure is quite robust. 

We have implemented the above algorithms and below we apply it two concrete problems: Section \ref{sec:toy model} applies $\mathbf{DISCRETIZE}$ to a two-dimensional SDE and Section \ref{sec:Mean reverting Heston model} studies the robustness of $\mathbf{DISCRETIZE}$ by applying it to a Heston model where the assumptions of Theorem~\ref{th:mm_d=2} are not met and we use case \eqref{itm:solve by constrained minimization} above to solve a constrained minimization problem. 
Python code to replicate the two experiments in the next Subsections is publicly available\footnote{\href{https://github.com/FraCose/Tree-Recombination}{https://github.com/FraCose/Tree-Recombination}}; see also Section~\ref{app:details}. 

\subsection{Example: toy model}\label{sec:toy model}
We consider the following toy model
\begin{align}
\dif\left(\begin{array}{c}
X\\
Y
\end{array}\right)=&\left(\begin{array}{c}
\sin(X)\\
\cos(Y)
\end{array}\right)\dif t+\left(\begin{array}{cc}
\cos(Y)+2 & 0\\
0 & \sin(X)+2
\end{array}\right)\dif W_{t},\\X_{0}=Y_{0}=&0.
\end{align} 

\begin{figure}[hbt!]
    \centering
        \includegraphics[height=4cm,width=0.32\textwidth]{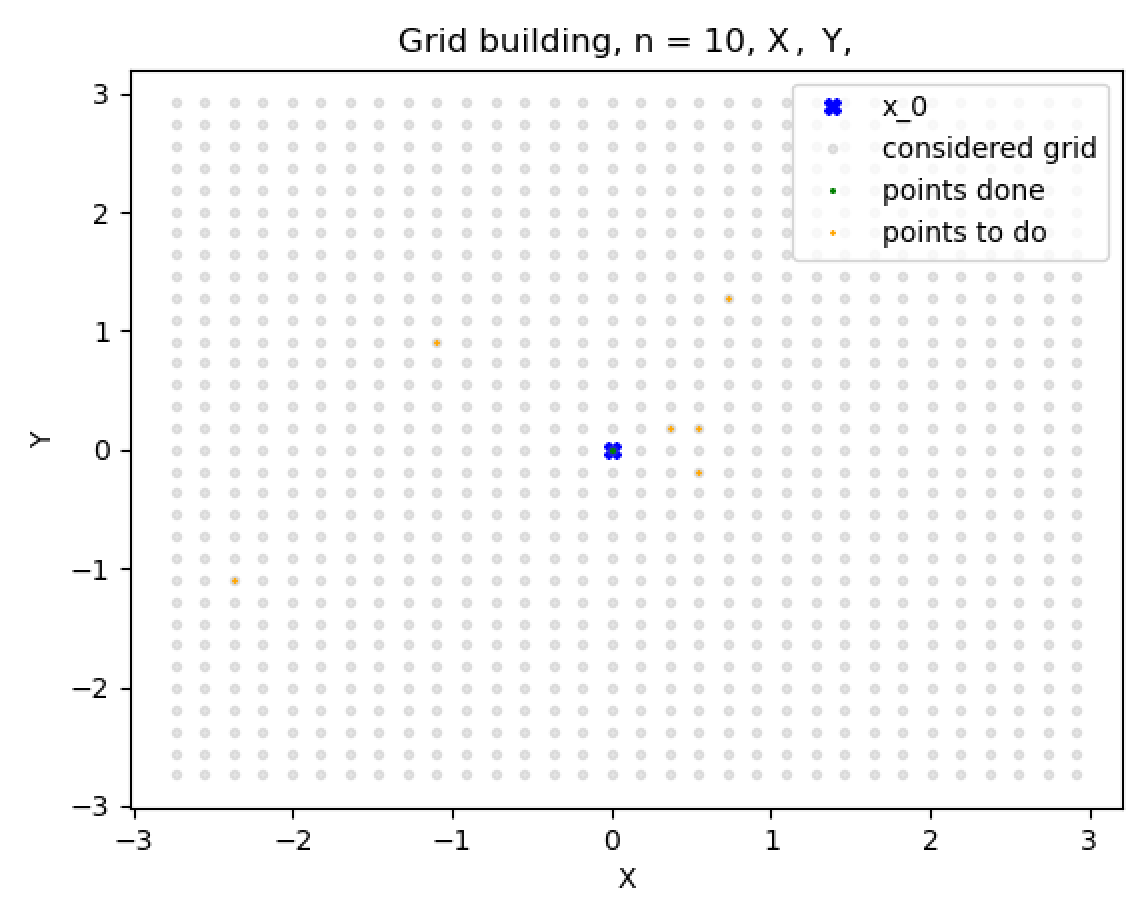}
        \includegraphics[height=4cm,width=0.32\textwidth]{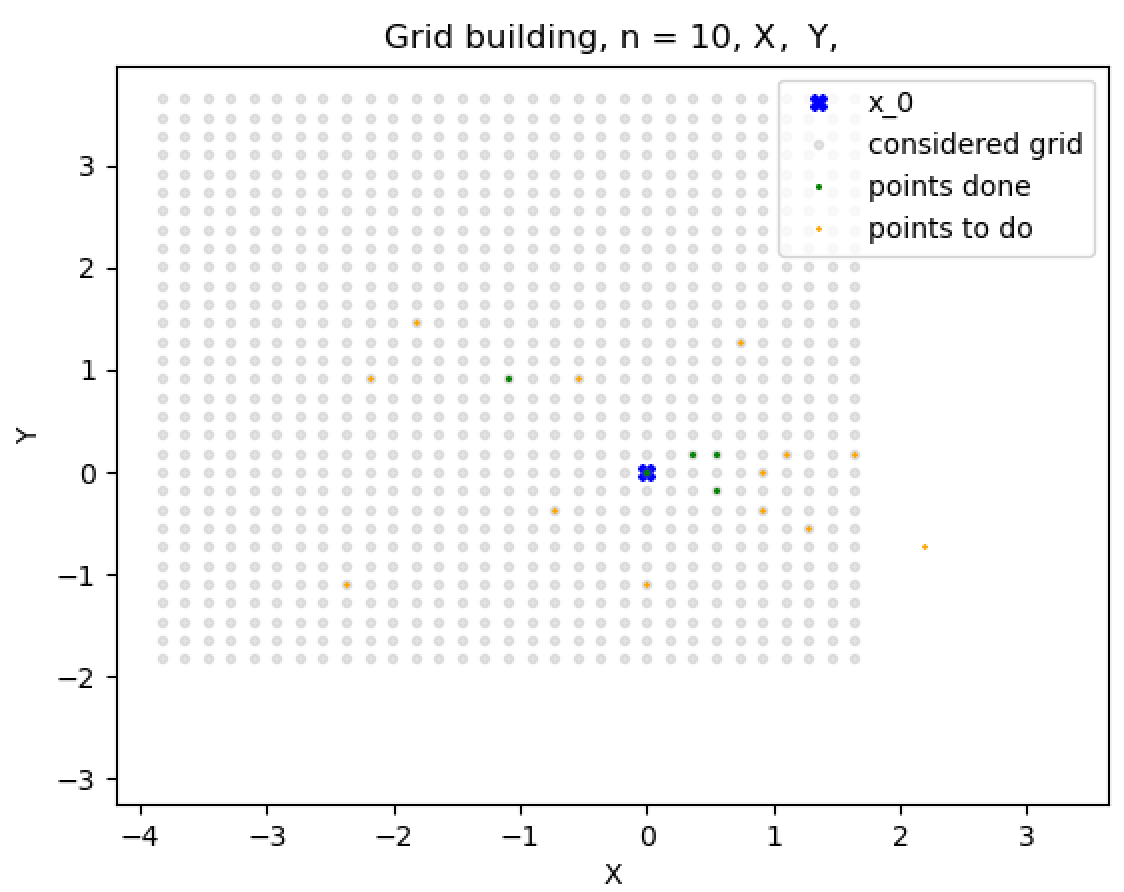}
        \includegraphics[height=4cm,width=0.32\textwidth]{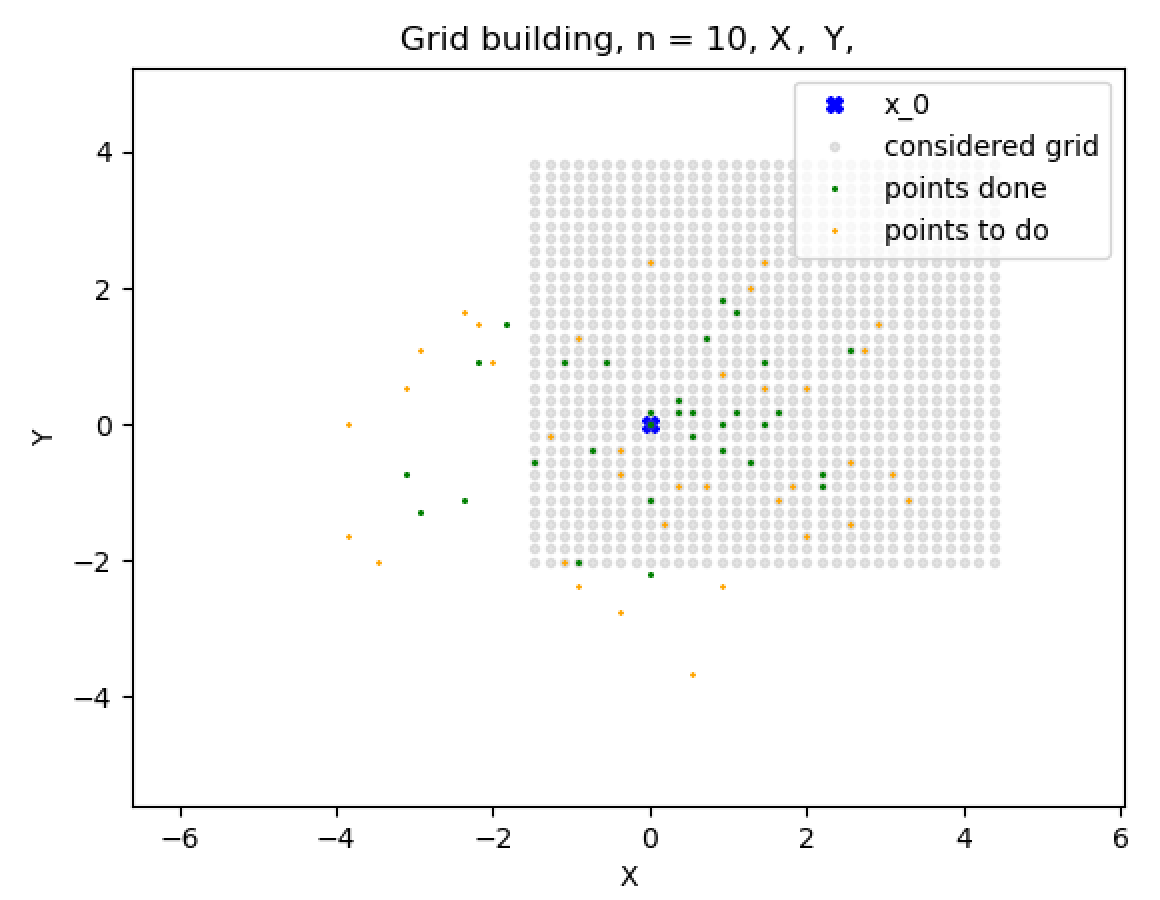}
        
        \includegraphics[height=4cm,width=0.32\textwidth]{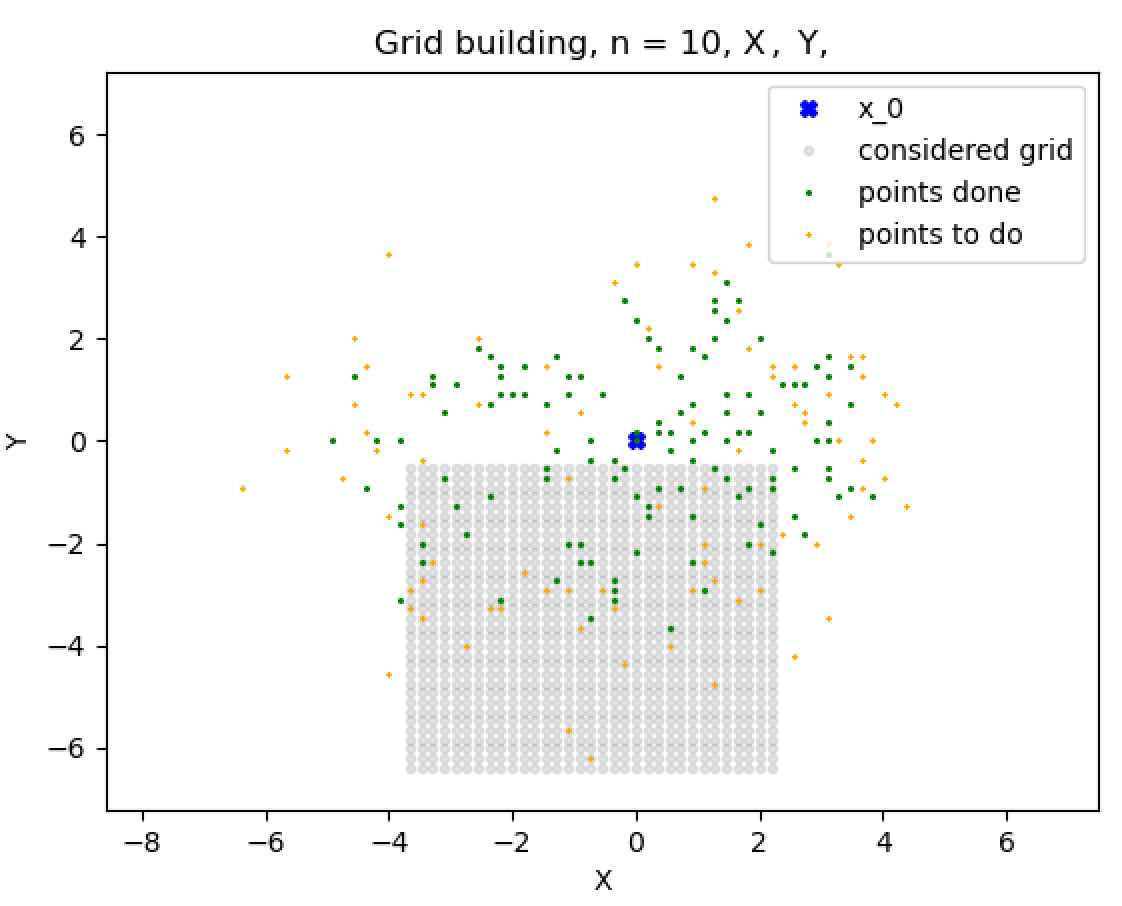}
        \includegraphics[height=4cm,width=0.32\textwidth]{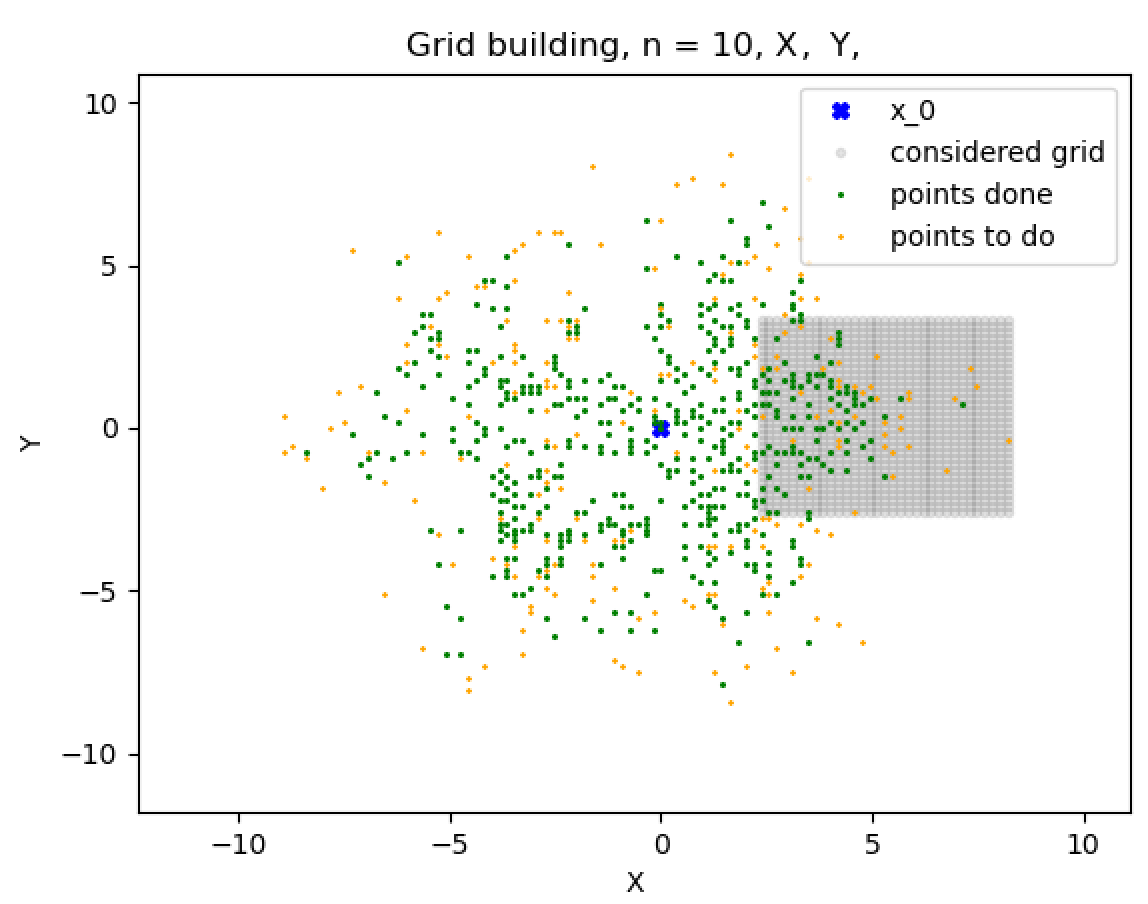}
        \includegraphics[height=4cm,width=0.32\textwidth]{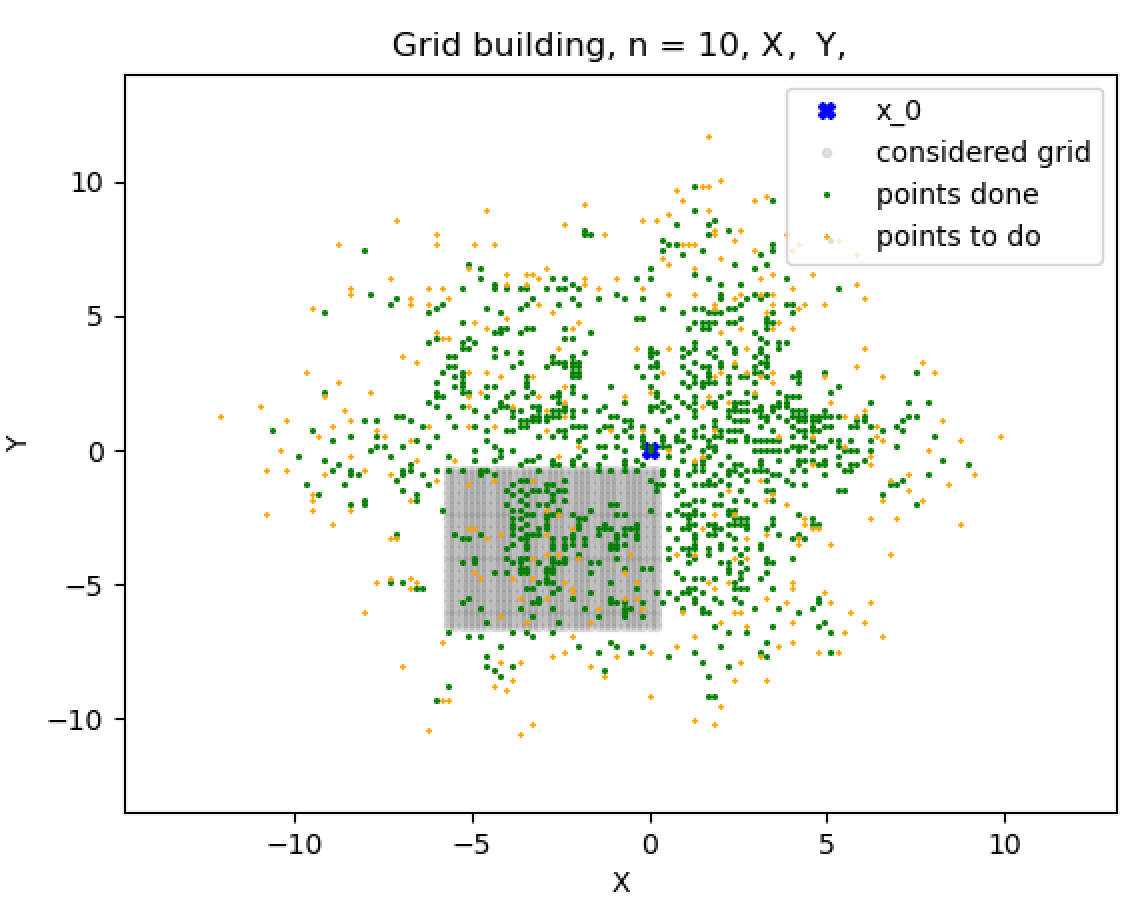}
    \caption[Construction of the transition matrix. ]{
    Construction of the transition matrix. 
    Starting from $x_0$ we consider to build the random variable $X_1$, such that it has support on the grey points and the first two moments are matched exactly. 
    In the top-left plot, $X_1$ has support on the yellow points. 
    Then this procedure is iterated for all the possible points in the support of the measure, following Algorithm~\ref{algo:tree_recomb}. 
    }
        \label{fig:transition_matrix_toy}
\end{figure}

\begin{figure}[hbt!]
    \centering
        \includegraphics[height=4cm,width=0.32\textwidth]{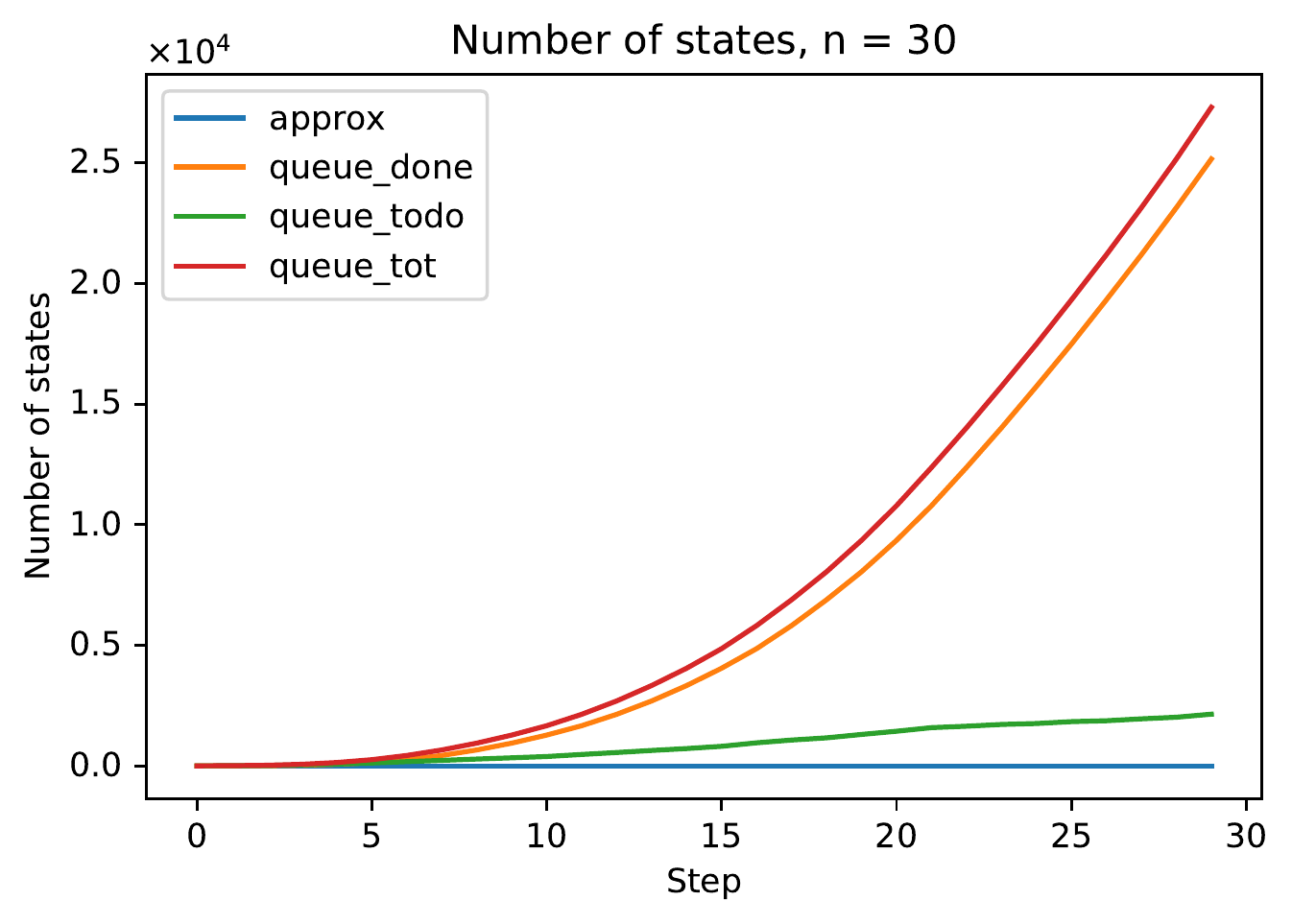}
        \includegraphics[height=4cm,width=0.32\textwidth]{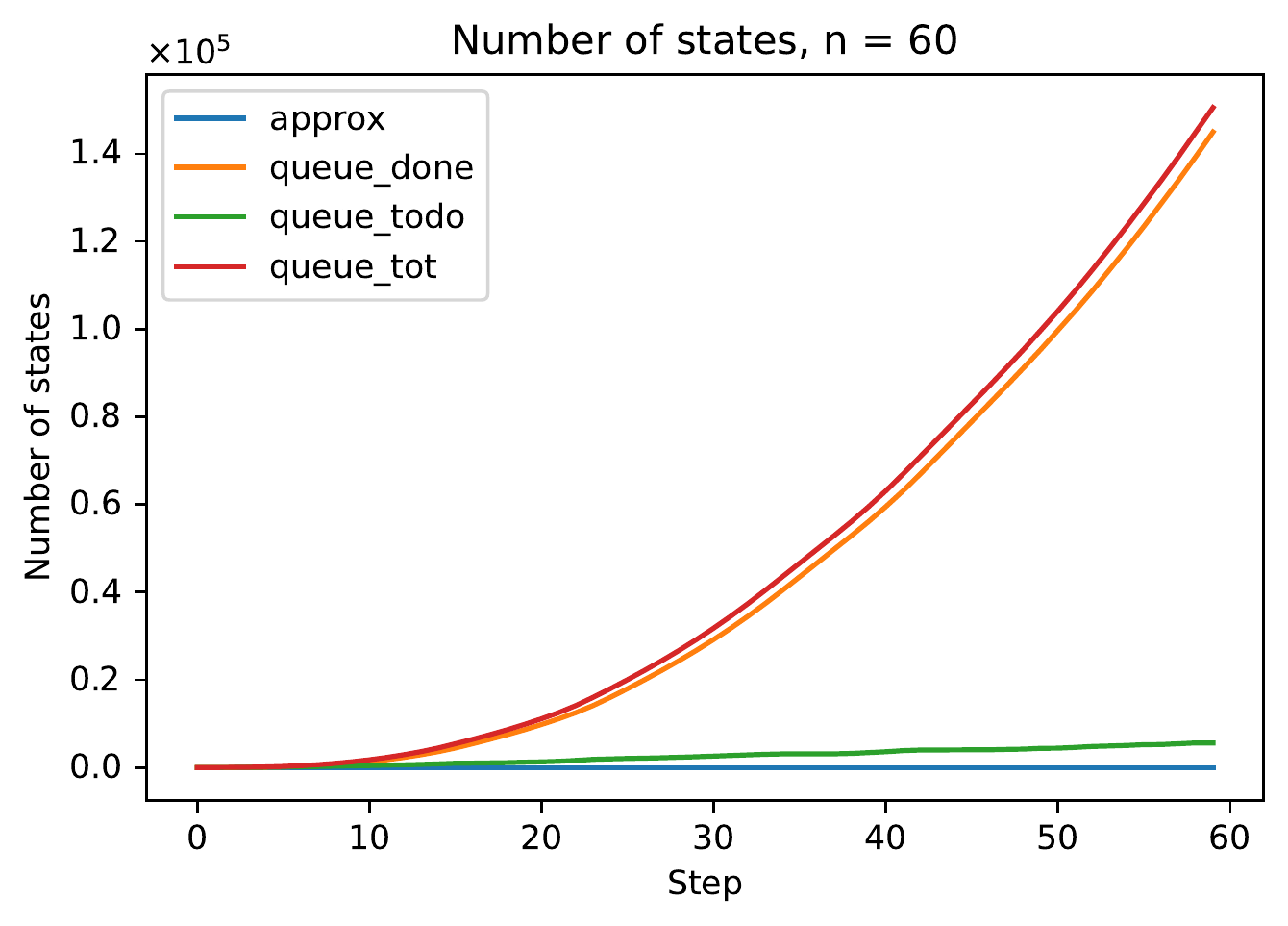}
        \includegraphics[height=4cm,width=0.32\textwidth]{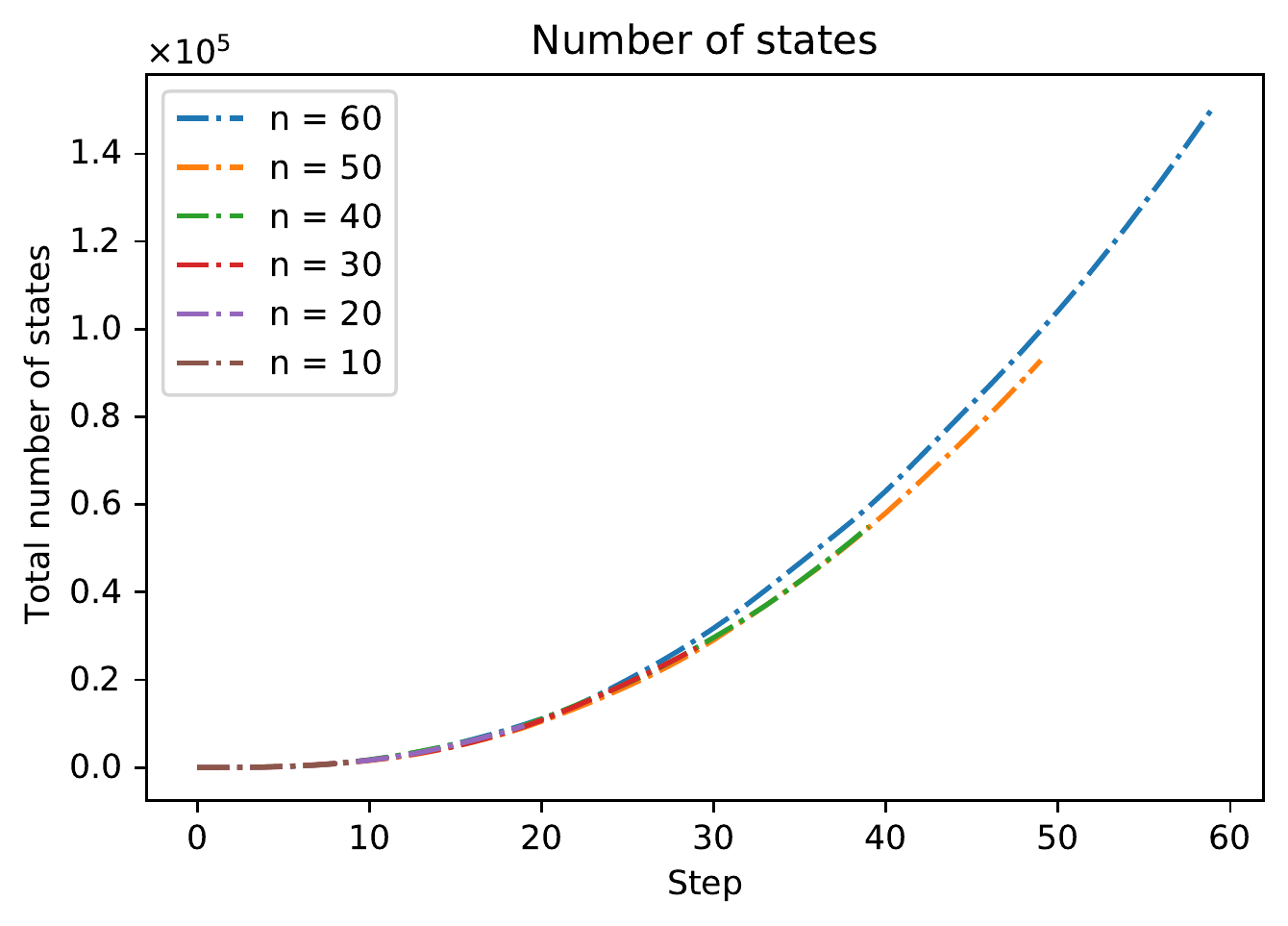}
    \caption[Growth of the number of states. ]{
    The left and middle plots show the growth of the cardinalities of the queues created by Algorithm~\ref{algo:tree_recomb} for fixed $n$ as $i$ increases. 
    The right plot shows a comparison of the growth of the number of states for different $n$ as $i$ increases. 
	}
    \label{fig:number_of_states_and_comparison_number_of_states_toy}
\end{figure}

\begin{figure}[hbt!]
    \centering
        \includegraphics[height=4cm,width=0.32\textwidth]{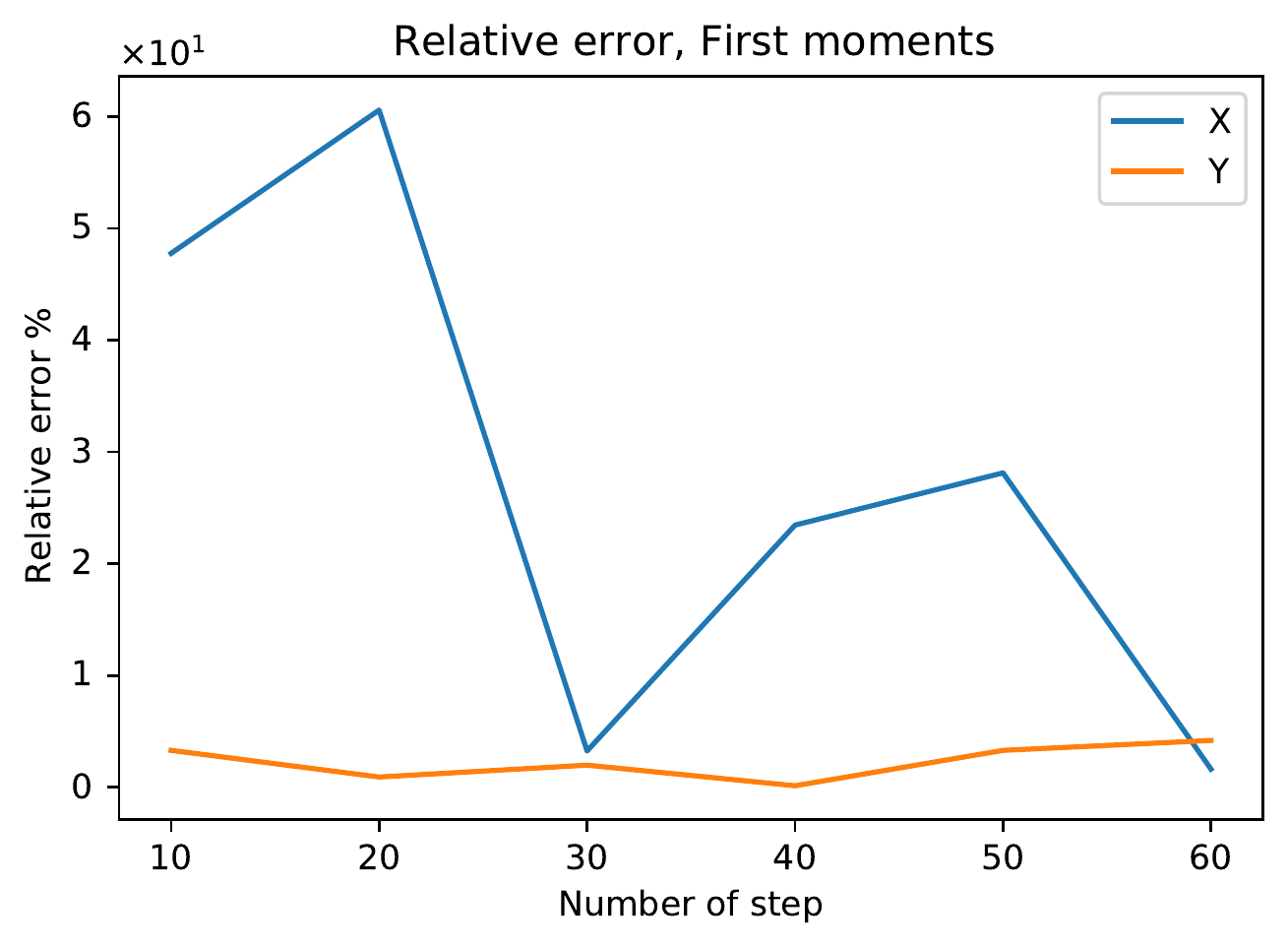}
        \includegraphics[height=4cm,width=0.32\textwidth]{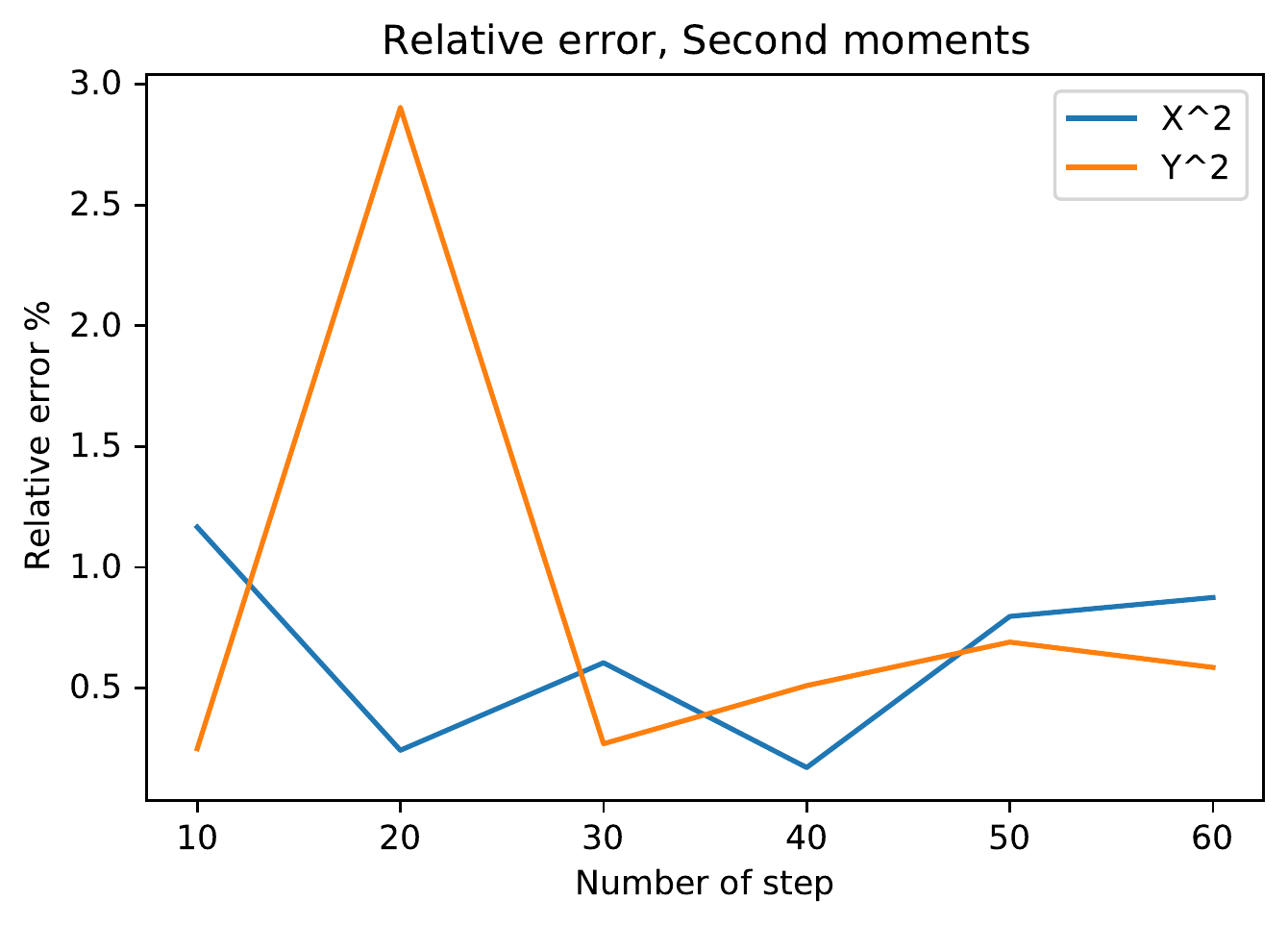}
        
        \includegraphics[height=4cm,width=0.32\textwidth]{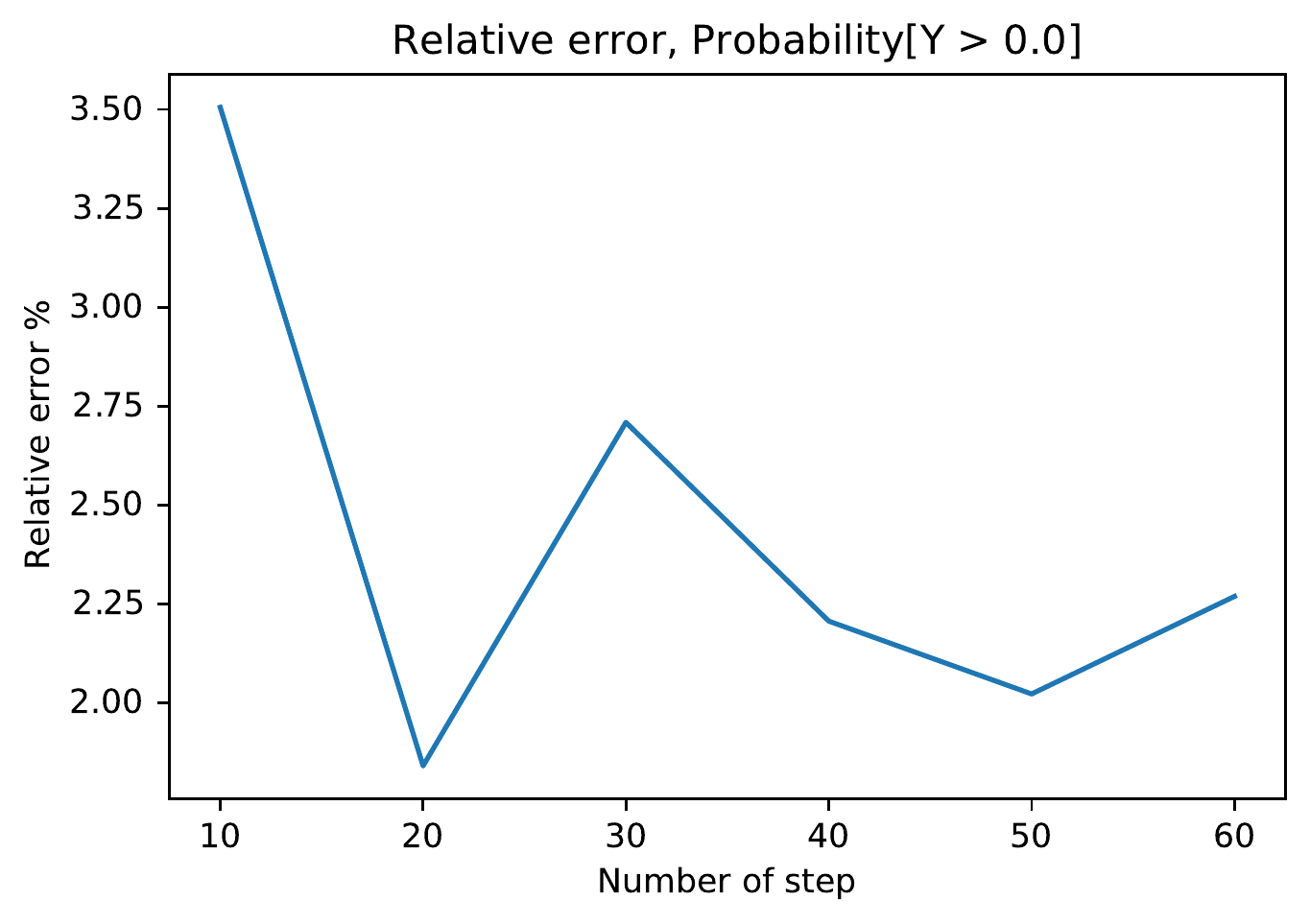}
        \includegraphics[height=4cm,width=0.32\textwidth]{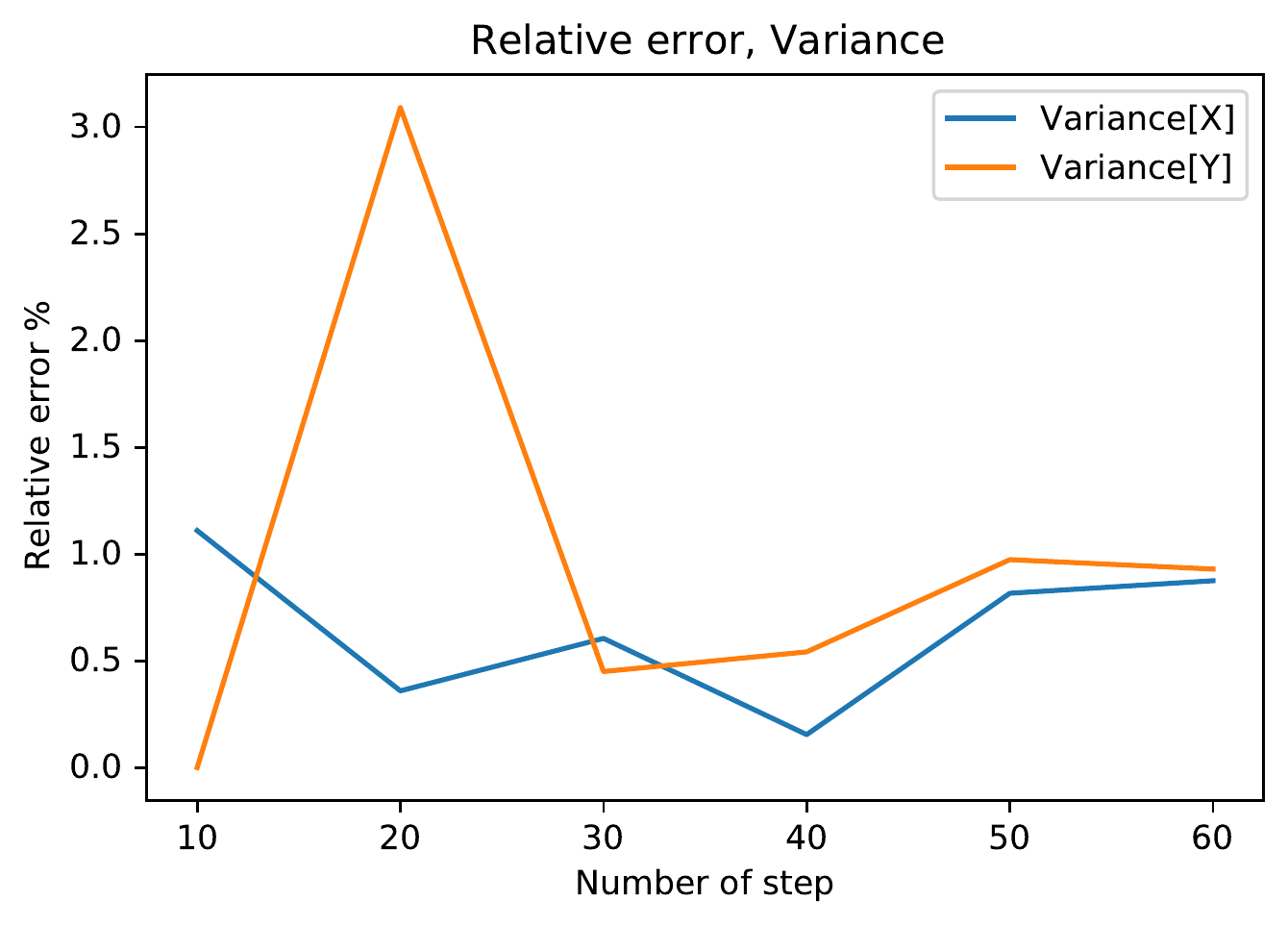}
    \caption[Relative errors. ]{Relative errors for some functionals of $X_1, Y_1$. 
    We compute the first  and the second moments, the variances and the probability that $Y_1$ is greater than 0, respectively $\EE[X_1], \EE[Y_1], \EE[X_1^2], \EE[Y_1^2],\VV[X_1], \VV[Y_1]$ and $\Prob[Y_1>0]$. 
    As ground truth we consider an Euler scheme Monte Carlo simulation of the mode with $n=1000$ steps.  
    }
    \label{fig:rel_errors_toy}
\end{figure}

\begin{figure}[h!]
    \centering
        \includegraphics[height=4cm,width=0.32\textwidth]{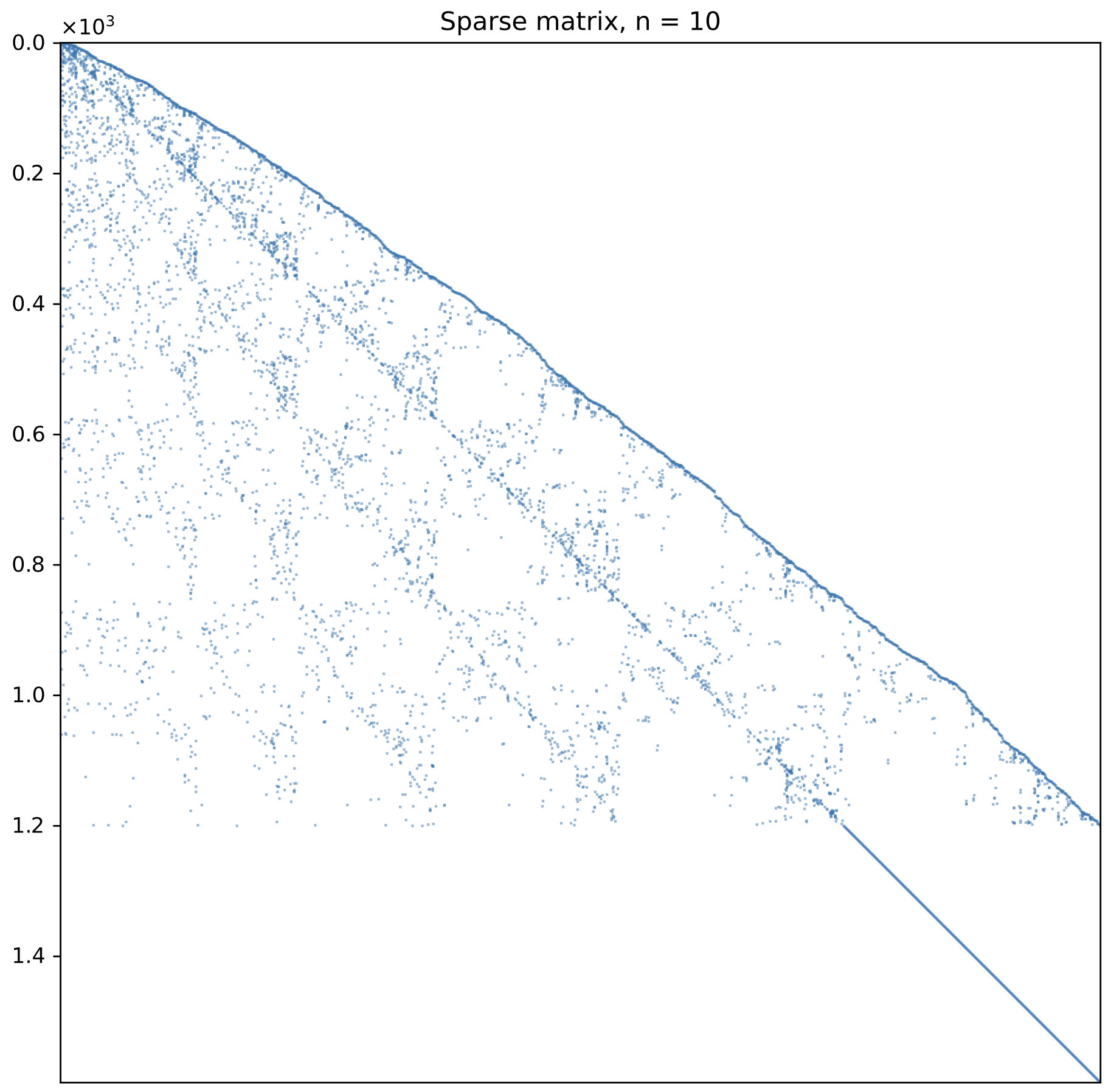}
        \includegraphics[height=4cm,width=0.32\textwidth]{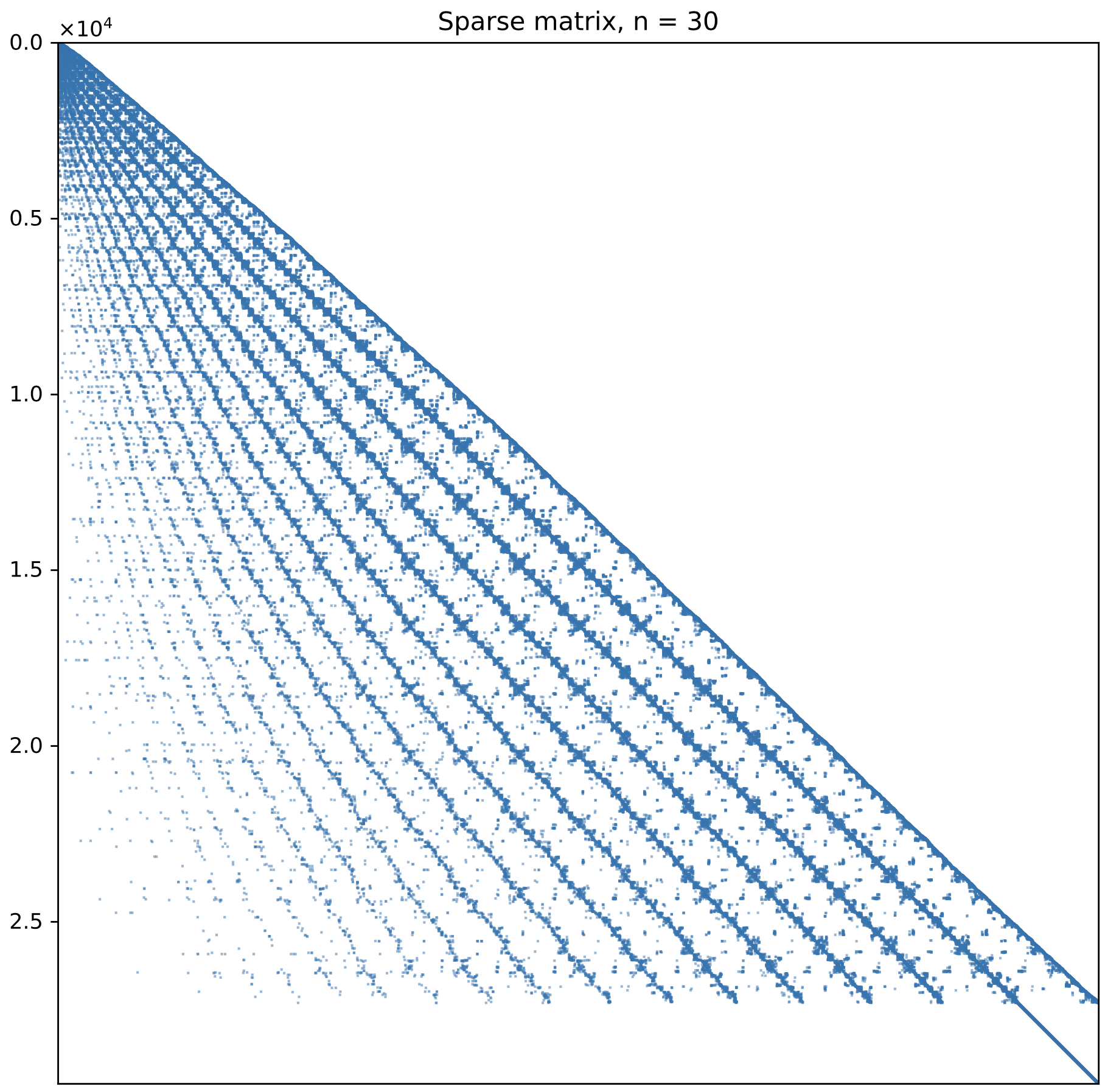}
        \includegraphics[height=4cm,width=0.32\textwidth]{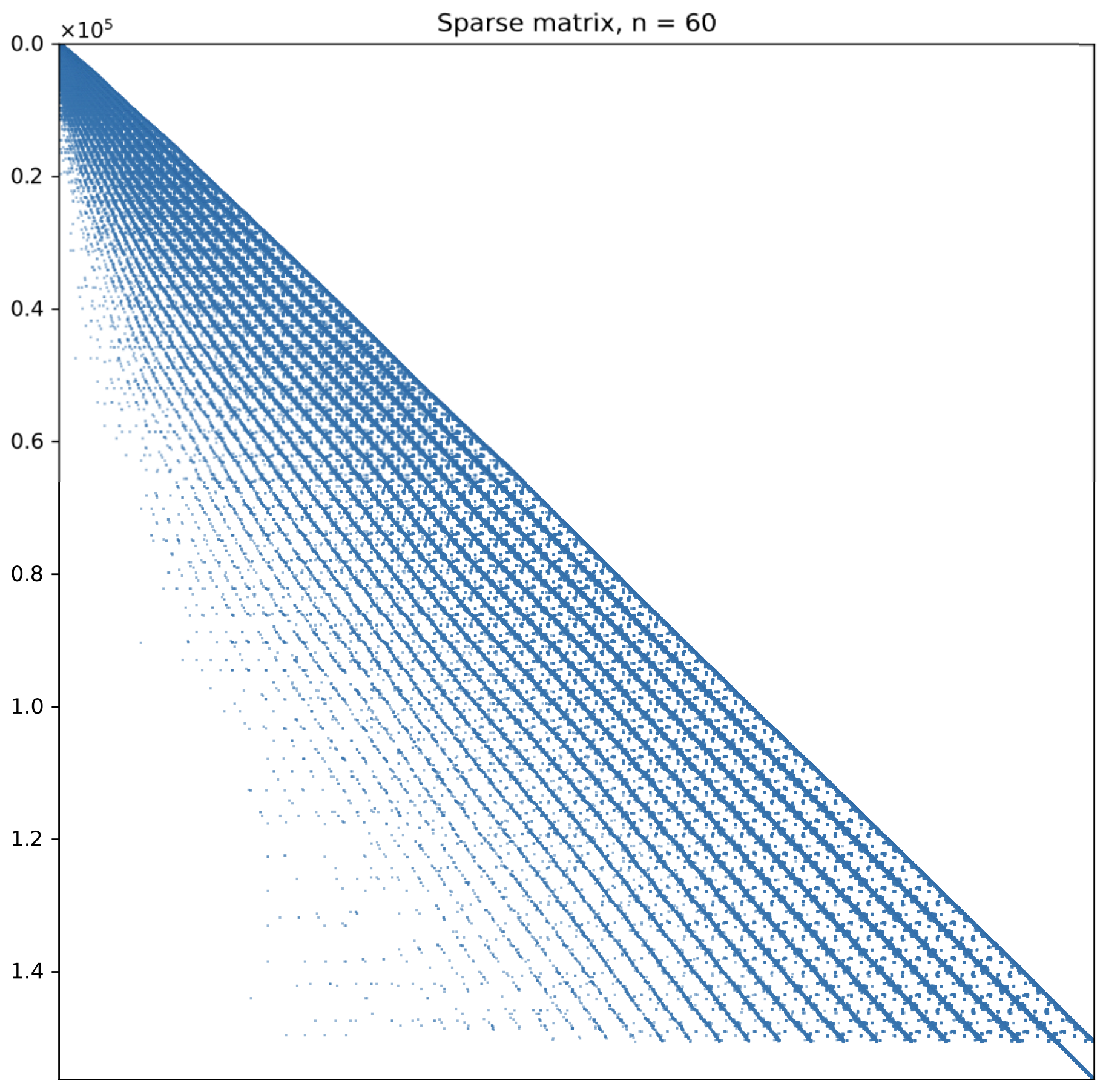}
    \caption[Sparse Matrix.]{Representation as sparse matrix of the Probability Transition Matrix for different values of $n$.}
    \label{fig:sparsity of the matrix_toy}
\end{figure}

It is easy to check that this model satisfies the assumptions of Theorem~\ref{th:mm_d=2}, indeed the vector fields are bounded and the SDE is strictly elliptic: 
$\epsilon=\inf_{x}\lambda_{\min}(x)=\inf_{x}\text{eigenvalues}(\sigma\sigma^{\top}(x))=1$, 
thus we can consider $\gamma_{n}=n^{-1/2}\sqrt{1/3}$.

\paragraph{Results.} 
Figure~\ref{fig:transition_matrix_toy} visualizes the iterative nature of the MC construction: starting from the initial point $x_0$ the algorithm selects at every time step a few possible attainable states on the two-dimensional lattice, such that the local consistency condition is satisfied. 
As time progresses, new states get added and many previous states are revisited (leading to a recombining tree).
Figure~\ref{fig:number_of_states_and_comparison_number_of_states_toy} shows the growth of the state space as a function of time for different values of $n$,  and Figure~\ref{fig:rel_errors_toy} shows the relative error,
Finally, Figure~\ref{fig:sparsity of the matrix_toy} shows the resulting sparse transition matrices.
Unsurprisingly, the transition matrices are sparse, as seen in Figure~\ref{fig:sparsity of the matrix_toy}, since this is optimized by construction, see also Appendix~\ref{app:details}.

\subsection{Example: Mean reverting Heston model}\label{sec:Mean reverting Heston model}
For the second experiment, we consider the Heston model (HM) \cite{Heston1993}, a popular stochastic volatility model in mathematical finance
\begin{align}\label{eq:HM}
\dif P_{t}=&\mu P_{t}\dif t+\sqrt{V_{t}}P_{t}\dif W_{t}^{P}, \\
\dif V_{t}=&\lambda(\theta-V_{t})\dif t+\xi\sqrt{V_{t}}\dif W_{t}^{V},
\end{align}
where $P_t$ is the price of an asset and $V_t$ its instantaneous
variance; 
$W_{t}^{P}, W_{t}^{V}$ are two standard Brownian Motions with correlation parameter $\rho$ and $\lambda, \theta, \mu, \xi$ are all constants. We recall that if $2\lambda\theta>\xi^2$, then $V_t$ is a.s. positive, see for example \cite{Cox1985}.
Tree approximations that are tailor-made to the HM model in \eqref{eq:HM}, e.g.~\citet{Zeng2019,Akyildirim2014}, 
have been developed but, to the best of our knowledge, these do not apply when $V$ depends on $P$. 
The latter is a desirable model property since empirical observation suggest the price and the volatility to be inversely related, and that $V$ has a mean reverting behaviour; for theoretical properties of such an SDE, see \citet{Romano1997}. 
A simple example of such price-variance relation that we use for our numerical experiment is  
\[
 \theta_t \coloneqq \theta(P_t) \coloneqq c\frac{1}{P_t+1}+k.
\]
Indeed, one can check that in this case $V$ is mean-reverting\footnote{By using
  the differential form of $d (e^{t\lambda}V_{t})$ and Ito's rule. We have to suppose the existence of $\mathbb{E}\theta_{\infty}$.}, 
$\lim_{t\to\infty}\mathbb{E}V_{t}=\mathbb{E}\theta_{\infty}$. 
To sum up, we consider the system
\begin{align}\label{eq:HM_mod}
\dif\left(\!\begin{array}{c}
P_{t}\\
V_{t}
\end{array}\!\right)\!=\!\left(\!\begin{array}{c}
\mu P_{t}\\
\lambda\left(\frac{c}{1+P_{t}}+k-V_{t}\right)
\end{array}\!\right)\! \dif t\!+\!\sqrt{V_{t}}\left(\!\begin{array}{cc}
P_{t} & 0\\
0 & \xi
\end{array}\!\right)\!\left(\!\begin{array}{cc}
1 & 0\\
\rho & \sqrt{1-\rho^{2}}
\end{array}\!\right)\!\dif\!\left(\!\begin{array}{c}
W_{t}^{1}\\
W_{t}^{2}
\end{array}\!\right),
\end{align}
where $W_{t}^{1}, W_{t}^{2}$ are two independent Brownian Motions.
Moreover, in a finance context it is more natural model $\log P_t$ instead of $P_t$, and we reformulate the above system as 
\begin{align}
  \dif\left(\begin{array}{c}
           \ln P_{t}\\
           V_{t}
         \end{array}\right)=
         &\left(\begin{array}{c}
                  \mu-\frac{1}{2}V_{t}\\
                  \lambda\left(\frac{c}{1+P_{t}}+k-V_{t}\right)
                \end{array}\right)\dif t+\sqrt{V_{t}}\left(\begin{array}{cc}
                                                       1 & 0\\
                                                       \xi\rho & \xi\sqrt{1-\rho^{2}}
                                                     \end{array}\right)\dif\left(\begin{array}{c}
                                                                                W_{t}^{(1)}\\
                                                                                W_{t}^{(2)}
                                                                              \end{array}\right).
\end{align}

\begin{figure}[hbt!]
    \centering
        \includegraphics[height=4cm,width=0.32\textwidth]{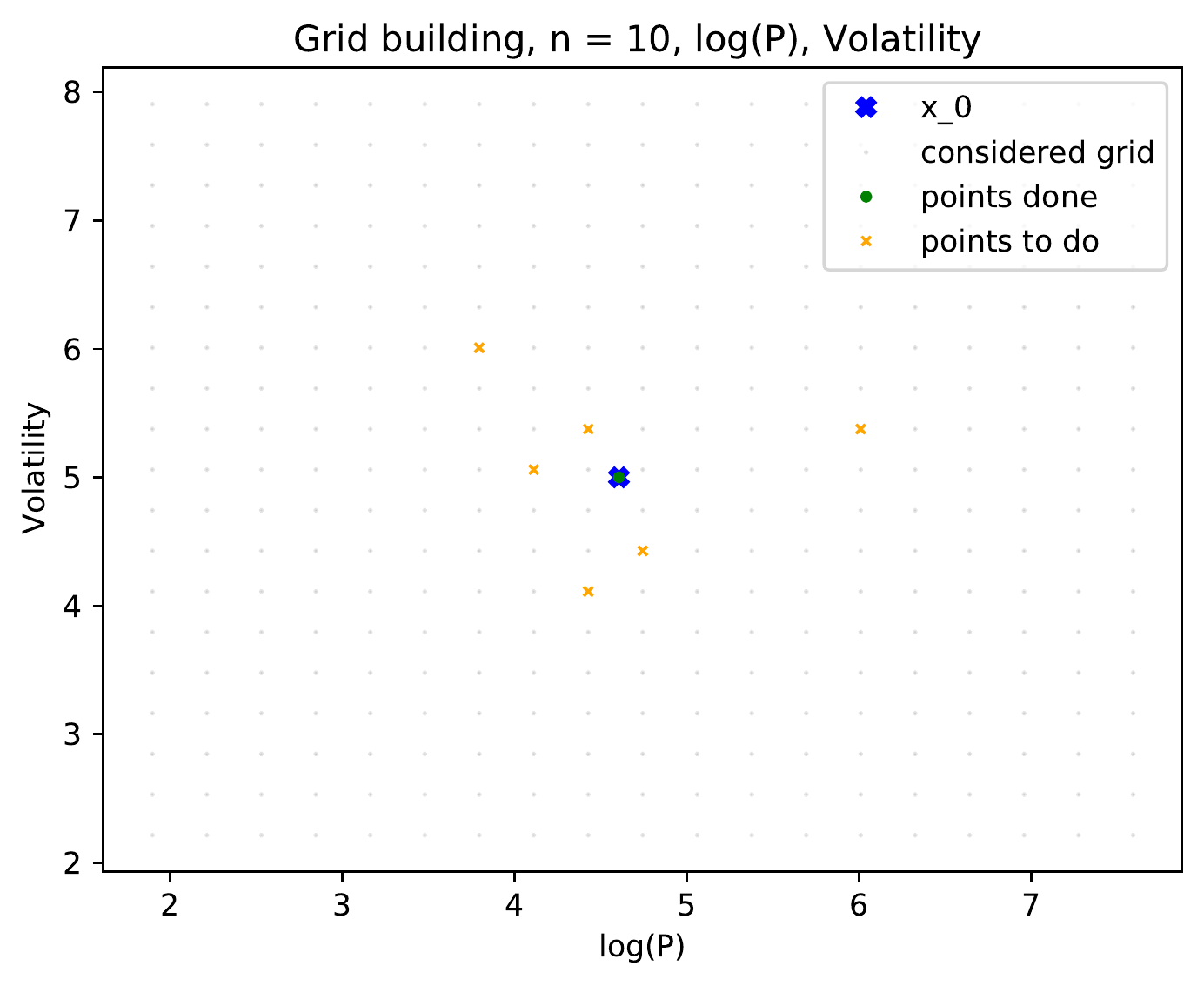}
        \includegraphics[height=4cm,width=0.32\textwidth]{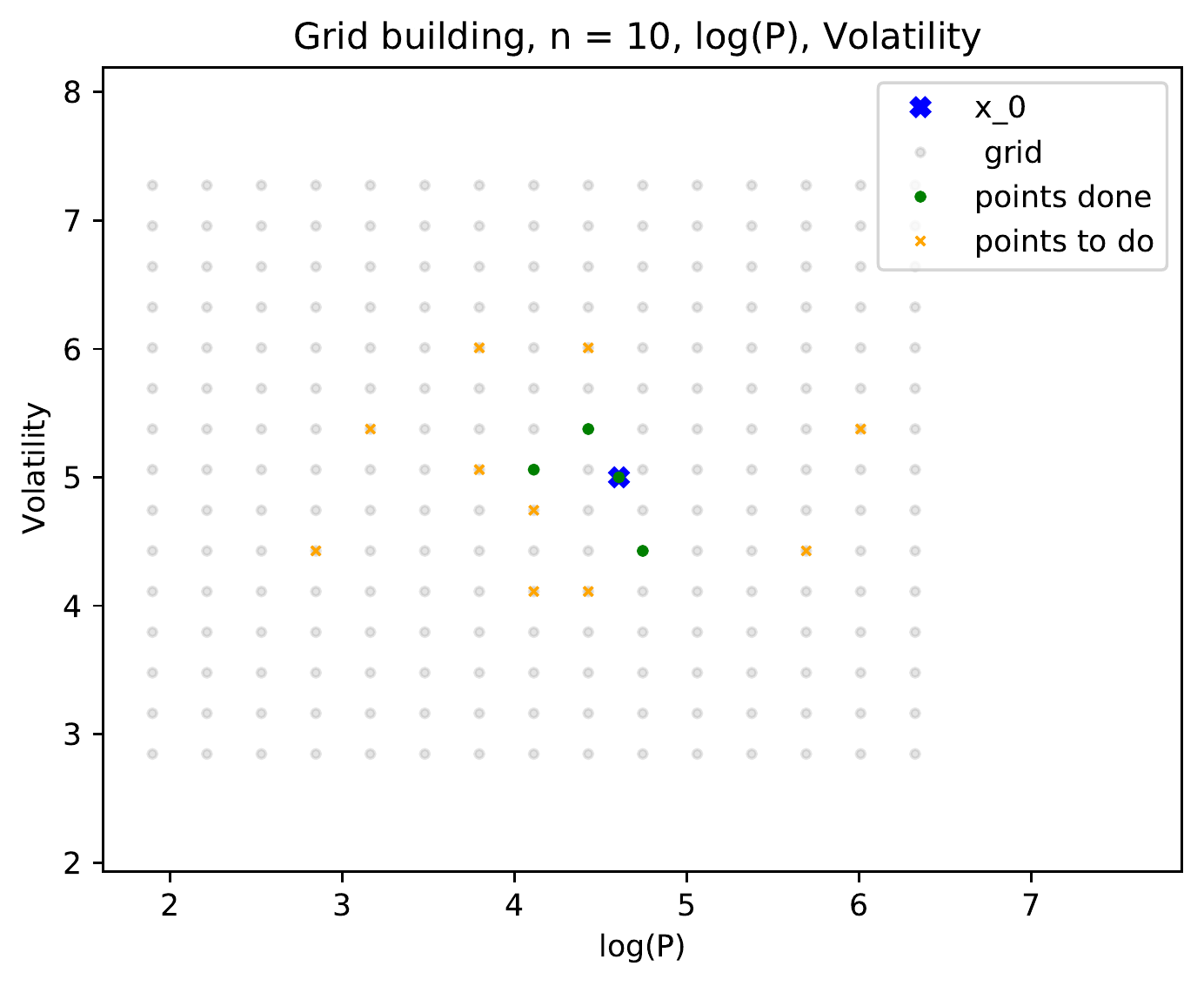}
        \includegraphics[height=4cm,width=0.32\textwidth]{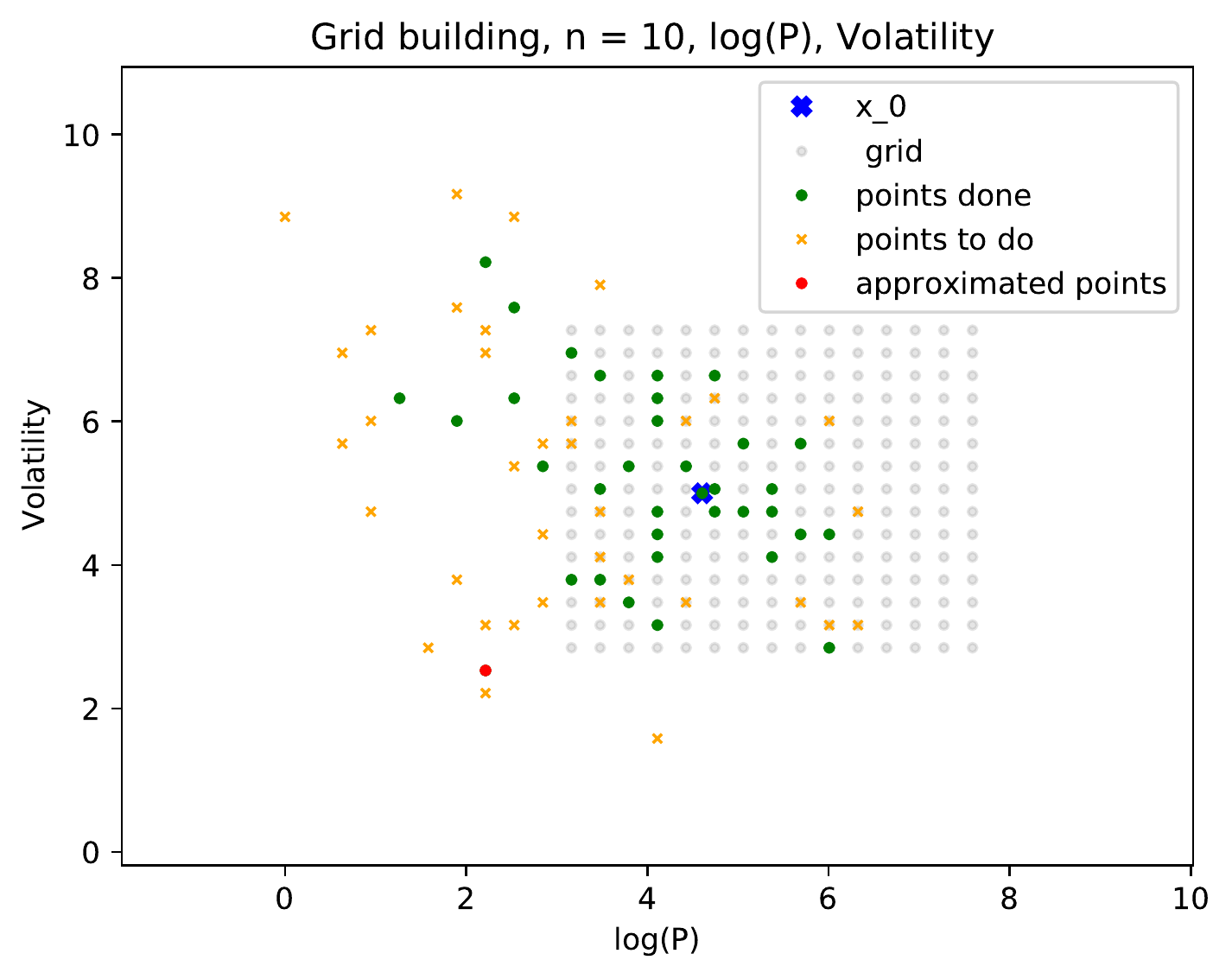}
        
        \includegraphics[height=4cm,width=0.32\textwidth]{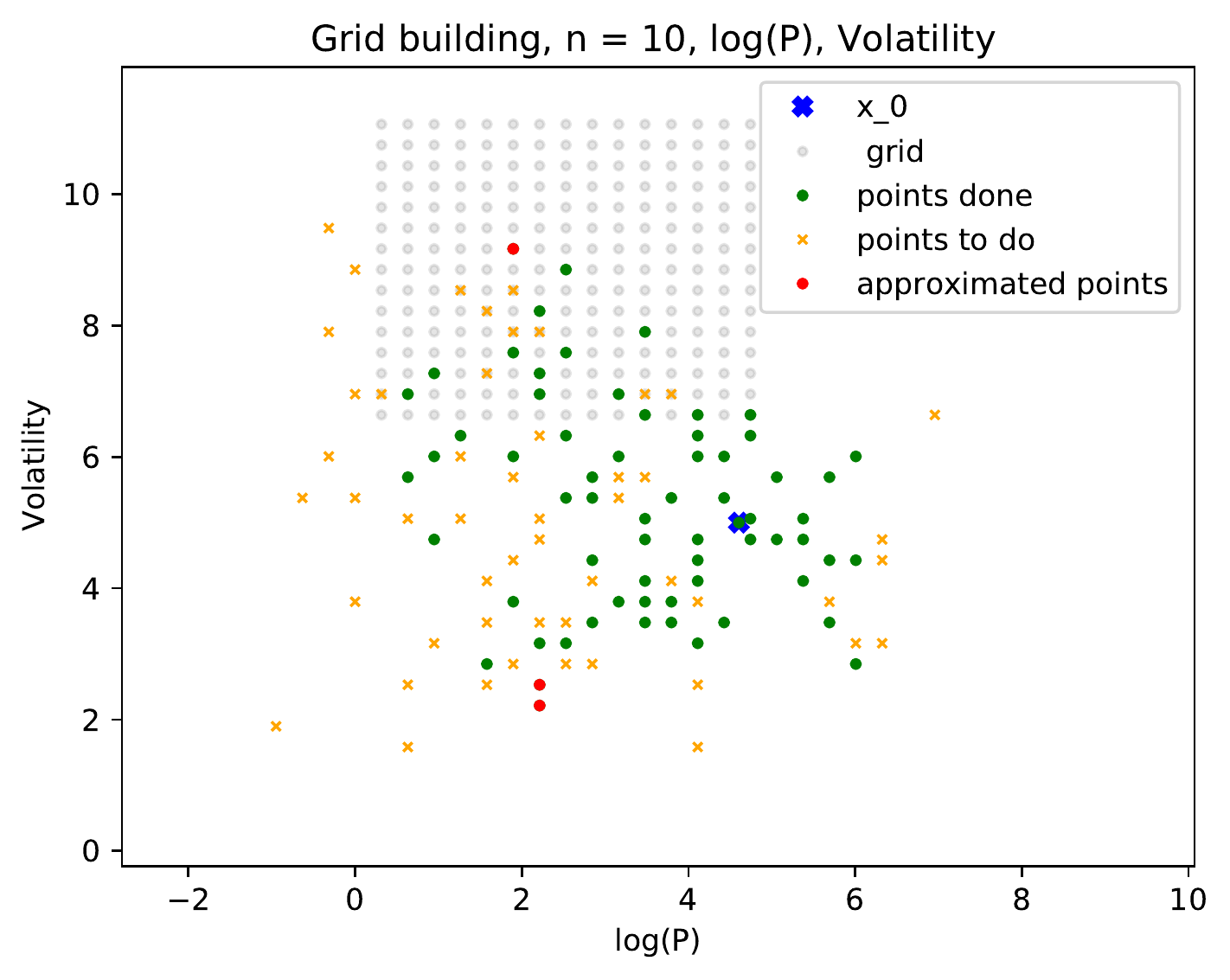}
        \includegraphics[height=4cm,width=0.32\textwidth]{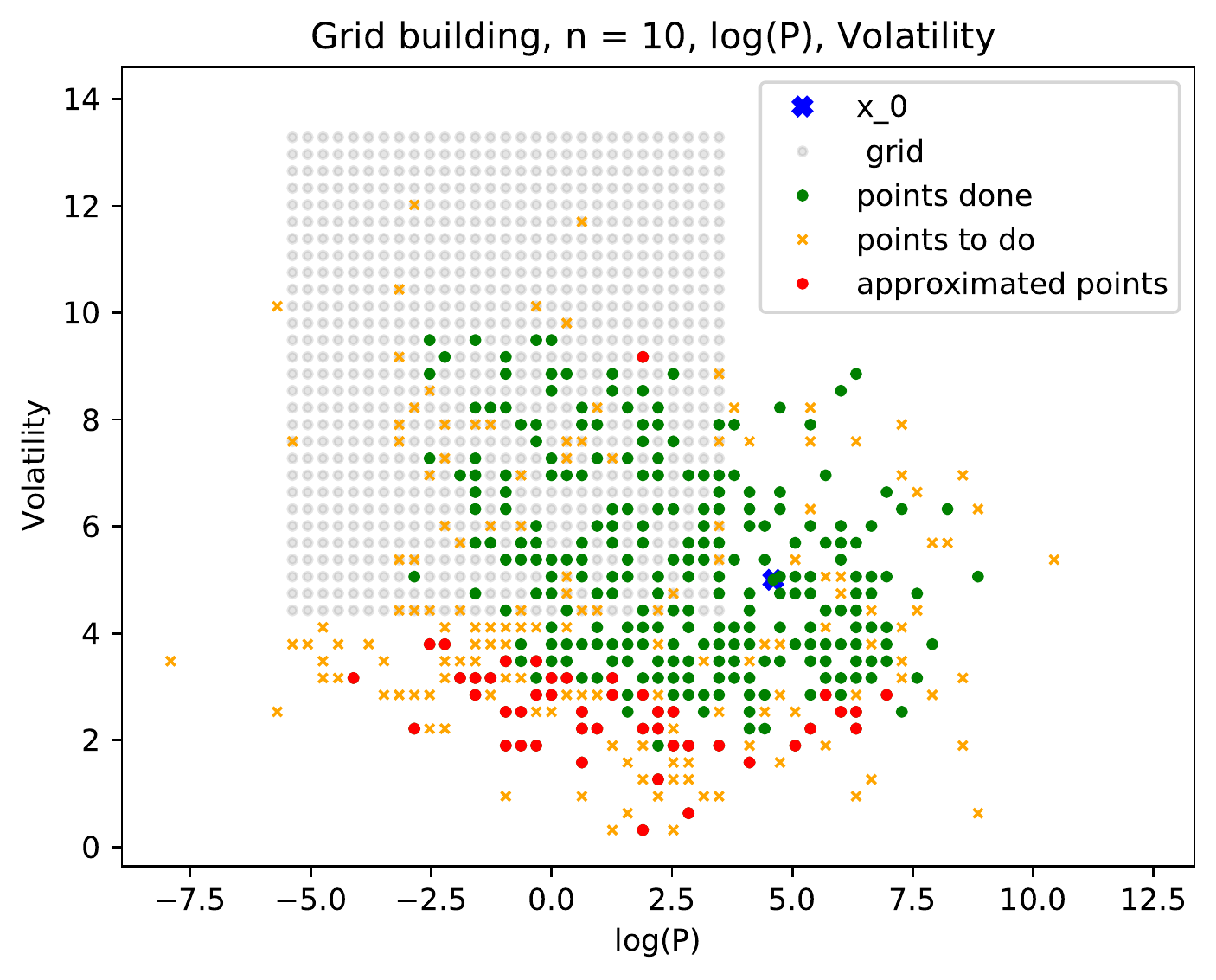}
        \includegraphics[height=4cm,width=0.32\textwidth]{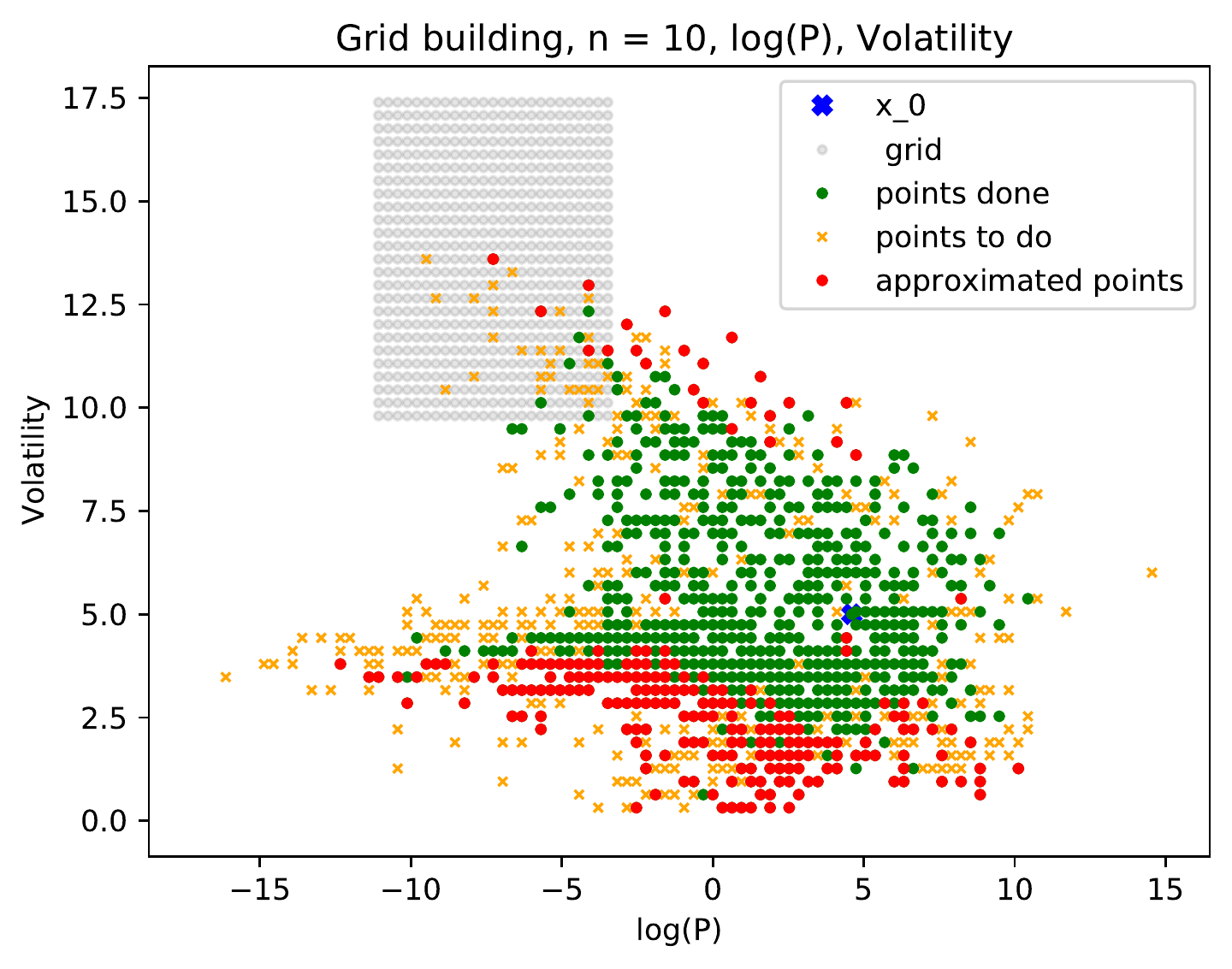}
    \caption[Construction of the transition matrix. ]{
    Construction of the transition matrix. 
    Starting from $x_0$ we consider to build the random variable $X_1$ such that it has support on the grey points and the first two moments are matched exactly. 
    In the top-left plot, $X_1$ has support on the yellow points. 
    Then this procedure is iterated for all the possible points in the support of the measure, following Algorithm~\ref{algo:tree_recomb}. 
    In contrast to Figure~\ref{fig:transition_matrix_toy} of the previous example, now also red points appear since we are in the situation of case~\eqref{itm:solve by constrained minimization} discussed above. 
    }
    \label{fig:transition_matrix}
\end{figure}
\begin{figure}[h!]
    \centering
        \includegraphics[height=4cm,width=0.32\textwidth]{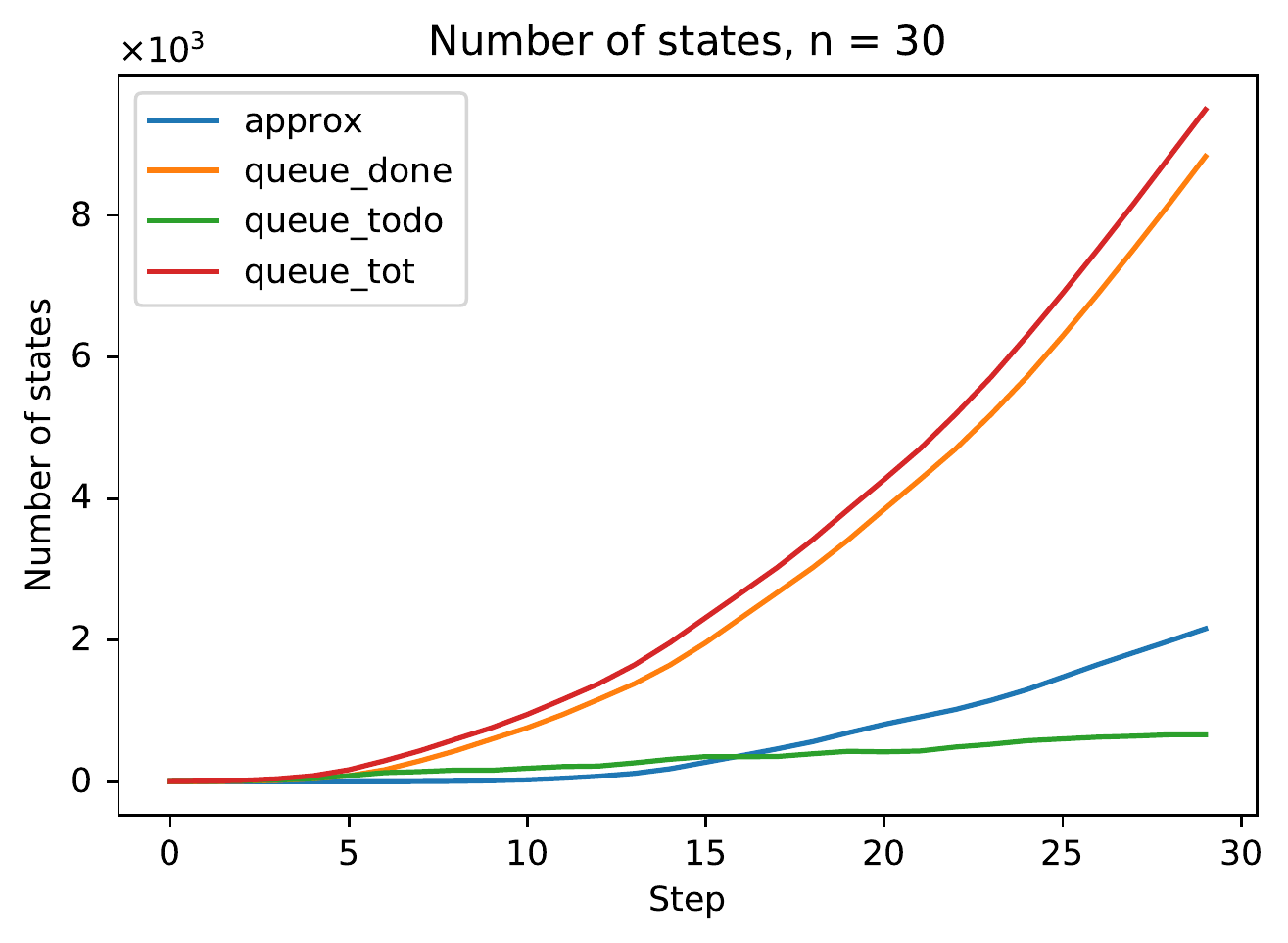}
        \includegraphics[height=4cm,width=0.32\textwidth]{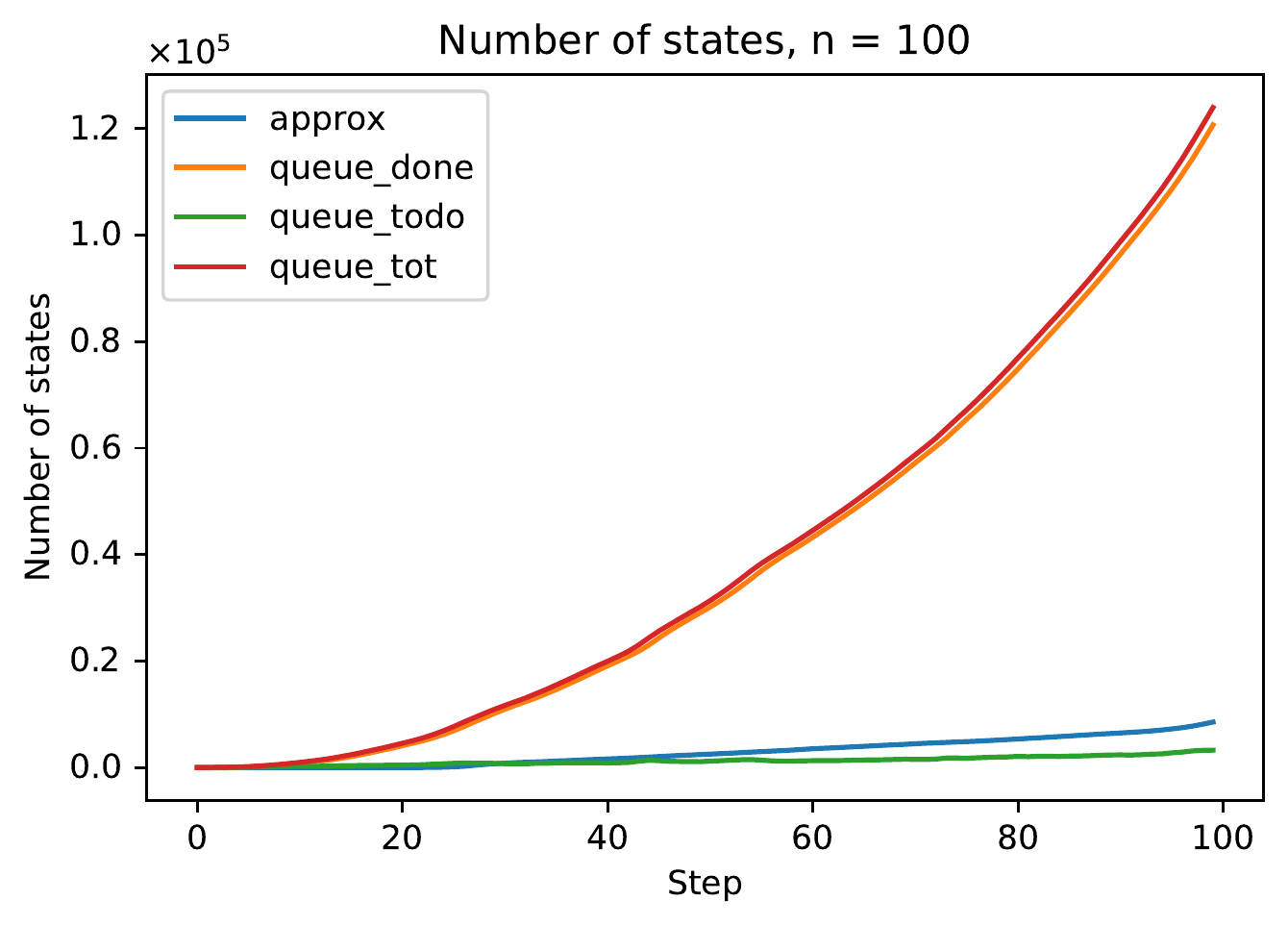}
        \includegraphics[height=4cm,width=0.32\textwidth]{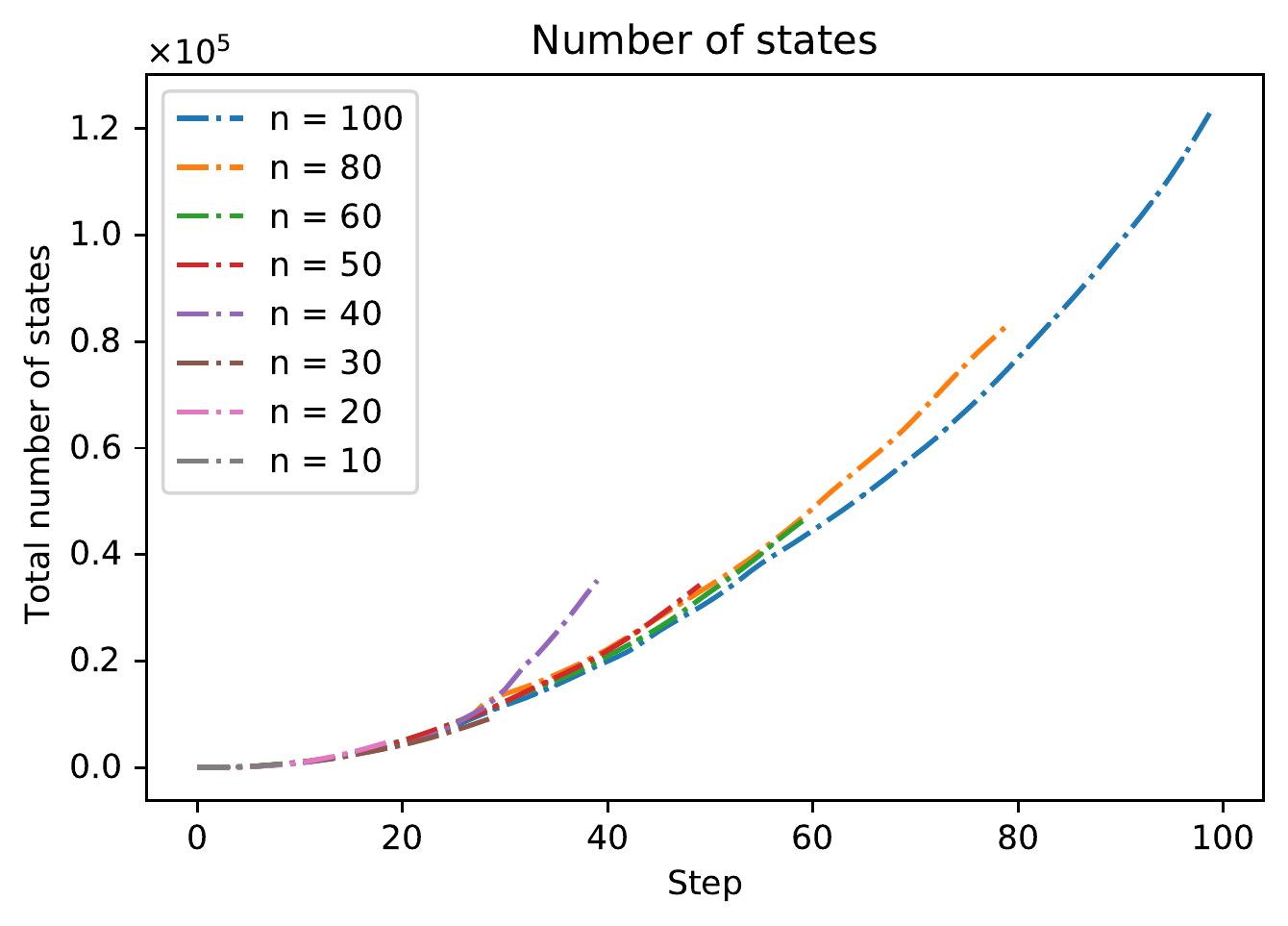}
    \caption[Growth of the number of states. ]{
    The left and middle plots show the growth of the cardinalities of the queues created by Algorithm~\ref{algo:tree_recomb} for fixed $n$ as $i$ increases. 
    The right plot shows a comparison of the growth of the number of states for different $n$ as $i$ increases. 
	}
    \label{fig:number_of_states_and_comparison_number_of_states}
\end{figure}
\begin{figure}[h]
    \centering
        \includegraphics[height=4cm,width=0.32\textwidth]{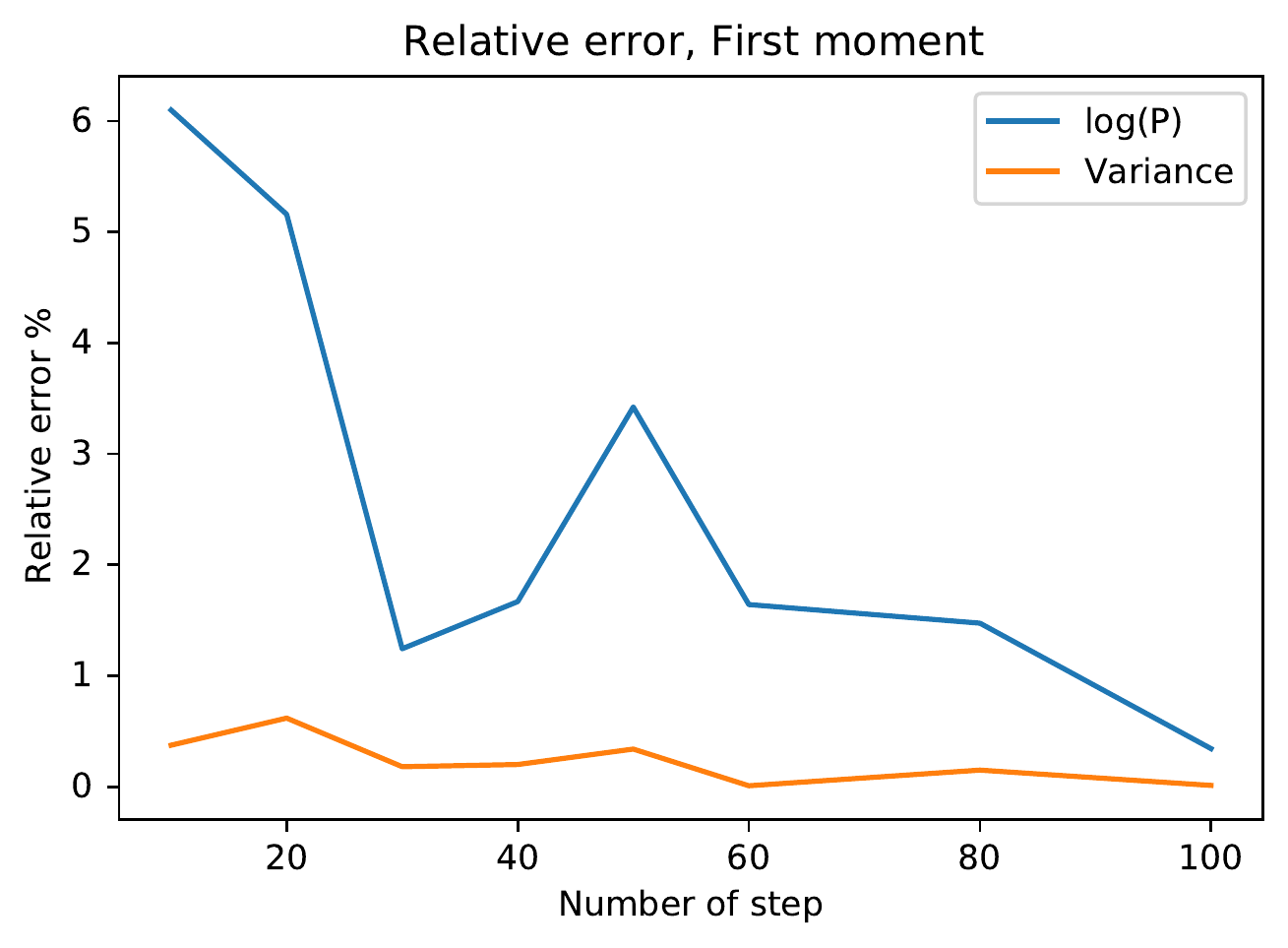}
        \includegraphics[height=4cm,width=0.32\textwidth]{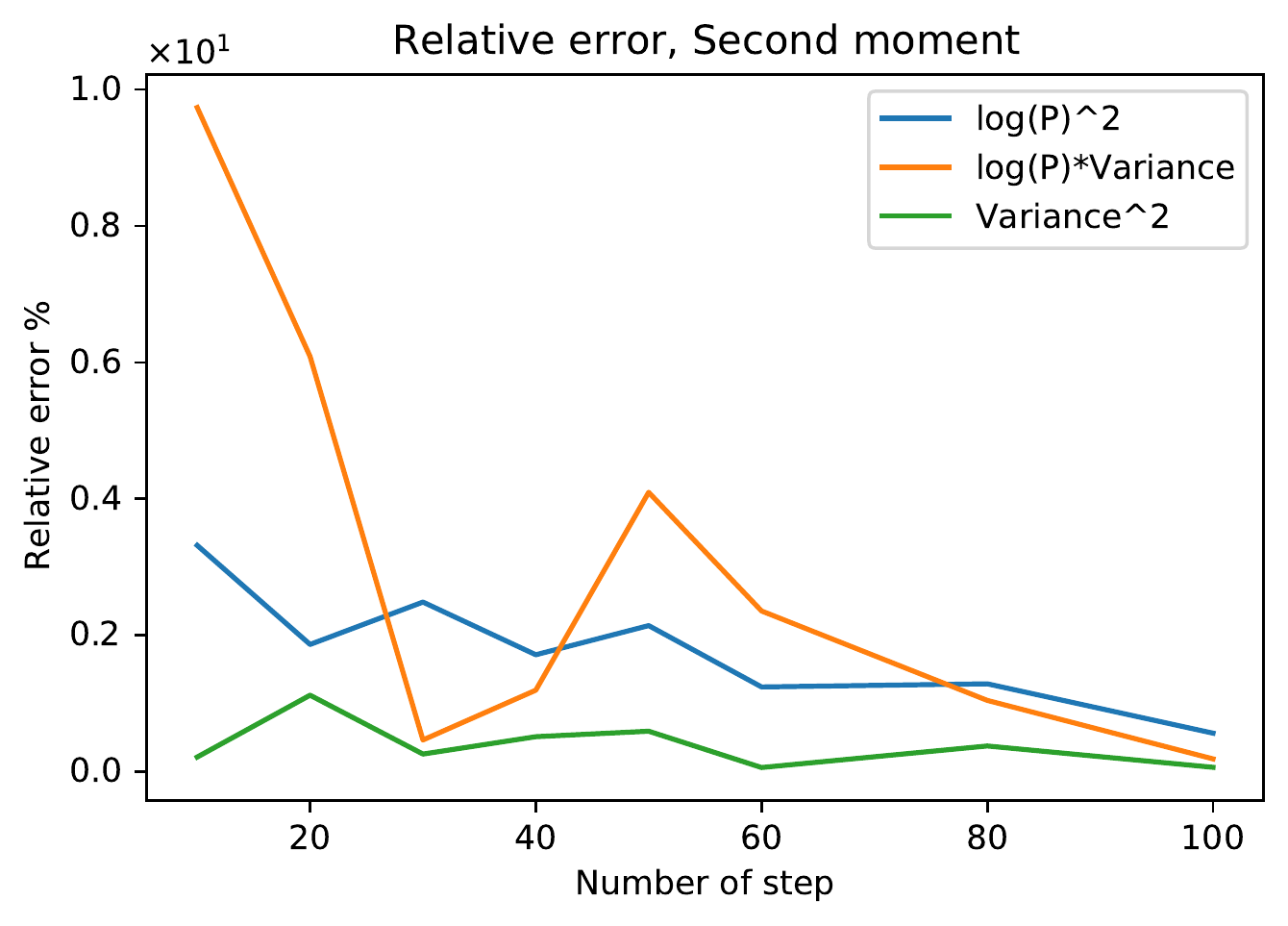}
        \includegraphics[height=4cm,width=0.32\textwidth]{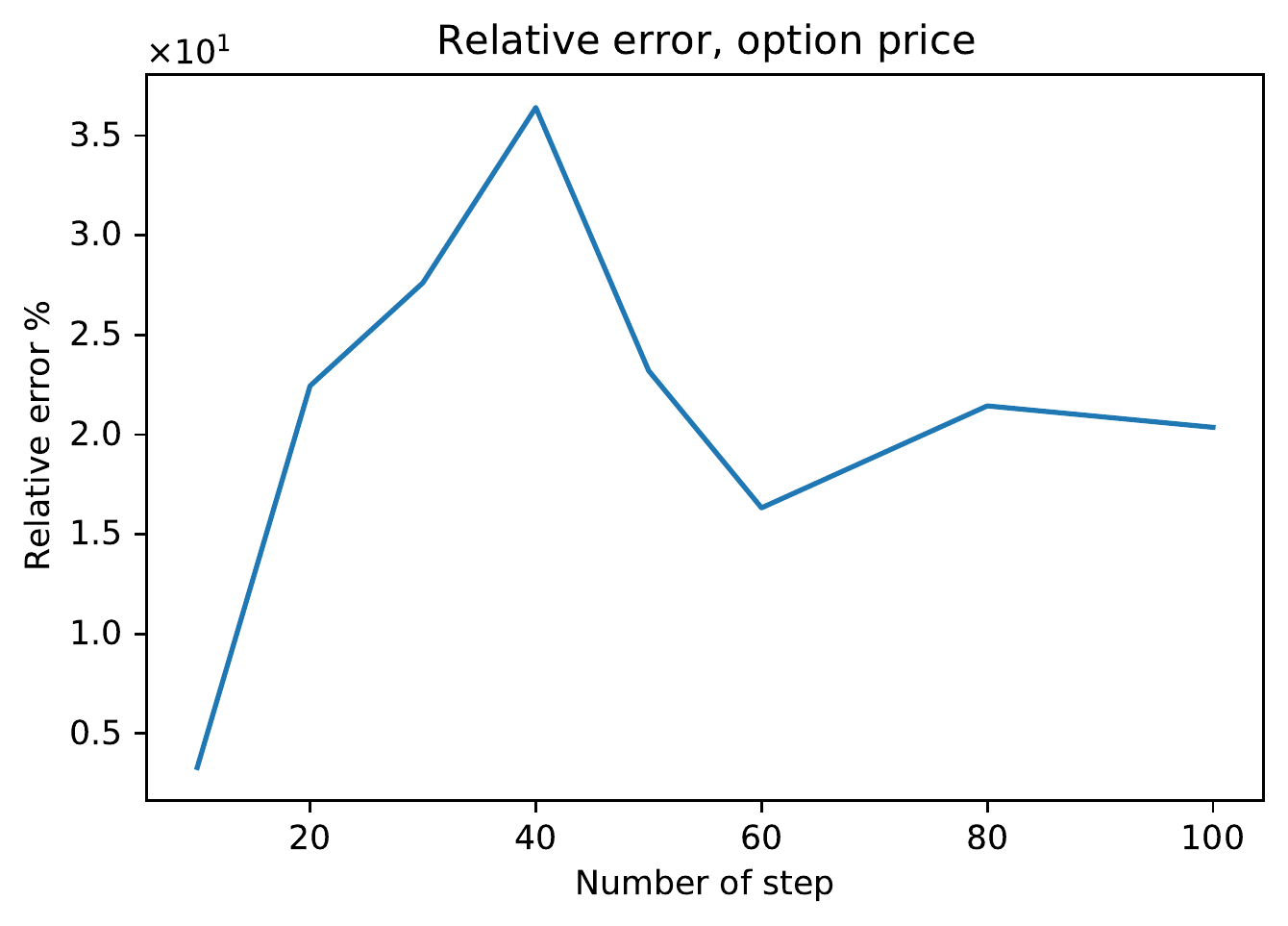}
        
        \includegraphics[height=4cm,width=0.32\textwidth]{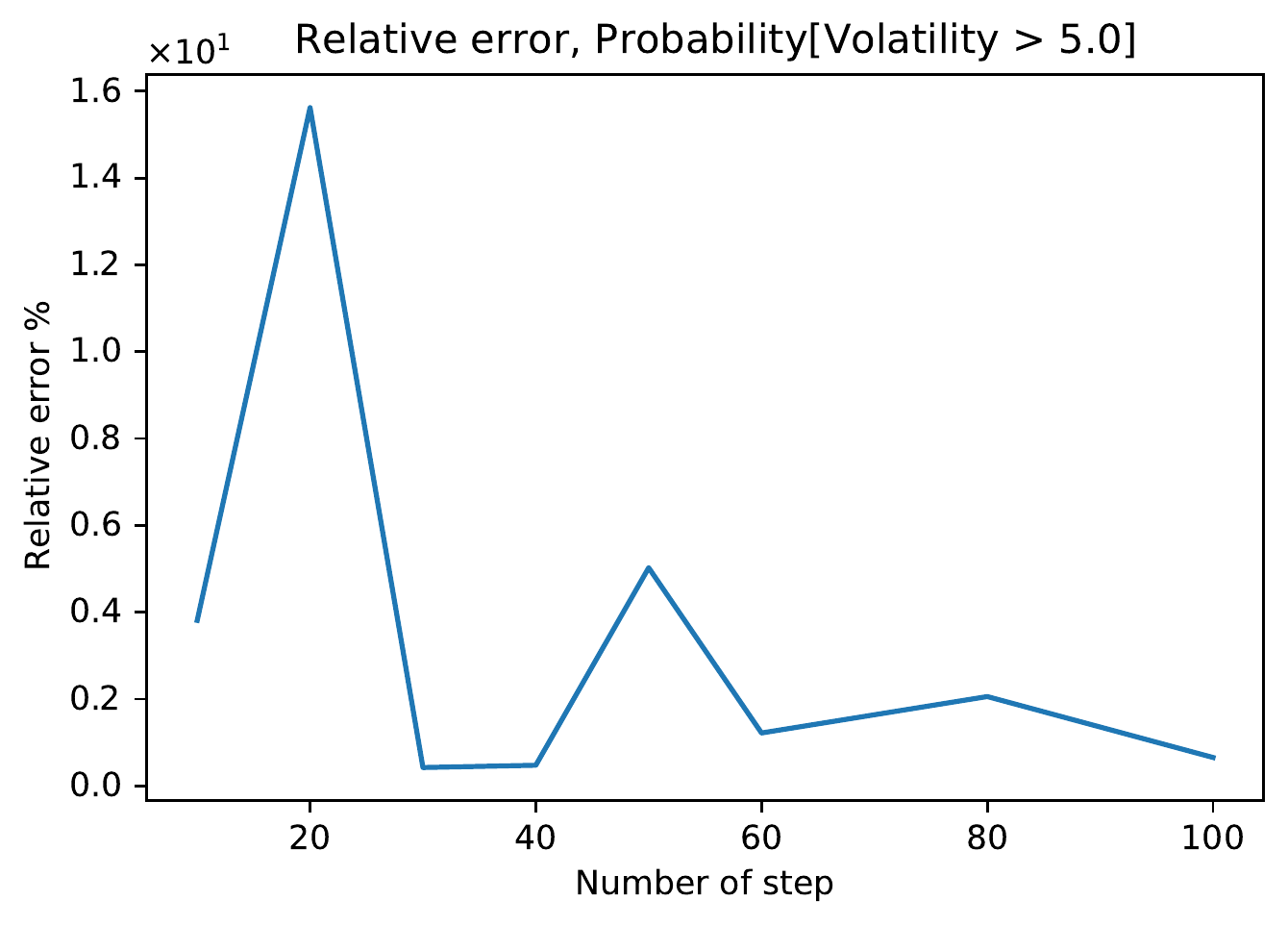}
        \includegraphics[height=4cm,width=0.32\textwidth]{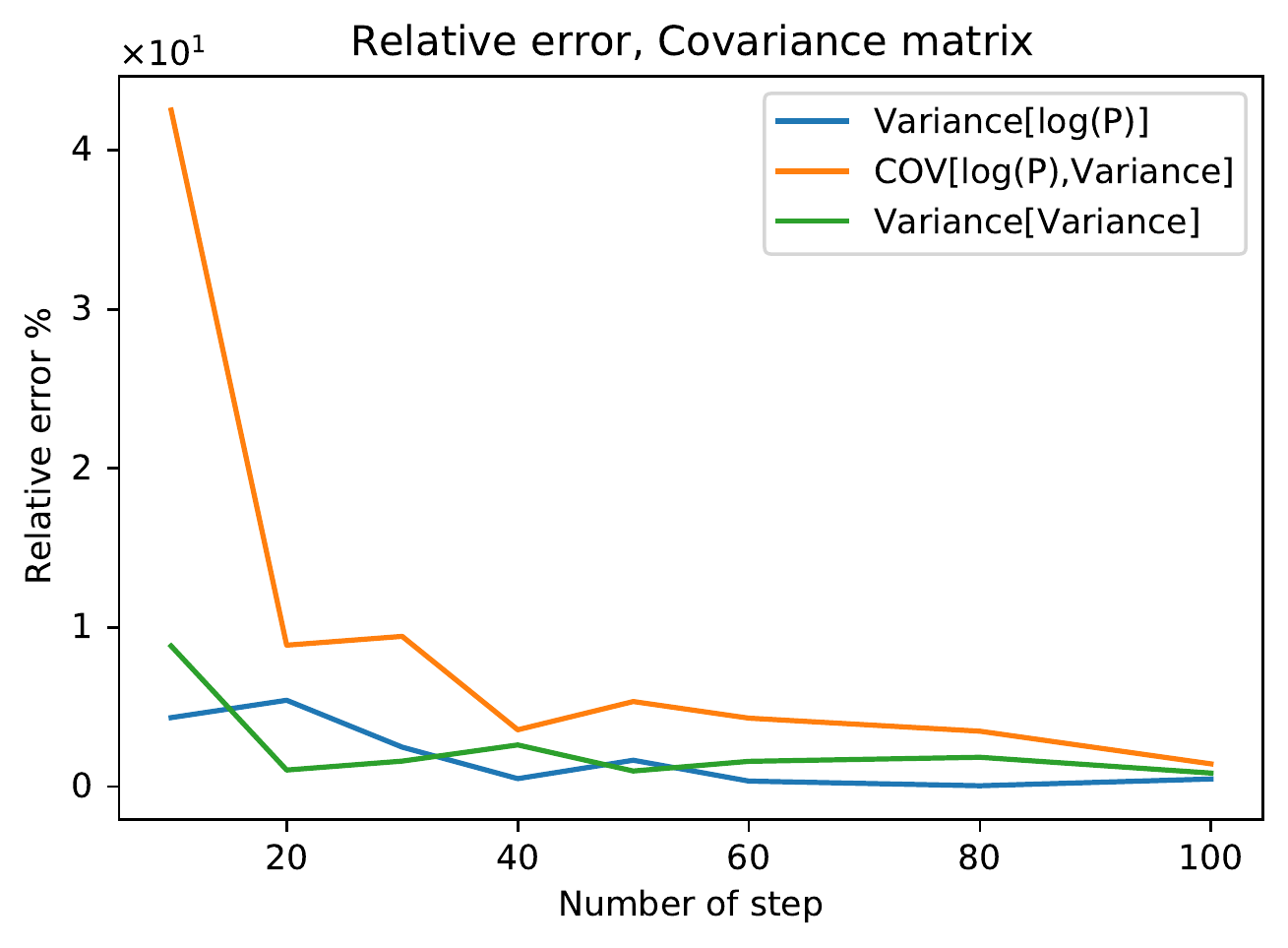}
    \caption[Relative errors. ]{Relative errors for some functionals of $P_1, V_1$. For the option, we consider a European Option with strike $100$.
    We compute the first  and the second moments, the variances and the probability that $Y_1$ is greater than $5$, respectively $\EE[(\log(P_1),V_1)], \EE[(\log(P_1), V_1)^{\otimes 2}],\VV[(\log(P_1), V_1)]$ and $\Prob[V_1>5]$.  The formula for the option price is $\EE[\max(P_1-100,0)]$. 
    As ground truth we consider an Euler scheme Monte Carlo simulation of the mode with $n=1000$ steps.  
    }
    \label{fig:rel_errors}
\end{figure}

\begin{figure}[h!]
    \centering
        \includegraphics[height=4cm,width=0.32\textwidth]{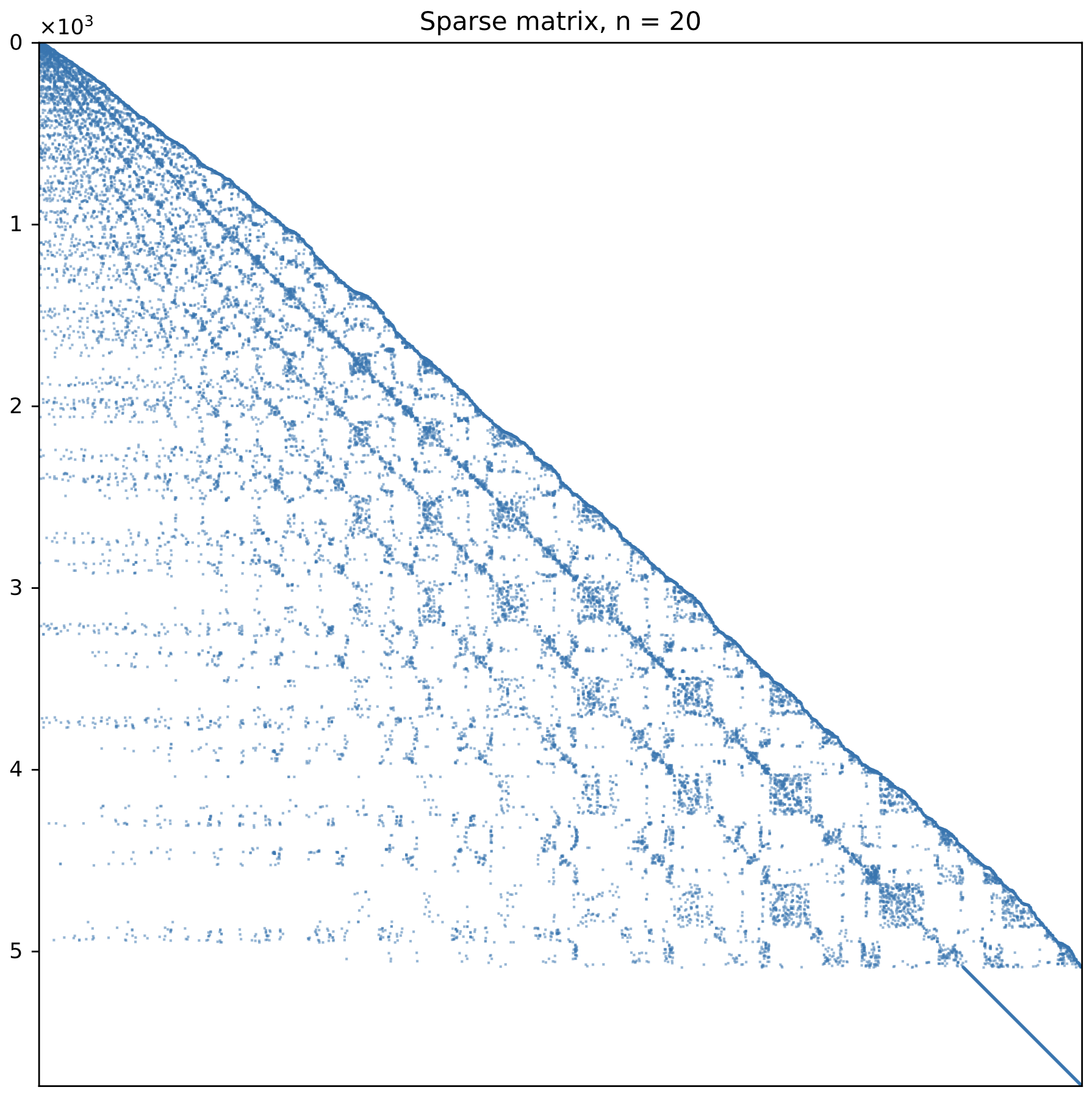}
        \includegraphics[height=4cm,width=0.32\textwidth]{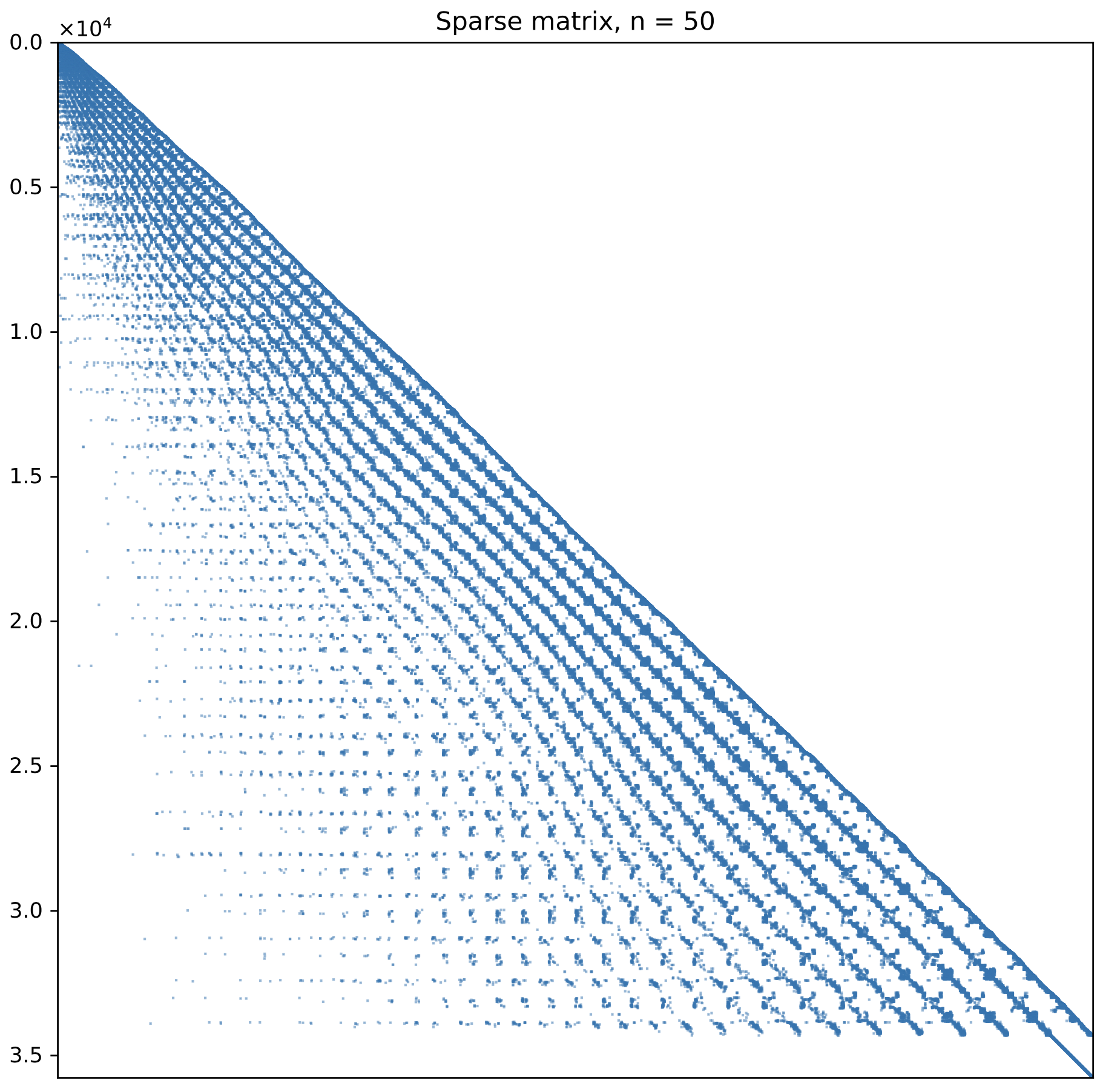}
        \includegraphics[height=4cm,width=0.32\textwidth]{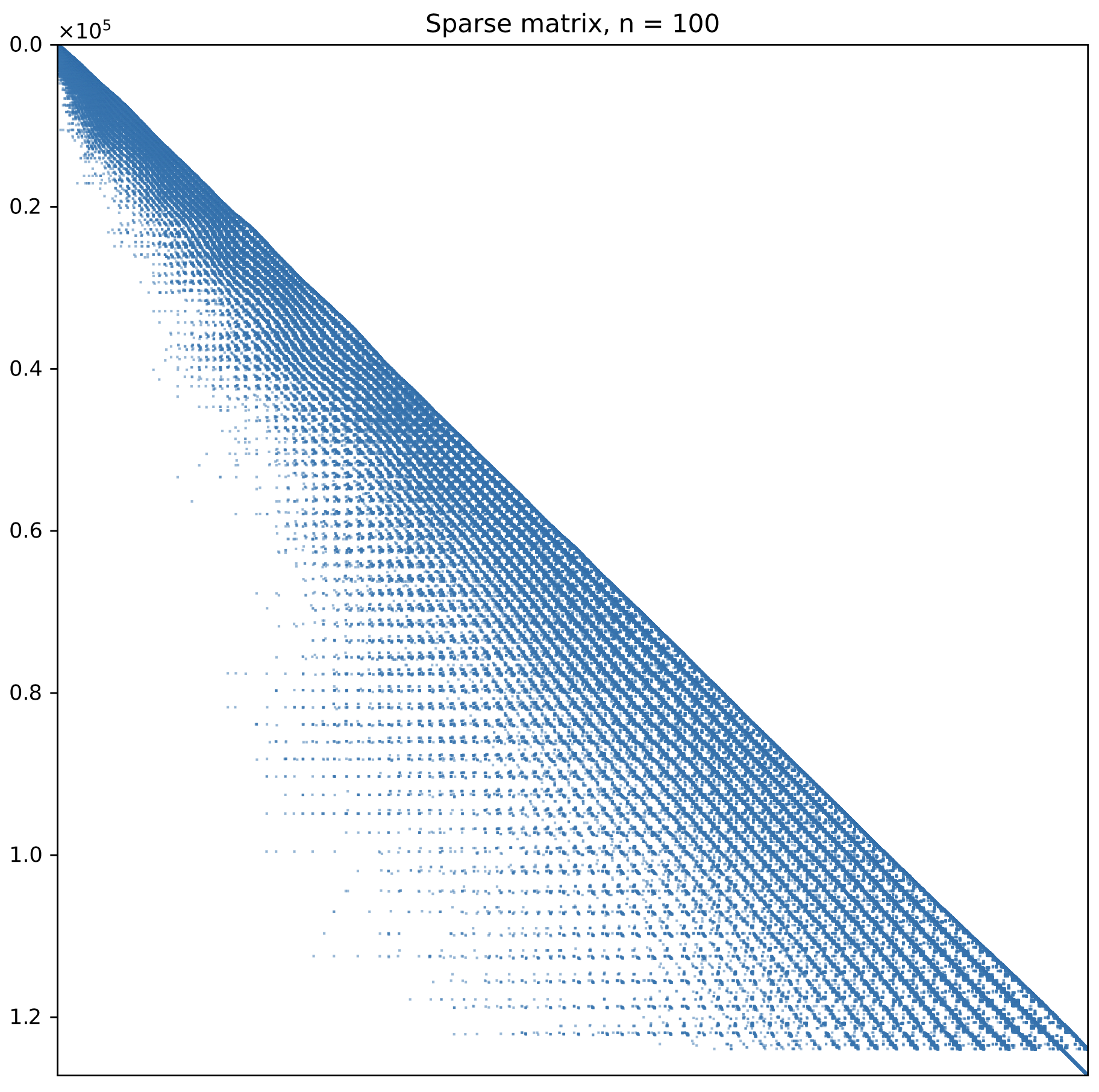}
    \caption[Sparse Matrix.]{Representation as sparse matrix of the Probability Transition Matrix for different values of $n$.}
    \label{fig:sparsity of the matrix}
\end{figure}

The assumptions of Theorem~\ref{th:mm_d=2} are not met since {the eigenvalues of the volatility are not bounded away from $0$ due to the presence of $V_t$. }.
Nevertheless, we can follow the approach outlined at the end of the above section, that is step~\eqref{itm:solve by constrained minimization} which aims to solve the constrained optimization problem~\eqref{eq:minim}. \\

\paragraph{Results.} For our simulations we set $\log(P_0)= \log(100)$, $V_0 = 5$, $\mu = 0$, $\lambda = 2$, $\rho = 0.2$, 
$\xi = 1$, $c = 2$, $k = 5$ and $t\in [0,1]$.
The set of experiments is the same as those presented in the previous Subsection. 
Figure~\ref{fig:transition_matrix} visualizes the iterative nature of the MC construction. 
It can be noticed that Figure~\ref{fig:transition_matrix} includes the presence of red points, which represent approximated points, i.e. the ones ``recombined'' numerically using the optimization problem~\ref{eq:minim}, whilst in Figure~\ref{fig:transition_matrix_toy} these are not present since the assumptions of Theorem~\ref{th:mm_d=2} are satisfied.  
Figure~\ref{fig:number_of_states_and_comparison_number_of_states} shows the growth of the state space as a function of time for different values of $n$, 
and Figure~\ref{fig:rel_errors} shows the relative error.  
Finally, Figure~\ref{fig:sparsity of the matrix} shows the resulting sparse transition matrices.
Note that the approximations error, Figure \ref{fig:rel_errors}, supports the intuition we have used in step~\eqref{itm:solve by constrained minimization}: although the assumptions of Theorem~\ref{th:mm_d=2} are not met, the model error resulting from using the MC is small. 
Again, it is not surprising  that the transition matrices are sparse,  Figure~\ref{fig:sparsity of the matrix}, since this is optimized by construction.

\clearpage
\bibliography{Biblio_tree}

\clearpage
\appendix

\section{Appendix - Lemmas}
\begin{lemma}\label{lem:eqv_loc_cons}
If $\EE | X_i^n |^q \le c(q, T, X_0)$, then Lemma~\ref{lemma: convergence assumptions}-Item~\eqref{it:my_local_consistency} implies Equation \eqref{eq:KP_consistency}.
\end{lemma}
\begin{proof}%
Let us recall~\cite{Kloeden1992} that if the coefficients
$\mu, \sigma\in C^4_b$, we have that for any $x\in\R^d$ and for some $q\ge 0$
  \begin{align}
\EE[X_{t_{i+1}}-X_{t_{i}}|X_{t_{i}}=x]-n^{-1}\mu(x)=&O\left((1+|x|^{q})n^{-2}\right),\\\EE[(X_{t_{i+1}}-X_{t_{i}})^{\otimes2}|X_{t_{i}}=x]-n^{-2}\mu(x)^{\otimes2}+n^{-1}\Sigma(x)=&O\left((1+|x|^{q})n^{-2}\right),
\end{align}
in particular the second inequalities implies that 
\[
\EE[(X_{t_{i+1}}-X_{t_{i}})^{\otimes2}|X_{t_{i}}=x]-n^{-1}\Sigma(x)=O\left((1+|x|^{q})n^{-2}\right)
\]
From Lemma~\ref{lemma: convergence assumptions}-Item~\eqref{it:my_local_consistency}, {dividing} by $n^{-1}$, we obtain that
\begin{align}
\EE\left[\frac{X_{i+1}^{n}-X_{i}^{n}}{n^{-1}}\Big|X_{i}^{n}\right]-\mu(X_{i}^{n})=&O\left((1+|X_{i}^{n}|^{q})n^{-\alpha}\right),\\
\EE\left[\frac{(X_{i+1}^{n}-X_{i}^{n})^{\otimes2}}{n^{-1}}\Big|X_{i}^{n}\right]-\Sigma(X_{i}^{n})=&O\left((1+|X_{i}^{n}|^{q})n^{-\alpha}\right).
\end{align}
Taking the square on both sides and then the expectation, using that $\EE | X_i^n |^q \le c(q, T, X_0)$, we obtain the thesis.
\end{proof}
 \begin{lemma}\label{lem:maximum_scheme}
 If $X^n$ represents the schemes built in Theorem~\ref{th:mm_d=1}, \ref{th:mm_d=2}, \ref{th:mm_d=d} and {$n\gamma_n^2<c$} for some $c>0$, then 
 \begin{align}
 \EE[\max_{i}|X_{i}^{n}|^{2q}]\leq & c(1+|X_{0}|^{2q}), \quad \text{ for } q\ge 1.
 \end{align}
 \end{lemma}
\begin{proof}
Let us first note that 
\[
 X^n_{i+1} - X^n_i | X_i^n=x \sim Y_x 
\] 
In Theorems~\ref{th:mm_d=1},~\ref{th:mm_d=2} and~\ref{th:mm_d=d} we have proved that if the SDE coefficients are bounded the r.v. $Y_x$, can be built such that $|Y_x|\leq c_{\mu,\sigma}n^{-1/2}+k\gamma_n$, whilst if the coefficients are linearly bounded then the r.v. $Y_x$, can be built s.t. $|Y_x|\leq c_{\mu,\sigma}(1+|x|)n^{-1/2}+k\gamma_n$, for some $k>0$. 
Since we are interested in the finite time horizon we suppose for simplicity $T=1$, which means $i\leq  n$.\\
Let us define $\mY_{i}^{n}=\sum_{j=0}^{i-1} Y_{X_{j}^{n}}-n^{-1}\mu(X^n_{j})$. 
By construction, to satisfy Item \autoref{lemma: convergence assumptions}-\eqref{it:my_local_consistency}, we build the r.v. 
$Y_{(\cdot)}$ s.t. $\mathbb{E}_{j}Y_{X_{j}^{n}}=n^{-1}\mu(X^n_{j})$, which {implies that $\mY_{i}^{n}$ is a martingale}. Thus, we can apply the Burkholder-Davis-Gundy (BDG) Inequality  
\begin{align}
\mathbb{E}\left[\max_{i\leq n}\left|\mY_{i}^{n}\right|^{2q}\right]\leq c_{q}\mathbb{E}\left\langle \mY_{n}^{n}\right\rangle ^{q}\leq&c_{q}\mathbb{E}\left[\sum_{j=0}^{n-1}\left|Y_{X_{j}^{n}}-n^{-1}\mu(X_{j}^{n})\right|^{2}\right]^{q}\\\leq&c_{q}n^{q-1}\sum_{j=0}^{n-1}\mathbb{E}\left|Y_{X_{j}^{n}}-n^{-1}\mu(X_{j}^{n})\right|^{2q}\\\leq&c_{q}n^{q-1}\left[\sum_{j=0}^{n-1}c_{\mu,\sigma}(1+\mathbb{E}|X_{j}^{n}|^{2q})n^{-q}+k\gamma_{n}^{2q}\right]\\\leq&c_{q}n^{q-1}\left[nk\gamma_{n}^{2q}+c_{\mu,\sigma}n^{-q}n+c_{\mu,\sigma}n^{-q}\sum_{j=0}^{n-1}\mathbb{E}|X_{j}^{n}|^{2q}\right]\\\leq&c_{q}(n\gamma_{n}^{2})^{q}+c_{q,\mu,\sigma}+c_{q,\mu,\sigma}n^{-1}\sum_{j=0}^{n-1}\mathbb{E}|X_{j}^{n}|^{2q}\\\leq&c_{\mu,\sigma,q}\left(1+n^{-1}\sum_{j=0}^{n-1}\mathbb{E}|X_{j}^{n}|^{2q}\right),
\end{align}
where at the $4$-th step  ($3$-rd line) we have used the fact that $|Y_{X_{j}^{n}}|\leq c_{\mu,\sigma}(1+|X_j^n|)n^{-1/2}+k\gamma_n$ and $\mathbb{E}Y_{X_{j}^{n}}=n^{-1}\mu(X^n_{j})$, so $n^{-1}\mu(X^n_{j})\in \text{Convex Hull}\{\operatorname{supp}(Y_{X_{j}^{n}})\}$; 
at the last step we have used the assumption that $n\gamma_n^2<c$, for some $c>0$. \\
We can now proceed to bound $\mathbb{E}\left[\max_{ i\leq n}|X_{i}^{n}|^{2q}\right]$ 
\begin{align}
\mathbb{E}\big[\max_{i\leq n}|&\left.X_{i}^{n}|^{2q}\right]=\mathbb{E}\left[\max_{i\leq n}\left|X_{0}\!+\!\sum_{j=0}^{i-1}Y_{X_{j}^{n}}\right|^{2q}\right]\\=&\mathbb{E}\left[\max_{i\leq n}\left|X_{0}\!+\!\sum_{j=0}^{i-1}n^{-1}\mu(X_{j}^{n})\!+\!\sum_{j=0}^{i-1}Y_{X_{j}^{n}}-n^{-1}\mu(X_{j}^{n})\right|^{2q}\right]\\\leq&c_{q}\left\{ \left|X_{0}\right|^{2q}\!+\!\mathbb{E}\left[\max_{i\leq n}\left|n^{-1}\sum_{j=0}^{i-1}\mu(X_{j}^{n})\right|^{2q}\right]\!+\!\mathbb{E}\left[\max_{i\leq n}\left|\sum_{j=0}^{i-1}Y_{X_{j}^{n}}-n^{-1}\mu(X_{j}^{n})\right|^{2q}\right]\!\right\} \\\leq&c_{q}\left\{ \left|X_{0}\right|^{2q}\!+\!n^{-2q}n^{2q-1}\sum_{j=0}^{n-1}\mathbb{E}\left|\mu(X_{j}^{n})\right|^{2q}\!+\!\mathbb{E}\left[\max_{i\leq n}\left|\mY_{i}^{n}\right|^{2q}\right]\right\} \\\leq&c_{q}\left\{ \left|X_{0}\right|^{2q}\!+\!n^{-1}\sum_{j=0}^{n-1}c_{q,\mu,\sigma}\left(1\!+\!\mathbb{E}|X_{j}^{n}|^{2q}\right)\!+\!c_{\mu,\sigma,q}\!+\!c_{\mu,\sigma,q}n^{q-1}n^{-q}\sum_{j=0}^{n-1}\mathbb{E}|X_{j}^{n}|^{2q}\right\} \\\leq&c_{q}\left\{ \left|X_{0}\right|^{2q}\!+\!c_{q,\mu,\sigma}n^{-1}\sum_{j=0}^{n-1}\mathbb{E}|X_{j}^{n}|^{2q}\!+\!2c_{q,\mu,\sigma}\!+\!c_{q,\mu,\sigma}n^{-1}\sum_{j=0}^{n-1}\mathbb{E}|X_{j}^{n}|^{2q}\right\} \\\leq&c_{q,\mu,\sigma}\left\{ \left|X_{0}\right|^{2q}\!+\!1\!+\!n^{-1}\sum_{j=0}^{n-1}\mathbb{E}|X_{j}^{n}|^{2q}\right\} .
\end{align}
It remains now to bound $\mathbb{E}|X_{i}^{n}|^{2q}$: 
applying the {discrete Grönwall’s Lemma}~\ref{lem:Gronw_dis}, from the previous equation we know that 
\begin{align}
\mathbb{E}|X_{i}^{n}|^{2q}\leq&\mathbb{E}\left[\max_{1\leq i\leq n}|X_{i}^{n}|^{2q}\right]\\\leq&c_{q,\mu,\sigma}\left\{ \left|X_{0}\right|^{2q}+1+n^{-1}\sum_{j=0}^{n-1}\mathbb{E}|X_{j}^{n}|^{2q}\right\} \\\leq&c_{q,\mu,\sigma}\left\{ \left|X_{0}\right|^{2q}+1\right\} +c_{q,\mu,\sigma}\left\{ \left|X_{0}\right|^{2q}+1\right\} \sum_{j=0}^{i-1}n^{-1}\exp\left\{ \sum_{j=0}^{n-1}n^{-1}\right\} \\\leq&c_{q,\mu,\sigma}\left\{ \left|X_{0}\right|^{2q}+1\right\} (1+e).
\end{align}
To conclude the proof we can notice that using the Jensen Inequality and the sub-additivity of the root we have that
\begin{align}
\mathbb{E}\left[\max_{i\leq n}|X_{i}^{n}|^{q}\right]\leq&\sqrt{\mathbb{E}\left[\max_{i\leq n}|X_{i}^{n}|^{2q}\right]}\\\leq&\sqrt{c_{q,\mu,\sigma}\left\{ \left|X_{0}\right|^{2q}+1\right\} (1+e)}\leq c_{q,\mu,\sigma}\left\{ \left|X_{0}\right|^{q}+1\right\} (1+e).
\end{align}
\end{proof}
We state here the discrete Grönwall’s Lemma that we have used in some previous proofs of this work. 
\begin{lemma}\label{lem:Gronw_dis} [Discrete Grönwall's {Lemma \cite{Holte2009}}]
Let $y_n, f_n$, and $g_n$ be non-negative sequences such that for $n>0$, 
$
y_{n}\leq f_{n}+\sum_{k=0}^{n}g_{k}y_{k},
$
then 
\[
y_{n}\leq f_{n}+\sum_{k=0}^{n}f_{k}g_{k}\exp\left(\sum_{j=k}^{n}g_{j}\right).
\]
\end{lemma} 
\section{Appendix - Details on Experiments}\label{app:details}
 In both the experiments we decided on the following strategy 
 \begin{enumerate}[(a)]%
 \item \label{itm:try rand algo} try to solve~\eqref{eq:fundamental_system} using Algorithm 2 in \cite{Cosentino2020}; 
 \item if after a specified maximal number of iterations, step \eqref{itm:try rand algo} has not returned a solution, solve the minimization problem~\eqref{eq:minim} and then we apply Algorithm 3 in \cite{Cosentino2020}, given the sequence of weights solving~\eqref{eq:minim} and so~\eqref{eq:fundamental_system}. 
 Remember that Algorithm 3 in \cite{Cosentino2020} is a combination of the Algorithms in \cite{maria2016a} and \cite{Cosentino2020}. 
 \end{enumerate}
 The results in \cite{Cosentino2020} can also help to reduce further the cardinality of the state space. 
 \cite[Theorem 3, item 3]{Cosentino2020}
 tells us that given a solution of~\eqref{eq:fundamental_system} we could possibly modify it changing the points (and therefore the weights) if the new points satisfy a specified condition. 
 The idea is thus to search if there are points $|\bar l_j|\leq c, \bar l_j\in\gamma\Z^2$ such that 
 $\bar l_j$ satisfies the condition explicited in \cite[Theorem 3, item 3]{Cosentino2020} and 
 $\bar l_j$ have been already used in the past - indeed let us recall that we build the tree recursively starting from the initial point. 
 Briefly, we indicate $\{l_j^\star\}, \{w_j^\star\}$ %
 a solution of~\eqref{eq:fundamental_system}, when the Caratheodory's reduction has already taken place, 
 and $C(\{\cdot\})$ indicates the cone built from the points $\{\cdot\}$. 
 \cite[Theorem 3, item 3]{Cosentino2020} claims that any point into the cone $C(\{-l_{j}^{\star}\}\setminus-l_{1}^{\star})$ can be exchanged with $l_{1}^{\star}$ to have a new solution to~\eqref{eq:fundamental_system}. 
 Namely, considering that we want to find a solution in $\gamma\Z^2$, we can look for any point $\bar l$ in 
 \[
 \{|l|<c,\,\,l\in\gamma\Z^{2}\}\cap\{l\in\gamma\Z^{2}:x+l\text{ has been used to build previous r.v.}\}\cap C(\{-l_{j}^{\star}\}\setminus-l_{1}^{\star}).
 \]
 The points $\bar l\cup \{l_{j}^{\star}\}\setminus l_{1}^{\star}$ constitute another solution to~\eqref{eq:fundamental_system}, changing accordingly the weights. 
 This can be generalized for any $l_{i}^{\star}$, not only $l_{1}^{\star}$ and can be iteratively done for all of them. 
 This procedure can help us to reduce further the total number of states necessary to approximate the considered SDE, being able to re-use points already used, and therefore increasing the ``level of recombination'' of the tree. As a drawback, this procedure leads to a higher computational cost.

\end{document}